\documentclass[11pt]{amsart}

\usepackage[T1]{fontenc}
\usepackage[utf8]{inputenc}
\usepackage{lmodern}
\usepackage{layout}
\usepackage[american]{babel}
\usepackage[babel, final]{microtype}
\usepackage{amssymb}
\usepackage{amscd}
\usepackage[mathscr]{eucal}

\usepackage{outlines}
\usepackage{enumitem}

\usepackage{float}

\usepackage[pdftex, dvipsnames]{xcolor}
\usepackage[pdftex, final]{graphicx}
\usepackage{caption}
\usepackage{subcaption}
\usepackage{pinlabel}
\usepackage[pdftex, letterpaper, includehead, includefoot, nomarginpar, 
margin=1.1in]{geometry}

\usepackage[pdftex, final, colorlinks=true, urlcolor=NavyBlue, 
linkcolor=NavyBlue, citecolor=NavyBlue, filecolor=NavyBlue, menucolor=NavyBlue, 
bookmarks=true, bookmarksdepth=3, bookmarksnumbered=true, bookmarksopen=true, 
bookmarksopenlevel=2]{hyperref}
\hypersetup{
  pdftitle={Equivalence of Contact Gluing Maps in Sutured Floer Homology},
  pdfauthor={Ryan Leigon, Federico Salmoiraghi},
  pdfsubject={Gluing Maps},
  pdfkeywords={Heegaard Floer homology, contact geometry, bordered Floer homology, gluing maps
    57M27, 57R58}
}
\usepackage{tikz}
\usetikzlibrary{arrows, automata}
\usetikzlibrary{cd}

\usepackage[final]{showlabels} %Use "inline" to show labels, "final" to hide

\makeatletter
\g@addto@macro\@floatboxreset\centering
\makeatother

\allowdisplaybreaks[4]

\addto\extrasamerican{%
}
\let\fullref\autoref
\newtheorem{maintheorem}{Theorem}

\newtheorem{theorem}{Theorem}[section]
\newtheorem{corollary}{Corollary}[section]
\newtheorem{proposition}{Proposition}[section]
\newtheorem{lemma}{Lemma}[section]

\theoremstyle{definition}
\newtheorem{definition}{Definition}[section]
\newtheorem{remark}{Remark}[section]

\makeatletter
\let\c@maincorollary=\c@maintheorem
\let\c@corollary=\c@theorem
\let\c@proposition=\c@theorem
\let\c@lemma=\c@theorem
\let\c@remark=\c@theorem
\let\c@definition=\c@theorem
\let\c@example=\c@theorem
\makeatother
\def\makeautorefname#1#2{\expandafter\def\csname#1autorefname\endcsname{#2}}
\makeautorefname{maintheorem}{Theorem}%
\makeautorefname{maincorollary}{Corollary}%
\makeautorefname{theorem}{Theorem}%
\makeautorefname{corollary}{Corollary}%
\makeautorefname{proposition}{Proposition}%
\makeautorefname{lemma}{Lemma}%
\makeautorefname{remark}{Remark}%
\makeautorefname{definition}{Definition}%
\makeautorefname{example}{Example}%
\makeautorefname{equation}{equation}%

\DeclareMathOperator{\SFH}{\mathrm{SFH}}
\DeclareMathOperator{\SFC}{\mathrm{SFC}}
\DeclareMathOperator{\BSD}{\widehat{\mathrm{BSD}}}
\DeclareMathOperator{\BSA}{\widehat{\mathrm{BSA}}}

\DeclareMathOperator{\BSAA}{\widehat{\mathrm{BSAA}}}
\DeclareMathOperator{\BSDD}{\widehat{\mathrm{BSDD}}}

\DeclareMathOperator{\EH}{\mathrm{EH}}

\DeclareMathOperator{\id}{\mathrm{Id}}
\DeclareMathOperator{\INT}{\mathrm{Int}}

\newcommand{\HKM}{\mathrm{HKM}}
\newcommand{\A}{\mathcal{A}}

\newcommand{\HD}{\mathcal{H}}

\newcommand{\W}{\mathcal{W}}
\newcommand{\D}{\mathcal{D}}

\newcommand{\M}{\mathcal{M}}
\newcommand{\SN}{\mathcal{N}}

\newcommand{\SF}{\mathcal{F}}
\newcommand{\SZ}{\mathcal{Z}}

\newcommand{\TW}{\mathcal{TW}}

\newcommand{\F}{\mathbb{F}}

\newcommand{\Z}{\mathbb{Z}}

\newcommand{\balpha}{\boldsymbol{\alpha}}
\newcommand{\bbeta}{\boldsymbol{\beta}}

\newcommand{\bx}{\mathbf{x}}
\newcommand{\by}{\mathbf{y}}
\newcommand{\bz}{\mathbf{z}}

\newcommand{\BZ}{\mathbf{Z}}

\begin{document}
%%%%%%%%%%%%%%%%%%%%%%%%%%%%%%%%%%%%%%%%%%%%%%%%%%%%%%%

%\thispagestyle{empty}
%
\title[Equivalence of Contact Gluing Maps in Sutured Floer Homology]{Equivalence of Contact Gluing Maps in Sutured Floer Homology}
\author{Ryan Leigon}
\address{Department of Mathematics \\ Louisiana State University \\ Baton 
  Rouge, LA 70803}
\email{\href{mailto:cleigo2@math.lsu.edu}{cleigo2@lsu.edu}}
\urladdr{\url{http://www.math.lsu.edu/~cleigo2/}}

\author{Federico Salmoiraghi}
\address{Department of Mathematics \\ The Israel Institute of Technology \\ Haifa, Israel}
\email{\href{mailto:salmoiraghi@campus.technion.ac.il}{salmoiraghi@campus.technion.ac.il}}
%\urladdr{\url{http://www.math.lsu.edu/~fsalmo1/}}

%\keywords{Heegaard Floer homology}
%\subjclass[2010]{57M27; 57R58}

%%%%%%%%%%%%%%%%%%%%%%%%%%%%%%%%%%%%%%%%%%%%%%%%%%%%%%%

%%%%%%%%%%%%%%%%%%%%%%%%%%%%%%%%%%%%%%%%%%%%%%%%%%%%%%%
\begin{abstract}
We show that the contact gluing map of Honda, Kazez, and Matic has a natural algebraic description. In particular, we establish a conjecture of Zarev, that his gluing map on sutured Floer homology is equivalent to the contact gluing map.
\end{abstract}
%%%%%%%%%%%%%%%%%%%%%%%%%%%%%%%%%%%%%%%%%%%%%%%%%%%%%%%

\maketitle

%%%%%%%%%%%%%%%%%%%%%%%%%%%%%%%%%%%%%%%%%%%%%%%%%%%%%%%

%\tableofcontents

%%%%%%%%%%%%%%%%%%%%%%%%%%%%%%%%%%%%%%%%%%%%%%%%%%%%%%%
\section{Introduction} % (fold)
\label{sec:intro}
%%%%%%%%%%%%%%%%%%%%%%%%%%%%%%%%%%%%%%%%%%%%%%%%%%%%%%%

The contact invariant in Heegaard Floer homology has profoundly advanced our understanding of contact geometry in dimension three. Its first iteration was introduced by Ozsv\'ath and Szab\'o \cite{OzSz3} for closed contact $3$-manifolds, while Honda, Kazez, and Mati\'c \cite{HKM2} extended it to contact $3$-manifolds $(M,\xi)$ with convex boundary. In the latter setting the contact invariant is denoted $\EH(\xi)$ and takes values in the sutured Floer homology $\SFH(-M,-\Gamma)$\footnote{Homology groups are over $\F_2$-coefficients through the entirety of this paper.} of the manifold $M$ with reversed orientation and dividing set $\Gamma$ induced by $\xi$. These and other versions of the contact invariant have proven to be extremely useful in different contexts. Ghiggini \cite{Ghi05,Ghi06,Ghi07} has used them to study the fillability of closed $3$-manifolds. Invariants commonly referred to as LOSS invariants for Legendrian and transverse knots were introduced by Lisca, Ozsv\'ath, Stipsicz, and Szab\'o \cite{LOSS01}; connections with the invariant defined by Honda, Kazez, and Mati\'c were first shown by Stipscz and Ve\'rtesi \cite{SV01}, and extended by Etnyre, Vela-Vick, and Zarev \cite{EVZ17}.

The importance of contact invariants makes understanding their behaviour under various topological operations desirable. In \cite{HKM3} Honda, Kazez, and Mati\'c constructed a map $\Phi_{\xi}$, commonly referred to as the \textit{HKM map}, associated to a proper inclusion of sutured manifolds $(M,\Gamma)\subset(M',\Gamma')$ and a compatible contact structure $\xi$ on the complement $M'\setminus \INT(M)$. When $M'\setminus \INT(M)$ has no connected components disjoint from $\partial M'$, the gluing map takes the form
\[
\Phi_{\xi}:\SFH(-M,-\Gamma)\to \SFH(-M',-\Gamma').
\]
This map preserves the contact invariant $\EH$ in the following way. If $\zeta$ is a contact structure on $(M,\Gamma)$ compatible with $\xi$, then
\[
\Phi_{\xi}(\EH(\zeta)) = \EH(\zeta\cup\xi).
\]
The $\HKM$ map is also functorial in the sense that it satisfies identity and composition laws; see Section 6 of \cite{HKM3}.

Because of these properties, HKM maps play an important role in both sutured Floer theory and contact geometry. Of note, Juh\'asz \cite{Juh16} employed them in his construction of cobordism maps on sutured Floer homology and Golla  \cite{Gol16} used HKM maps to relate the LOSS invariant for Legendrian knots in $S^3$ with the invariant defined by Honda, Kazez, and Mati\'c.
HKM maps can often track contact invariants under cut-and-paste operations. In this regard, Massot \cite{Mas01} used them to systematically produce isotopy classes of universally tight, torsion-free contact structures with vanishing Ozsv\'ath--Szab\'o contact invariant. Recently, Juh\'asz and Zemke \cite{JuZe20} gave an alternate description of the HKM map in terms of contact handles and used it to prove several results about cobordisms maps in sutured Floer homology.

Outside the realm of contact geometry, Zarev \cite{Zar11} defined a gluing operation $\cup_{\SF}$ for sutured manifolds $(M_1,\Gamma_1)$, $(M_2,\Gamma_2)$ with suitable sutured subsurfaces $\SF\subset\partial M_1$, $\overline{\SF}\subset\partial M_2$. He also defined an associated gluing map which takes the form of a pairing 
\[
\Psi_{\SF}:\SFH(M_1,\Gamma_1)\otimes\SFH(M_2,\Gamma_2)\to\SFH((M_1,\Gamma_1) \cup_{\SF} (M_2,\Gamma_2).
\]

If $M_1$ is equipped with a contact structure $\xi$ compatible with $\Gamma_1$, then we can define a map which is formally similar to the $\HKM$ map
\[
	\Psi_\xi:\SFH(-M_2,-\Gamma_2) \to \SFH((-M_1,-\Gamma_1) \cup_{-\SF} (-M_2,-\Gamma_2))
\]
by $\Psi_\xi(\by) = \Psi_{-\SF}(\EH(\xi)\otimes\by)$. We will refer to $\Psi_{\xi}$ as the \textit{bordered contact gluing map}, since Zarev used bordered sutured theory to define the map $\Psi_{\SF}$.

Due to its algebraic nature, it not clear a priori whether this map deserves the moniker of contact gluing map. Indeed, it is not apparent from the definition that $\Psi_{\xi}$ sends contact invariants to contact invariants. Despite this, Zarev claimed that the $\HKM$ and bordered contact gluing maps are equivalent. The main result of this paper affirms this conjecture; we give a rough formulation here and the precise statement in \fullref{thm:main duplicate}.

\begin{maintheorem}
\label{thm:main}

Suppose that $\Phi_{\xi}$ is the $\HKM$ map for  a proper inclusion $(M,\Gamma)\subset(M',\Gamma')$ of sutured manifolds with no isolated components and compatible contact structure $\xi$. There is a sutured contact manifold $(M'',\Gamma'',\xi'')$ and graded isomorphisms $f,g$ such that $(M',\Gamma') \cong  (M'',\Gamma'')\cup_{\SF}(M,\Gamma)$, the contact structures $\xi$ and $\xi''$ have equivalent contact handle decompositions, and the following diagram commutes.
\begin{center}
\begin{tikzcd}
\SFH(-M,-\Gamma) \arrow[r, "f" ] \arrow[d, "\Phi_{\xi}"]
&\SFH(-M,-\Gamma)\arrow[d, "\Psi_{\xi''}"]\\
\SFH(-M',-\Gamma') \arrow[r, "g" ]
& \SFH((-M'',-\Gamma'') \cup_{-\SF} (-M,-\Gamma))
\end{tikzcd}
\end{center}

\end{maintheorem}

The strategy for attacking this theorem is as follows. We first describe the $\HKM$ maps for contact handle attachments in terms of certain diagrammatic maps for contact handles. We then identify these diagrammatic maps with bordered contact gluing maps for contact handle attachments. Next, we prove that an arbitrary bordered contact gluing map is a composition of such handle attachment maps; along the way, we must show that $\Psi_{\xi}$ maps contact invariants to contact invariants and satisfies an appropriate identity law. Combining all these results with the functorial properties of the $\HKM$ map yields \fullref{thm:main}.

\begin{remark}
\label{rmk:isolated components}
Note that the hypothesis of \fullref{thm:main} requires that $M'\setminus int(M)$ have no components disjoint from $\partial M'$ (no ``isolated components''); this ensures that $\xi$ can be decomposed using only contact 1-and 2-handles.
\end{remark}

\begin{remark}
\label{rmk:derived category}
It is natural to ask whether the graded isomorphisms in \fullref{thm:main} can be chosen to be identity maps. As of the writing of this paper, the map $\Psi_{\SF}$ is only known to be well-defined up to maps induced by graded homotopy equivalence, since it is defined using bordered sutured theory. A stronger version of \fullref{thm:main} would require naturality of bordered sutured Floer homology.
\end{remark}

\fullref{thm:main} has several interesting consequences. First, we obtain a bordered version of the $\HKM$ map.

\begin{corollary}
\label{cor:bordered hkm}
Let $\M=(M,\gamma,\SF,\SZ)$ and $\M'=(M',\gamma',\SF,\SZ)$ be bordered sutured manifolds with $M\subset M'$ and $M'\setminus int(M)$ a sutured manifold equipped with a compatible contact structure $\xi$. Then there exists a map of type-D structures induced by $\xi$
\[
	\Phi_\xi: \BSD(-\M)\to\BSD(-\M'),
\]
which is natural with respect to gluing bordered sutured manifolds along $\SF$.
\end{corollary}

Similar statements can be formulated for type-A structures and the various bimodule structures involved in the bordered sutured theory.  Of note, Etnyre, Vela-Vick, and Zarev \cite{EVZ17} used this corollary as a fundamental tool to prove their alternate characterization of the minus version of knot Floer homology. As Mathews \cite{Mat01} points out, this corollary also has interesting implications for the relationship between sutured Floer homology and the contact category defined by Honda and Tian \cite{HoTi01}.

\fullref{thm:main} also provides an alternate proof of Juhasz and Zemke's description of contact gluing maps for contact handle attachments found in \cite{JuZe20}. Honda, Kazez, and Mati\'c's original description of their gluing map required Heegaard diagrams to satisfy a number of technical conditions, collectively referred to as contact-compatibility. Even for contact handle attachments, the use of contact-compatible Heegaard diagrams often renders the associated HKM gluing map uncomputable on the level of chain complexes. As a result, most prior applications relied only on formal properties of the HKM maps. However, there are a priori diagram-dependent maps $\sigma_i$ associated to contact handle attachments which are simple to compute at the level of chain complexes; see \fullref{sec:hkm handle attachments}. It is possible to describe a given $\HKM$ map in terms of diagrammatic maps; see \fullref{lem:padded handles decomposition}. It turns out that these diagrammatic maps are well-defined up to graded isomorphism. This means that an arbitrary diagrammatic map induces the $\HKM$ up to graded ismorphism.

\begin{corollary}
\label{cor:simple maps}
The diagrammatic maps $\sigma_i$ are well-defined up to graded isomorphism. In particular, any diagrammatic map induces the corresponding $\HKM$ map up to graded isomorphism.
\end{corollary}

Finally, we use \fullref{cor:simple maps} to provide an algorithm for computing the $\HKM$ map combinatorially following Plamenevskaya's work in \cite{Pla07}.

\begin{corollary}
\label{cor:combinatorial hkm}
Given a proper inclusion of sutured manifolds $(M,\Gamma)\subset(M',\Gamma')$ with no isolated components and compatible contact structure $\xi$ on $M'\setminus int(M)$, there are nice Heegaard diagrams $\HD,\HD'$ for $(-M,-\Gamma),(-M',-\Gamma')$ respectively and an inclusion of complexes $\SFC(\HD)\to\SFC(\HD')$  which can be used to compute the $\HKM$ map $\Phi_{\xi}:\SFH(-M,-\Gamma)\to\SFH(-M',-\Gamma')$ up to graded isomorphism.
\end{corollary}

\subsection*{Organization}
\label{sec:organization}

In \fullref{sec:hkm handle attachments}, we introduce the diagrammatic maps and relate them to the corresponding $\HKM$ maps; see \fullref{lem:padded handles decomposition}. In \fullref{sec:the join map}, we review the bordered sutured theory that goes into Zarev's gluing map $\Psi_{\SF}$ on sutured Floer homology. We define the associated contact gluing map $\Psi_{\xi}$ in \fullref{sec:the bordered contact gluing map} and prove that it satisfies the composition law \fullref{lem:zarev composition}. \fullref{sec:computing bordered} reviews more background on making concrete computations with the bordered gluing map and using the formula in \fullref{lem:elementary join} in particular. We relate the diagrammatic maps with the bordered contact gluing map in \fullref{sec:proof of the main theorem} by \fullref{lem:one handle maps} and \fullref{lem:two handle maps} and use them to prove two versions of \fullref{thm:main}. We discuss corollaries and applications in \fullref{sec:applications}.

\subsection*{Acknowledgment}
\label{sec:acknowledgment}
The authors would sincerely like to express their gratitude to their advisor Shea Vela-Vick for introducing them to the subject, for suggesting the present problem, and for his patient support throughout the development of the project. Without his guidance this work would not have been possible. They are also indebted to Mike Wong for his invaluable encouragement and suggestions for improving various versions of the manuscript. His wisdom has contributed to the quality of the exposition and the spirits of the authors.  Finally, the authors would like to thank  John Etnyre, Ko Honda, and Robert Lipshitz, for the illuminating conversations.

%%%%%%%%%%%%%%%%%%%%%%%%%%%%%%%%%%%%%%%%%%%%%%%%%%%%%%%
\section{HKM Handle Attachments} % (fold)
\label{sec:hkm handle attachments}
%%%%%%%%%%%%%%%%%%%%%%%%%%%%%%%%%%%%%%%%%%%%%%%%%%%%%%%

\begin{figure}
\includegraphics[width=1.0\textwidth]{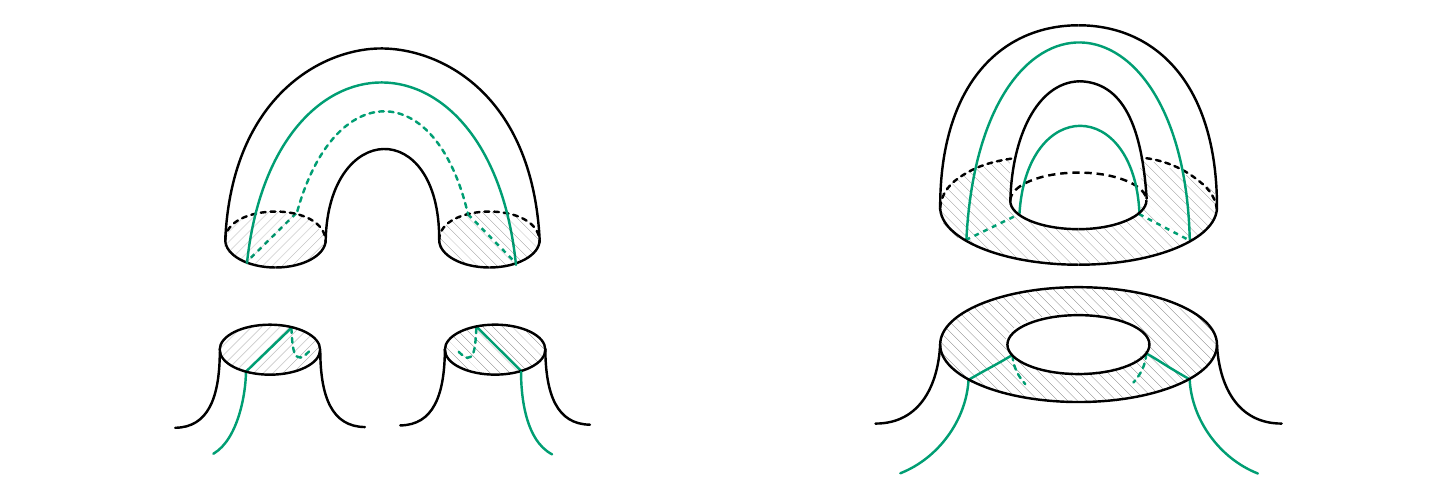}
  
  \caption{Contact $1$- and $2$-handle attachments}
	\label{fig:contact handles intro}
\end{figure}

The goal of this section is to describe the HKM maps associated to contact $i$-handle attachments in terms of the diagram-dependent maps $\sigma_i$. The ideas described here are familiar to experts, though we record them for the sake of clarity and self-containment.

\begin{remark}
Note that the discussion in this section parallels many ideas in \cite{JuZe20}. In particular, our padded contact handles are analogous their Morse-type contact handles and \fullref{lem:padded handles decomposition} is analogous to their Theorem 5.8.
\end{remark}

We assume familiarity with some basic ideas in contact geometry, particularly in convex surface theory. Throughout this paper, the contact structures we consider are defined up to isotopy. As such, we make implicit use of Giroux's uniqueness and flexibility theorems; see \cite{Gir01}. We also assume familiarity with Heegaard diagrams arising from partial open books, though we review some key features in \fullref{sec:hkm handles} to clarify notation. We refer to such diagrams as \textit{partial open book diagrams}; we will represent a partial open book diagram as in \cite{HKM3} by the quotient of a surface $S$ with embedded $\beta$-arcs and $\alpha$-arcs. The full Heegaard surface $\Sigma$ is formed by identifying pairs of intervals in $\partial S$, which we shade black, so that each arc becomes a closed curve after identification\footnote{While $\Sigma$ is usually defined by gluing a copy of a subsurface $Q\subset S$ to $S$ along $\partial Q\cap\partial S$, the condition that the $\alpha$-arcs in $S$ form a basis, guarantees that the diagram $(S\cup Q,\bbeta,\balpha)$ is diffeomorphic to $(S/\sim,\bbeta,\balpha)$, where $\sim$ is the relation which identifies pairs of intervals in $\partial S$.}. All our pictures represent diagrams $(\Sigma,\balpha,\bbeta)$ for $(M,-\Gamma)$ with the positive orientation of $\Sigma$ as depicted in \fullref{fig:Partial Open Book}. We reverse the roles of $\alpha$-curves and $\beta$-curves, and we think of these as diagrams $(\Sigma,\bbeta,\balpha)$ for $(-M,-\Gamma)$, as is standard when working with partial open books.

\subsection{Diagrammatic maps for handles and bypasses}
\label{sec:simple maps}

\begin{figure}[H]
\labellist
	\begin{footnotesize}
	\pinlabel $p$ at 120 50
	\pinlabel $q$ at 120 30
	\end{footnotesize}
	\begin{large}
	\pinlabel $\longrightarrow$ at 210 40
	\end{large}
\endlabellist
\includegraphics[width=1.0\textwidth]{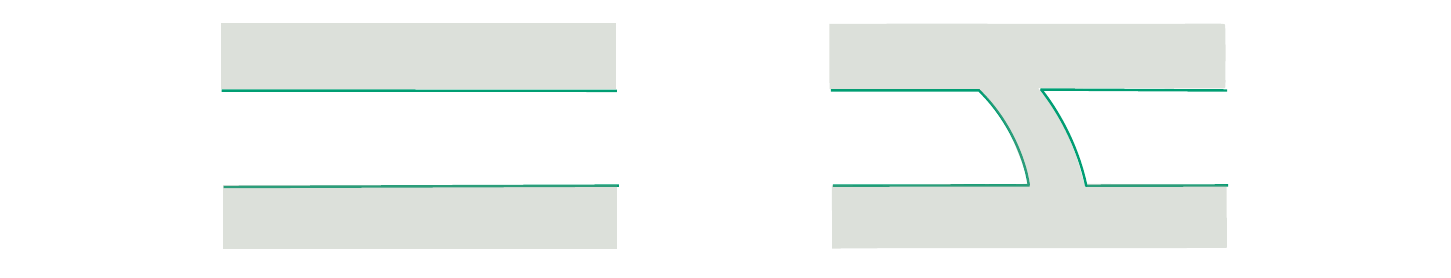}
  
  \caption{A diagrammatic contact 1-handle attachment.}
	\label{fig:one handle attachment}
\end{figure}

\begin{figure} [H]
\labellist
	\begin{footnotesize}
	\pinlabel $p$ at 120 58
	\pinlabel $q$ at 120 0
	\pinlabel $x_0$ at 300 58
	\pinlabel $x_0$ at 300 0
	\end{footnotesize}
	\begin{large}
	\pinlabel $\longrightarrow$ at 210 30
	\end{large}
\endlabellist
\includegraphics[width=1.0\textwidth]{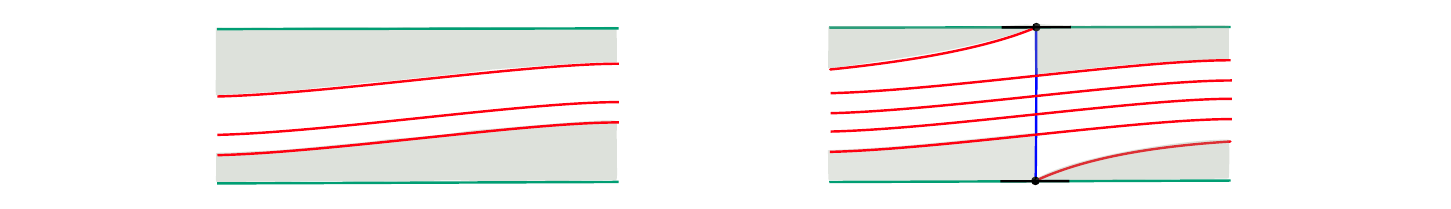}
  \caption{A diagrammatic contact 2-handle attachment.}
	\label{fig:two handle attachment}
\end{figure}

We need to define diagrammatic maps associated to contact handle attachments in order to relate the $\HKM$ and bordered contact gluing maps. We refer to the work of Giroux \cite{Gir01} and Ozbagci \cite{Ozb11} for the notion of contact handles in dimension three. Contact 1-and 2-handles are shown in \fullref{fig:contact handles intro} along with their dividing sets. Denote the result of attaching a contact $i$-handle $h^i$ to $(M,\Gamma,\xi)$ by $(M_i,\Gamma_i,\xi_i)$. Note too that we can attach a contact handle to a balanced sutured manifold $(M,\Gamma)$. Identify a neighborhood of $\partial M$ with $\partial M\times [-1,0]$ and let $\xi_{\Gamma}$ be the $[-1,0]$-invariant contact structure compatible with $\Gamma$. Then we can attach $h^i$ to $\partial M = \partial M\times \{0\}$, and the induced dividing set will make the result a balanced sutured manifold, denoted $(M_i,\Gamma_i)$.

Let $\HD = (\Sigma,\bbeta,\balpha)$ be a diagram for $(-M,-\Gamma)$. Recall that we recover $(-M,-\Gamma)$ from $\HD$ by attaching 2-handles along $(\bbeta\times\{0\})\sqcup(\balpha\times\{1\})$ to $(\Sigma\times[0,1],\partial\Sigma\times\{\frac12\})$. With this identification, a contact 1-handle attachment along $(\{p\}\times\{\frac12\})\sqcup (\{q\}\times\{\frac12\})\subset \Gamma$, corresponds to attaching a 2-dimensional 1-handle to $\Sigma$ along $ \{p,q\}$ to obtain a new surface $\Sigma_1$. Then $\HD_1 = (\Sigma_1,\bbeta,\balpha)$ is a diagram for $(-M_1,-\Gamma_1)$; see \fullref{fig:one handle attachment}. A contact $2$-handle attached along a curve $\delta$ with $\delta\cap\Gamma = \{p\}\times\{\frac12\}\cup \{q\}\times\{\frac12\}$ effects the following change at the level of diagrams. Attach a 2-dimensional 1-handle to $\Sigma$ along $ \{p,q\}$ to obtain a new surface $\Sigma_2$; call the core of the 1-handle $\lambda$. Define arcs $b = \delta\cap R_+(\Gamma)$ and $a= \delta\cap R_-(\Gamma)$, and let  $\beta_0 = b\cup\lambda$, $\alpha_0 = a\cup\lambda$. Perturb these curves so that $\beta_0$ and $\alpha_0$ intersect a single time positively in the 1-handle. Then $\HD_2 = (\Sigma_2,\bbeta\cup\beta_0,\balpha\cup\alpha_0)$ is a diagram for $(-M_2,-\Gamma_2)$; see \fullref{fig:two handle attachment}.

We refer to these operations on diagrams as $\textit{diagrammatic contact handle attachments}$. Furthermore, if $\HD$ is a partial open book diagram compatible with $\xi$, then the result of a diagrammatic handle attachment $\HD_i$ is a partial open book diagram compatible with $\xi_i = \xi\cup h^i$; see the proof of \fullref{prop:diagrammatic extension} or \cite{HKM2} for a more detailed discussion. Additionally, if $\bx$ is the canonical collection of intersection points representing $\EH(\xi)$, then $\bx$ represents $\EH(\xi_1)$ and $(\bx,x_0)$ represents $\EH(\xi_2)$. In the case of the 2-handle, call $x_0$ the \textit{preferred intersection point} of $\HD_2$; this parallels the terminology in \fullref{sec:hkm handles}.

In general, a diagrammatic handle attachment induces an obvious inclusion of the associated sutured Floer complexes.

\begin{definition}
\label{def:simple maps}
Let $\HD$ be a Heegaard diagram for $(-M,-\Gamma)$. Let $\HD_i$ be a diagram for $(-M_i,-\Gamma_i)$ obtained from $\HD$ by a diagrammatic contact $i$-handle attachment. For $i = 1$, the associated \textit{diagrammatic map} is the map $\sigma_1: \SFC(\HD) \to \SFC(\HD_1)$ given by $\sigma_1(\by) = (\by)$. For $i=2$, the associated \textit{diagrammatic map} is the map $\sigma_2 : \SFC(\HD) \to \SFC(\HD_2)$ given by $\sigma_2(\by) = (\by, x_0)$.
\end{definition}

Note that these maps depend on a choice of diagram, and it is not clear a priori whether or not they induce well-defined maps on $\SFH$. We will show in the course of this paper that the diagrammatic maps for $i$-handle attachments are well-defined up to graded homotopy equivalence and induce maps on homology which are equivalent to the $\HKM$ contact gluing maps for the corresponding handle attachments; see \fullref{cor:simple maps}. Note that Juh\'asz and Zemke have shown in \cite{JuZe20} that this is the case, though our proof is independent from theirs.

\begin{figure}[h]
\labellist
	\begin{footnotesize}
	\pinlabel $x_0$ at 295 75
	\pinlabel $x_0$ at 295 19
	\end{footnotesize}
\endlabellist
\includegraphics[width=1.0\textwidth]{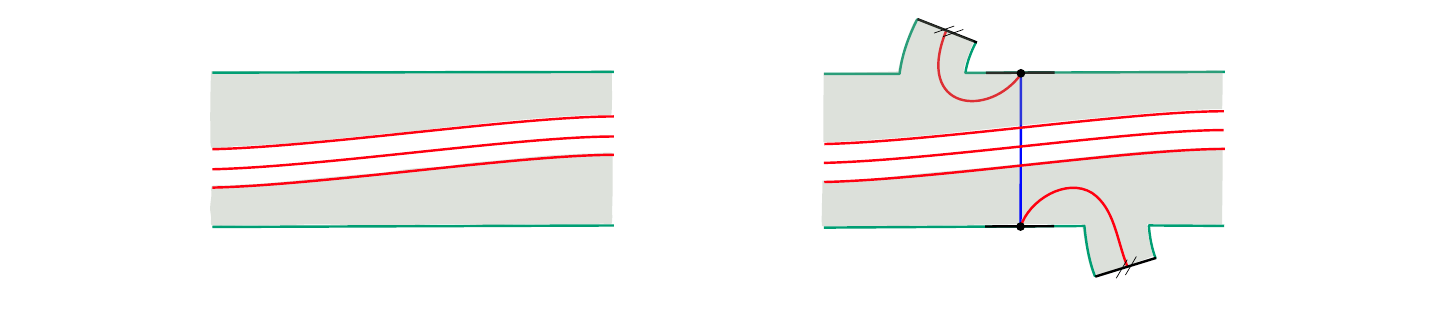}
  \caption{A diagrammatic positive trivial bypass.}
  \label{fig:positive bypass}
\end{figure}

\begin{figure}[h]
\labellist
	\begin{footnotesize}
	\pinlabel $x_0$ at 268 48
	\pinlabel $x_0$ at 318 48
	\end{footnotesize}
\endlabellist
\includegraphics[width=1.0\textwidth]{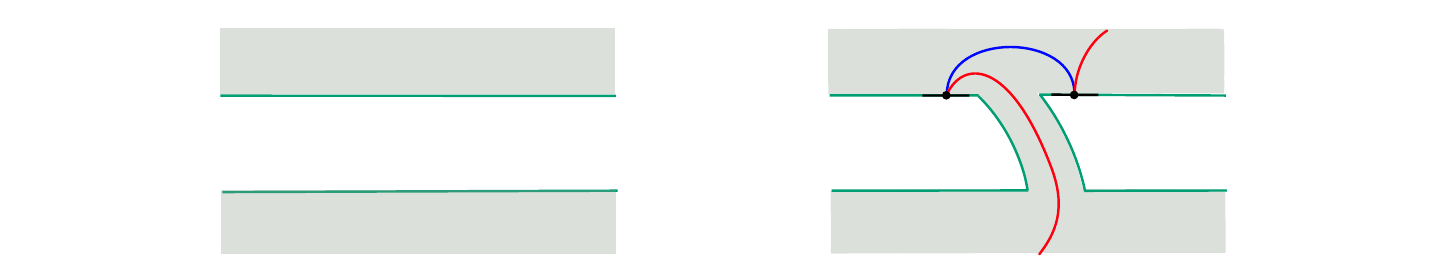}
  \caption{A diagrammatic negative trivial bypass.}
  \label{fig:negative bypass}
\end{figure}

In the course of identifying these diagrammatic maps with the $\HKM$ map, we will also need maps associated to trivial bypasses. Attaching a bypass can be viewed as attaching a contact 1-handle followed by a contact 2-handle; as such, trivial bypasses have associated diagrammatic maps.  A \textit{diagrammatic positive trivial bypass attachment} is depicted in \fullref{fig:positive bypass} and a \textit{diagrammatic negative trivial bypass attachment} is depicted in \fullref{fig:negative bypass}.

\begin{definition}
\label{def:simple bypass maps}
Let $\HD$ be a Heegaard diagram for $(-M,-\Gamma)$, and let $\HD_{12}$ be a diagram for $(-M,-\Gamma)$ obtained from $\HD$ by a diagrammatic trivial bypass attachment. The associated \textit{diagrammatic map} is the map $\tau:\SFC(\HD) \to \SFC(\HD_{12})$ given by $\tau(\by) = (\by,x_0)$.
\end{definition}
We will need the following standard computation; a stronger version can be found in \cite{HKM3}. 

\begin{lemma}
\label{lem:trivial bypass maps}
The map $\tau:\SFC(\HD) \to \SFC(\HD_{12})$ associated to a diagrammatic trivial bypass is a graded isomorphism of complexes.

\end{lemma}

\begin{proof}
For a positive bypass, $x_0$ is the only intersection point in the new $\alpha$-curve; see \fullref{fig:positive bypass}. For a negative bypass, $x_0$ is the only intersection point in the new $\beta$-curve; see \fullref{fig:negative bypass}. In either case, the map $\tau$ thus gives a one-to-one correspondence between generators of $\HD$ and generators of $\HD_{12}$. Furthermore, holomorphic disks in $\HD_{12}$ correspond exactly to holomorphic disks in $\HD$.
\end{proof}

\subsection{HKM maps for contact handle attachments.}
\label{sec:hkm handles}
 
Since $\HKM$ maps are only defined for proper inclusions $(M,\Gamma) \subset (M',\Gamma')$ of sutured manifolds, diagrammatic handle attachments do not obviously fit into the $\HKM$ framework. We will introduce proper inclusions associated to contact handle attachments which allow us to relate diagrammatic and $\HKM$ contact handle attachment maps.

\begin{definition}
\label{def:3-dimensional padded contact handle}
Given a sutured manifold $(M,\Gamma)$, let the \textit{padding} be the sutured contact manifold $(P=\partial M \times [0,1],-\Gamma \sqcup \Gamma,\xi_{\Gamma})$, where $-\Gamma\sqcup\Gamma = -\Gamma\times\{0\}\sqcup\Gamma\times\{1\} $ and $\xi_{\Gamma}$ is the $[0,1]$-invariant contact structure compatible with $-\Gamma \sqcup \Gamma$. A \textit{padded contact $i$-handle} $(P_i,-\Gamma\sqcup\Gamma_i,\xi_i)$ is the result of attaching a contact $i$-handle to $\partial M\times\{1\}$ in $(P,-\Gamma \sqcup \Gamma,\xi_{\Gamma})$.
\end{definition} 
We will often omit the contact structure $\xi_i$ from the notation and write $(P_i,-\Gamma\sqcup\Gamma_i)$ for a padded $i$-handle. The result of attaching a padded contact handle to $(M,\Gamma)$ is diffeomorphic to $(M_i,\Gamma_i)$. Moreover, the inclusion $M\subset M\cup P_i$ has an associated HKM map $\Phi_i:\SFH(-M,-\Gamma)\to\SFH(-M_i,-\Gamma_i)$. 

\begin{figure}[H]
\includegraphics[width=1.0\textwidth]{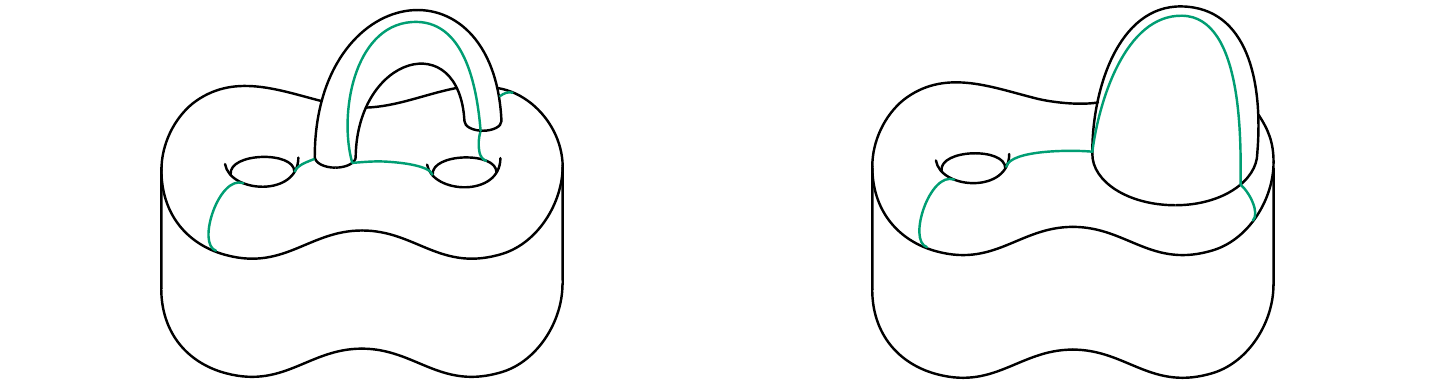}
\caption{Padded contact handles.}
\label{fig:Padded contact handles}
\end{figure}

Before identifying diagrammatic and $\HKM$ handle attachment maps, we will need to review the pertinent features of the definition of the $\HKM$ map. To give context and set notation, we first recall how partial open books are constructed following \cite{HKM2}; see also \cite{EtOz11}. Here and throughout this paper, $N(L)$ denotes a small neighborhood of a submanifold $L$. If $L$ is a Legendrian graph, $N(L)$ is a standard neighborhood unless otherwise indicated.

Given a sutured contact manifold $(M,\Gamma,\xi)$, a partial open book decomposition is determined by a graph $K\subset M$ satisfying the conditions of Theorem 1.1 of \cite{HKM2}. This guarantees that $M\setminus N(K)$ is product disk decomposable so that we can write $M\setminus N(K) = H_1 = (S\times [-1,1])/_\sim$ and $N(K) = H_2 = (Q\times [1,2])/_\sim$, where $Q\subset S$ and the relations $\sim$ collapse each interval of the form $\{s\}\times[-1,1], s\in \partial S$ and $\{q\}\times [1,2], q\in \partial Q\cap\partial S$ to a single point; the respective dividing sets are the images of $\partial S\times [-1,1]$ and $\partial Q\times[1,2]$ under $\sim$. This decomposition determines a monodromy map $h:Q\times\{2\}\to S\times\{-1\}$ with $h|_{Q\cap\partial S} = \id_{Q\cap\partial S}$ so that $(M,\Gamma)$ is diffeomorphic to $H_1\cup_{Q\times\{1\}} H_2/_{\sim h}$.

The Heegaard surface for an associated partial open book diagram is the image of $-S\times \{-1\}\cup Q\times \{1\}$ in the quotient. Curves are determined by a choice of arcs $\{a_1,\ldots,a_m\}$ in $Q$ which  form a \textit{basis} for $(S\setminus Q, R_+(\Gamma))$ in the sense that $(S\times\{1\})\setminus\cup a_k$ deformation retracts to $R_+(\Gamma)$. For each $k$, let $b_k$ be a perturbation of $a_k$ such that $\partial b_k$ is $\partial a_k$ moved in the $\partial S$ direction and $b_k\cap  a_k$ is a single intersection point. The sets of full curves $\balpha,\bbeta$ are given by the images of the $a_k\times\{1\}\cup a_k\times\{-1\}$ and $b_k\times\{1\}\cup h(b_k)\times\{-1\}$ respectively.

Observe that the image of $\partial \Sigma=((\partial Q\setminus \partial S) \times \{1\}) \cup (-(\partial S \setminus \partial Q) \times \{-1\})$ corresponds to $-\Gamma$ in the quotient, so that the diagram $(\Sigma,\balpha,\bbeta)$ is a balanced Heegaard diagram for $(M,-\Gamma)$. Then $(\Sigma,\bbeta,\balpha)$ is a diagram for $(-M,-\Gamma)$; we call any such diagram constructed this way a \textit{partial open book diagram}.

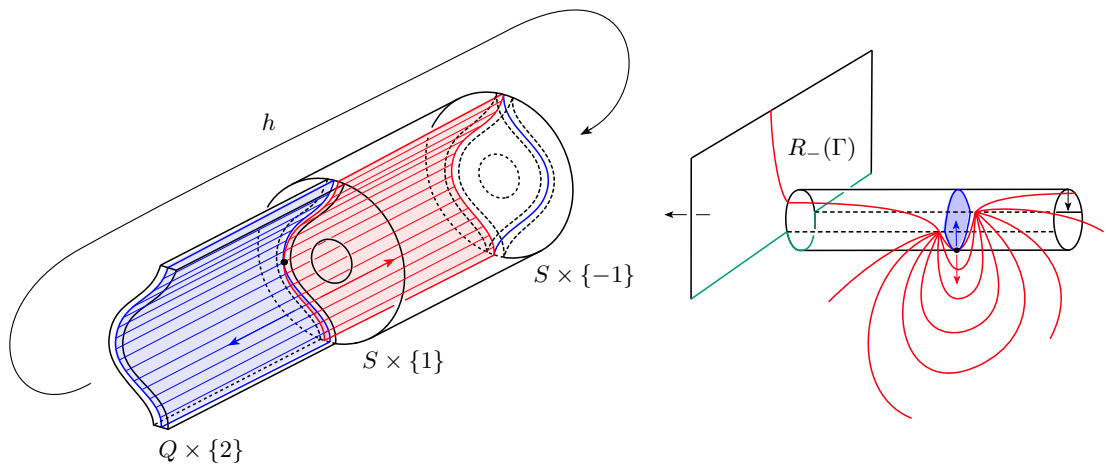
\begin{figure}[H]
\labellist
	\begin{footnotesize}
	\pinlabel $R_-(\Gamma)$ at 305 115
	
	\pinlabel $S\times \{-1\}$ at 217 68
	\pinlabel $S\times \{1\}$ at 150 36
	\pinlabel $Q\times \{2\}$ at 75 3
	\pinlabel $h$ at 100 125
	\end{footnotesize}
		
\endlabellist
\includegraphics[width=1.0\textwidth]{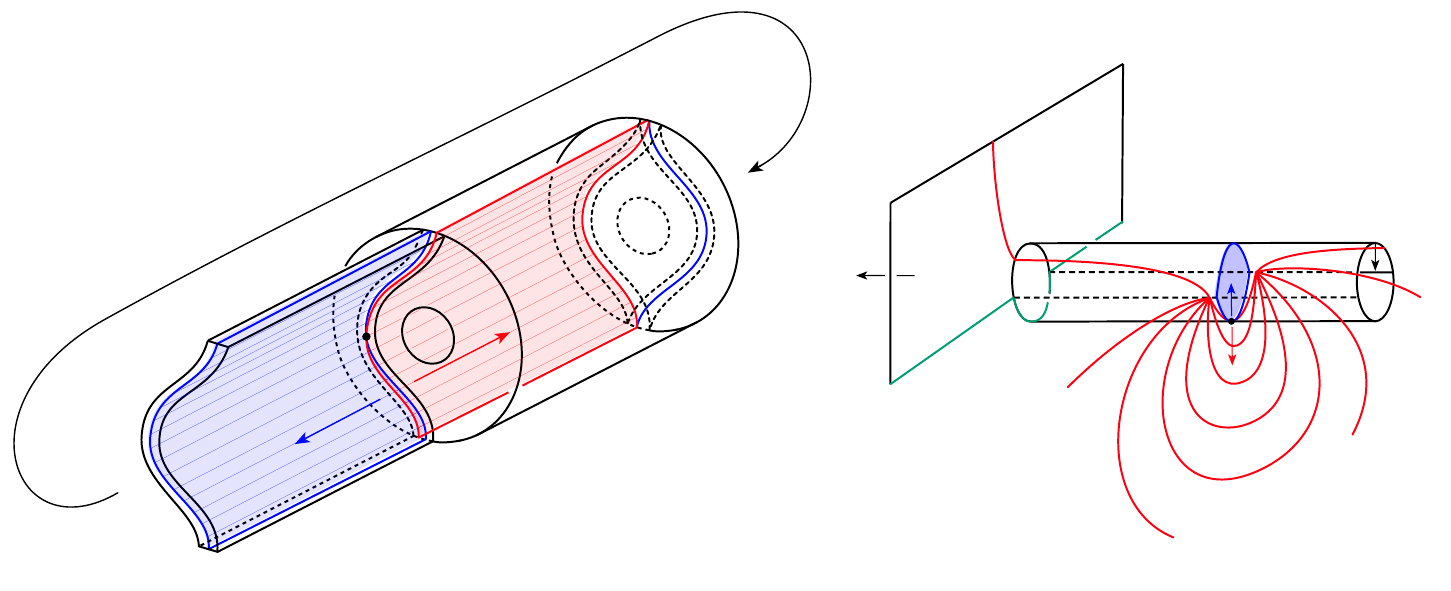}
\caption{On the left a partial open book $M=(S\times [-1,1]\cup Q\times [1,2])/_\sim$. On the right the Heegaard surface embedded in $(M,\Gamma)$; the black arrows represent the positive orientation of $\Sigma$.}
\label{fig:Partial Open Book}
\end{figure}

When defining the $\HKM$ map, one uses diagrams $\HD = (\Sigma,\bbeta,\balpha)$ for $(-M,-\Gamma)$ and $\HD' = (\Sigma'\supset\Sigma,\bbeta\cup\bbeta'',\balpha\cup\balpha'')$ for $(-M',-\Gamma')$ which satisfy a number of technical conditions collectively referred to as \textit{contact-compatibility}; in this context, $\HD$ is called a \textit{contact-compatible diagram} and $\HD'$ is called a \textit{contact-compatible extension} of $\HD$. The idea is that the these diagrams will mimic partial open book diagrams to a large degree, with some necessary adaptations. We will need to understand some of these conditions, which we now describe, following section 4 of \cite{HKM3}. Note that there are some notational differences; the roles of objects decorated by single primes and undecorated objects are reversed, and we use $Q$ instead of $P$.

Parametrize a neighborhood of $\partial M\subset M'$ by $\partial M\times [-1,1]$, where $\partial M = \partial M\times\{0\}$ and $V = \partial M\times [-1,0]$ is contained in $M$. Note that the sutures $\Gamma$ determine a $[-1,0]$-invariant contact structure $\zeta$ on $V$. The Heegaard splitting for a contact-compatible diagram of $(-M,-\Gamma)$ is determined by a choice of an embedded graph $K$ whose restriction to $V$ satisfies the conditions in Theorem 1.1 of \cite{HKM2} with respect to $\zeta$. (The conditions in Theorem 1.1 of \cite{HKM2} guarantee that a Legendrian graph determines a compatible partial open book decomposition.) In particular, $K|_V$ is Legendrian and the handlebodies $V\setminus N(K)$ and $N(K)|_V$ are each product disk decomposable with respect to the sutures induced by $\zeta$.

The essential choices made in constructing a contact-compatible extension of $\HD$ are as follows. Let $W = \partial M\times [0,1]$, and note that $\xi$ can be perturbed to restrict to the $[0,1]$-invariant contact structure on $W$. The extension is determined by a Legendrian graph $K_W\cup K''\subset M'\setminus \INT(M)$, with $K_W\subset W$ and $K''\subset M'' = M'\setminus \INT(M\cup W)$. (If one compares with \cite{HKM3}, a standard neighborhood of our $K_W$ is the complement of a standard neighborhood of their $K'''$.) Both $K_W$ and $K''$ are required to satisfy the conditions of Theorem 1.1 of \cite{HKM2} on $W$ and $M''$ with respect to $\xi$.
This splits $M'$ into handlebodies $H_1=M\setminus N(K \cup K_W \cup K'')$  and $H_2=N(K \cup K_W\cup K'')$; however, $H_1$ and $H_2$ need not be product disk decomposable, since $K$ is not necessarily Legendrian for some contact structure on $M\setminus V$. If we restrict our attention to $V \cup W \cup M''$, we have an honest contact manifold with contact structure $\zeta \cup \xi$. The restrictions $J_1$ and $J_2$ of the handlebodies $H_1$ and $H_2$ to $V\cup W\cup M''$ are product disk decomposable and can be written as $J_1=S' \times [-1,1]/_\sim$ and $J_2=Q' \times [1,2]/_\sim$, with $Q' \subset S'$ and $\sim$ the relation which collapses each interval $\{s'\}\times [-1,1]$, $s'\in\partial S'$ and $\{q'\}\times [-1,1]$, $q'\in\partial Q'$ to a point, just as for an ordinary partial open book decomposition.

Define the full Heegaard surface $\Sigma'$ to be the complement of $R_{+}(\Gamma')$ in $\partial H_1$. More explicitly, $\Sigma'$ is given by $\partial (M\setminus (N(K)\cup V)) \cup (S'\times \{-1\} \cup Q' \times \{1\})$. The curves in $\bbeta'$ and $\balpha'$ are defined to be the curves in $\bbeta$ and $\balpha$ together with additional curves $\bbeta'',\balpha''$ that are generated in the following way. Chose a collection of arcs $\{a_1'',...,a_m''\}$ on $Q'\times\{1\}$ which form a {\it basis} for $(S'\setminus Q,R_+(\Gamma'))$ in the sense that $(S'\setminus Q)\setminus\cup_{k=1}^m a_k''$ deformation retracts to  $R_+(\Gamma')$; here, $S'$ is identified with $S' \times \{1\}$, while $Q$ is the restriction of $Q' \times\{1\}$ to $V\subset M$.
For each $k$, let $b_k''$ be a perturbation of $a_k''$ such that $\partial b_k''$ is $\partial a_k''$ moved in the $\partial S'$--direction and $b_k''\cap a_k''$ is a single intersection point, as is standard procedure when constructing a partial open book diagram. Call the collection of these intersection points the \textit{preferred intersections} of the extension, and denote it by $\bx''$. The arcs $\{b_k''\},\{a_k''\}$ are completed to full sets of curves $\bbeta'' ,\balpha''$ in $\HD'$ using the monodromy map $h:Q'\times\{2\}\to S'\times \{-1\}$ obtained from the identification of $V\cup W\cup M''$ with  $J_1\cup_h J_2$. More explicitly, $\beta_i''$ and $\alpha_i''$ are the respective images of $b_i''\times\{1\}\cup h(b_i''\times\{2\})$ and $a_i''\times\{1\}\cup a_i''\times\{-1\}$ under the relation $\sim$.

\subsection{Identifying diagrammatic and HKM maps}
\label{sec:diagrammatic hkm}
In this section we identify the diagrammatic and $\HKM$ contact handle attachment maps. We will show that a diagrammatic handle attachment can be chosen to respect contact-compatibility and the preferred intersections.

\begin{figure}[H]
\labellist
	\begin{footnotesize}
	\pinlabel $R_-(\Gamma')$ at 153 96
	\pinlabel $K''$ at 235 60
	\pinlabel $K''$ at 20 60
	\pinlabel $R_-(\Gamma_1')$ at 400 102
	\end{footnotesize}
		
\endlabellist
\includegraphics[width=1.0\textwidth]{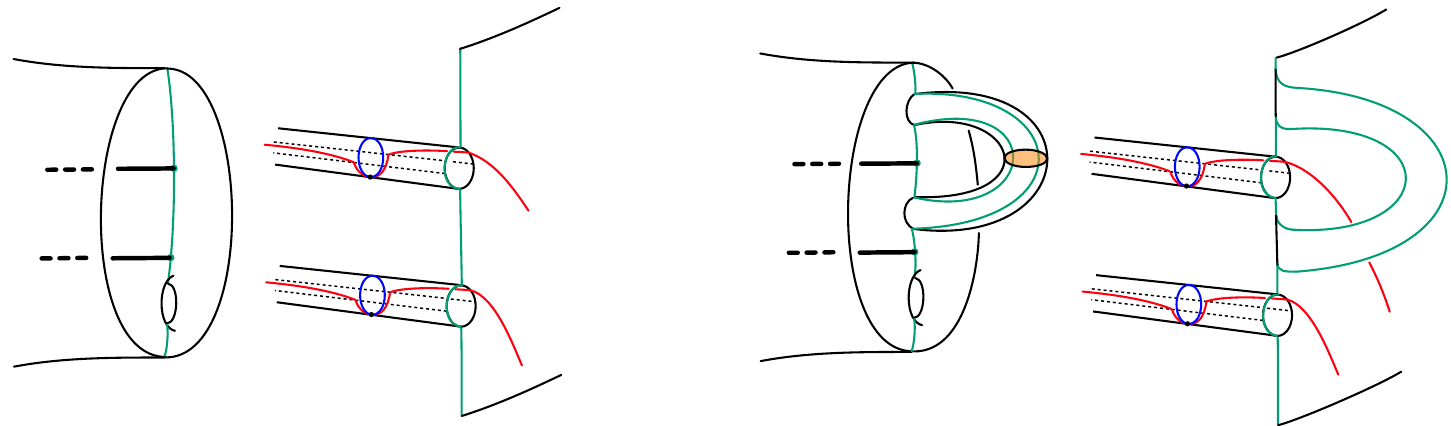}
\caption{Left: $(M',\Gamma')$ and $\HD'$; Right: $(M'_1,\Gamma'_1)$ and $\HD'_1$}
\label{fig:Contact compatible 1-handles}
\end{figure}

\begin{figure}[H]
\labellist
	\begin{footnotesize}
	\pinlabel $L$ at 100 100
	\pinlabel $K''$ at 173 68
	\pinlabel $K''$ at 25 68
	\pinlabel $L$ at 250 100
	\pinlabel $R_-(\Gamma_2')$ at 400 70
	\end{footnotesize}
		
\endlabellist
\includegraphics[width=1.0\textwidth]{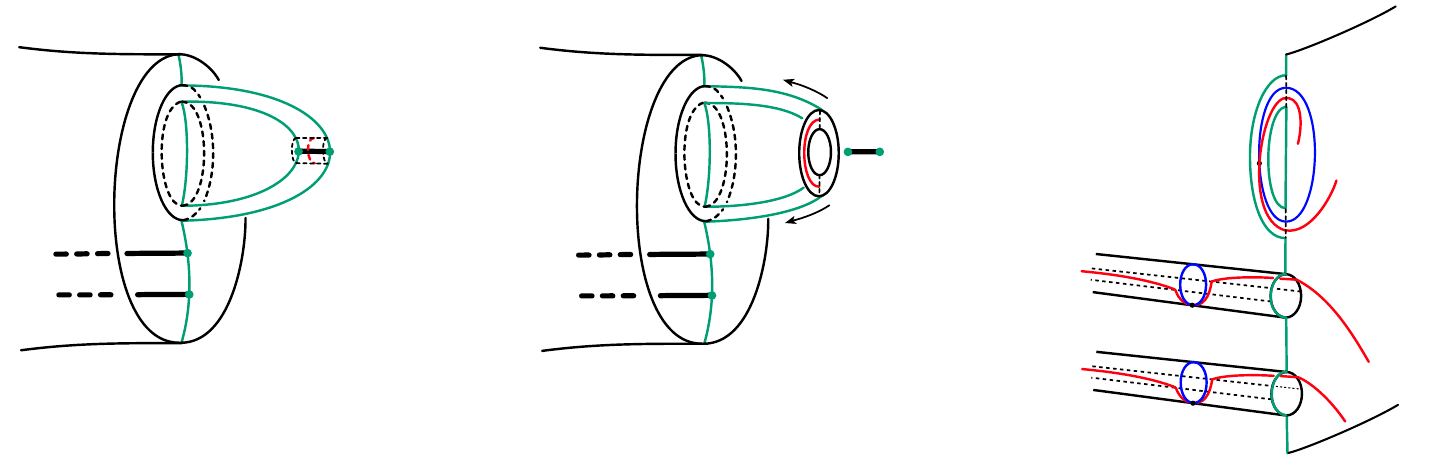}
\caption{Left: $(M'_2,\Gamma'_2)$; Center: $M'_2\setminus N(L)$; Right: $\HD'_2$ }
\label{fig:Contact compatible 2-handles}
\end{figure}

\begin{proposition}
\label{prop:diagrammatic extension}
Let $\HD = (\Sigma,\bbeta,\balpha)$ be a contact-compatible diagram for $(-M,-\Gamma)$ and $\HD' = (\Sigma',\bbeta\cup\bbeta'',\balpha\cup\balpha'')$ be a contact compatible extension of $\HD$ for $(-M',-\Gamma')$. Let $(M'_i,\Gamma'_i)$ be the result of attaching a contact $i$-handle to $(M',\Gamma')$. Then there is a contact-compatible extension $\HD'_i$ of $\HD$ for $(-M'_i,-\Gamma'_i)$ which is obtained from $\HD'$ by a diagrammatic handle attachment.
\end{proposition}

\begin{proof}
The graphs involved in the definition of a contact-capatible extension are all required to satisfy a number of conditions; however, $K''$ need only
\begin{enumerate}
\item satisfy the conditions of Theorem 1.1 of \cite{HKM2} with respect to $(M'',\xi|_{M''})$
\item have each component intersect $\Gamma'$ at least twice
\item have $\partial K''$ disjoint from $\partial K_W$.
\end{enumerate}
Attaching a contact handle to $\partial M'$ only alters $K''$ and the choice of basis arcs, so it suffices to verify that there are modifications of $K''$ and $\{a_k''\}$ which satisfy the properties describe above. The reader may verify that no other conditions in contact-compatibility are affected, and may be safely swept under the rug.

Given a 1-handle attachment $(M'_1,\Gamma'_1)$, the feet of the 1-handle can be isotoped to be disjoint from $\partial K''$, since $K''$ meets $\partial M'$ in a finite number of points by Theorem 1.1 of \cite{HKM2}. The graph $K''\subset M'\subset M'_1$ still satisfies the conditions of Theorem 1.1 of \cite{HKM2} with respect to $\xi_1 = \xi\cup h^1$, since one can cut $M''_1$ along the cocore of $h^1$ to recover $M''$, i.e. a system of compressing disks for $M_1''\setminus N(K'')$ can be obtained by adding the cocore of $h^1$ to a system of compressing disks for $M''\setminus N(K'')$. The new Heegaard surface is obtained from $\Sigma'$ by attaching a strip at the attaching sphere of $h^1$, since that is the corresponding change on $S'$; see \fullref{fig:Contact compatible 1-handles}. The basis $\{a_k''\}_{k=1}^m$ for $\HD'$ is still a basis for the new extension. Since the monodromy is unaffected, the curves also remain the same. This new extension is thus exactly a diagrammatic 1-handle attachment $\HD'_1$. 

Given a 2-handle attachment $(M'_2,\Gamma'_2)$ with attaching curve $\delta$, the attaching region can be isotoped to be disjoint from $K''$, since $K''$ meets $\partial M'$ in a finite number of points by Theorem 1.1 of \cite{HKM2}. Let $L$ be a Legendrian approximation of the cocore of $h^2$ with $\partial L\subset \Gamma'_2$. The graph $K''\cup L$ satisfies the conditions of Theorem 1.1 of \cite{HKM2} with respect to the contact structure $\xi_2$ obtained by attaching the contact 2-handle $h^2$ to $\xi$, since $M''_2\setminus N(K''\cup L)$ is contactomorphic to $M''\setminus N(K'')$. Note that $K''\cup L$ also satisfies the other two conditions listed above. Let $S'_2$ play the role of $S'$ for the new extension. While $S'_2$ is diffeomorphic to $S'$, there is an additional strip in $\Sigma'_2$ due to the new component $\partial N(L)\setminus\partial h^2$ of $Q'_2$ coming from $L$; see \fullref{fig:Contact compatible 2-handles}. This strip does not deformation retract to a subset of $R_+(\Gamma'_2)$, so the arcs $\{a''_k\}_{k=1}^m$ no longer form a basis.
Let $\mu$ be a meridian of $\partial N(L)\setminus\partial h^2$, and define $a_0''$ to be the arc $R_+(\partial N(L)\setminus\partial h^2)\cap \mu$. The collection $\{a_k''\}_{k=0}^m$ is now a basis for the new extension. 

Since the arc $b_0''$ is completed to a full $\beta$-curve using the core of the new strip attached along $\delta\cap\Gamma'$, this new extension is exactly a diagrammatic 2-handle attachment $\HD'_2$.
\end{proof}

\begin{proposition}
\label{prop:diagrammatic preferred intersection}
Let $\HD$, $\HD'$, and $\HD'_i$ be as above, and $\phi_{\xi}:\SFC(\HD)\to\SFC(\HD')$ be the map $\by\to(\by,\bx'')$ which induces the $\HKM$ map as described above. Then the map $\sigma_i\circ\phi_{\xi}:\SFC(\HD)\to\SFC(\HD'_i)$ induces the $\HKM$ map $\Phi_{\xi_i}:\SFH(-M,-\Gamma)\to\SFH(-M'_i,-\Gamma'_i)$.
\end{proposition}

\begin{proof}
For $i=1$, the preferred intersections $\bx''$ of $\HD'$ are also the preferred intersections of $\HD'_1$. The $\HKM$ map $\Phi_{\xi_1}$ is thus induced by the chain map $\by\to (\by,\bx'')$, which factors as $\sigma_1\circ\phi_{\xi}$.

For $i=2$, the preferred intersections of $\HD'_2$ are $\bx''\cup x_0$, where $x_0 = b_0''\cap a_0''$, the preferred intersection of the diagrammatic 2-handle. The $\HKM$ map $\Phi_{\xi_2}$ is thus induced by the chain map $\by\to (\by,\bx'',x_0)$, which factors as $\sigma_2\circ\phi_{\xi}$.
\end{proof}

We are now ready to relate diagrammatic and $\HKM$ maps.

\begin{lemma}
\label{lem:padded handles decomposition}
Let $(M,\Gamma)$ be a balanced sutured manifold. There is a contact-compatible diagram $\HD$ for $(-M,-\Gamma)$, a contact-compatible extension $\HD'$ for $(-(M\cup P),-\Gamma)$, a diagrammatic handle attachment $\HD'_i$ for $(-(M\cup P_i),-\Gamma_i)$, and a graded homotopy equivalence $\phi_{\xi_{\Gamma}}:\SFC(\HD)\to\SFC(\HD')$ such that the induced map on homology

\[
(\sigma_i\circ\phi_{\xi_{\Gamma}})_*:\SFH(-M,-\Gamma)\to\SFH(-(M\cup P_i),-\Gamma_i)\approx\SFH(-M_i,-\Gamma_i)
\]
is the $\HKM$ map $\Phi_i$.
\end{lemma}

\begin{proof} 
From the proof of Theorem 6.1 in \cite{HKM3}, there is a contact compatible extension $\HD'$ of a contact-compatible diagram $\HD$ such that the map $\phi_{\xi_{\Gamma}}:\SFC(\HD)\to\SFC(\HD')$ which induces $\Phi_{\xi_{\Gamma}}$ is a graded homotopy equivalence induced by a composition of diagrammatic trivial bypass maps. The lemma then follows immediately from \fullref{prop:diagrammatic preferred intersection}.
\end{proof}

%%%%%%%%%%%%%%%%%%%%%%%%%%%%%%%%%%%%%%%%%%%%%%%%%%%%%%%
\section{Background on Zarev's Gluing Map} % (fold)
\label{sec:the join map}
%%%%%%%%%%%%%%%%%%%%%%%%%%%%%%%%%%%%%%%%%%%%%%%%%%%%%%%

We now turn our attention to Zarev's bordered gluing map $\Psi_{\SF}$. In this section we review the requisite bordered sutured Floer theory for understanding the structure of this gluing map, though the interested reader should look in \cite{Zar11} for a full treatment of the topological and algebraic objects involved. Readers familiar with bordered sutured theory and Zarev's gluing map should skip to \fullref{sec:the bordered contact gluing map}, where we define the associated contact gluing map $\Psi_{\xi}$ and prove that it satisfies the composition law necessary for the proof of \fullref{thm:main}.

We deviate from Zarev's notational conventions in the following ways. Throughout, we reserve script letters for partially sutured manifolds, bordered sutured manifolds, and sutured surfaces. We avoid using them for honest sutured manifolds, except when the margins of the page force us to compress notation. We distinguish between suture sets on sutured manifolds and dividing sets on sutured surfaces by using a lower case $\gamma$ instead of $\Gamma$ for the dividing set on a sutured surface. As mentioned in  \fullref{sec:basic objects}, we denote the mirror of a partially sutured manifold $\W$ by $\overline{\W}$ instead of $-\W$. If $\HD$ is a diagram for a sutured manifold $(M,\Gamma)$, then Zarev treats $\SFC(M,\Gamma)$ as the $\A_\infty$--homotopy type of the complex $\SFC(\HD)$; morphisms between sutured Floer complexes are then defined up to $\A_\infty$--homotopy equivalence. We will instead work with representatives of these invariants, treating sutured Floer complexes as honest graded $\F_2$-vector spaces, and graded chain maps as morphisms. Similarly, we treat the bordered invariants $\BSA$, $\BSD$, and $\BSAA$ as honest Type-A, Type-D, and Type-AA structures, instead of their  $\A_\infty$--homotopy classes.

\subsection{Some basic objects}
\label{sec:basic objects}

We first review some topological objects which play an important role in bordered sutured theory. A \textit{sutured surface} is an oriented surface $F$ with boundary and a number of signed points $\Lambda\subset\partial F$ which alternate in sign as one traverses each component of $\partial F$; each component of $\partial F$ is required to intersect $\Lambda$. A sutured surface $\SF = (F,\Lambda)$ has associated sutured surfaces $-\SF = (-F,-\Lambda)$ and $\overline{\SF} = (-F,\Lambda)$. A \textit{dividing set} on a sutured surface is a collection of embedded, oriented arcs and closed curves $\gamma\subset F$ such that $\partial \gamma = -\Lambda$ and such that each component of $F\setminus\gamma$ can be oriented so that the induced orientation on the boundary agrees with the orientation of $\gamma$. Components where this orientation agrees with the orientation on $F$ constitute the $R_+(\gamma)$ region, while the other components constitute the $R_-(\gamma)$ region.  Unless otherwise indicated, we assume that all sutured surfaces are \textit{non-degenerate}, i.e. they have no closed components.

A \textit{partially sutured manifold} $\W = (W,\gamma,\SF)$ is a manifold $M$ whose boundary is the union of two sutured surfaces with common boundary, one of which (usually referred to as the ``sutured part of the boundary'') has a dividing set $\gamma$ and the other being $\SF$. If $\W_1 = (W_1,\gamma_1,\SF_1\sqcup \SF_2)$ and $\W_2 = (W_2,\gamma_2,-\SF_2\sqcup\SF_3)$ are partially sutured manifolds, then we can \textit{glue} or \textit{concatenate} them along $\SF_2$ to obtain
\[
\W_1\cup_{\SF_2}\W_2 = (W_1\cup_F W_2,\gamma_1\cup_{\Lambda_2}\gamma_2,\SF_1\sqcup\SF_3).
\]

If $\W = (W,\gamma, \SF)$ is a partially sutured manifold, we denote $(-W,-\gamma,-\SF)$ by $-\W$ and $(-W,\gamma,\overline{\SF})$ by $\overline{\W}$. Note that this differs from Zarev's notation, where there is no notation for the former and the latter is denoted $-\W$. In the special case of an honest sutured manifold, i.e. $\SF = (\emptyset,\emptyset)$, we have $-(M,\Gamma) = (-M,-\Gamma)$, and $\overline{(M,\Gamma)} = (-M,\Gamma)$.

 If $(M,\Gamma)$ is a sutured manifold, and $\W = (W,\gamma,\SF)$ is a partially sutured manifold with $W\subset M$, $\INT(F)\subset M$, $\partial W\setminus \INT(F)\subset \partial M$, and $\gamma\subset \Gamma$, we denote the partially sutured manifold $(\overline{M\setminus W}, \Gamma\setminus \INT(\gamma), -\SF)$ by $(M,\Gamma)\setminus \W$.

A \textit{bordered sutured manifold} $(M,\gamma,\SF,\SZ)$ is a partially sutured manifold whose non-sutured part $\SF$ of the boundary is parametrized by an arc diagram $\SZ$. We discuss arc diagrams in \fullref{sec:arc diagrams}

 Let $\SF = (F,\Lambda)$ be a sutured surface and $\gamma$ a dividing set on $\SF$. The \textit{cap} associated to $\gamma$ is the partially sutured manifold $(F\times [0,1],(\Lambda\times [0,1])\cup (\gamma\times\{1\}),-\SF)$. The caps associated to contact 1-and 2-handle attachments are shown in \fullref{fig:unbordered caps}.

The \textit{positive twisting slice} $\TW_{\SF,+}$ for $\SF$ is the partially sutured manifold $(F\times [0,1],\gamma,-\SF\cup-\overline{\SF})$, where $\gamma$ is obtained by giving $\Lambda\times [0,1]$ a minimal fractional Dehn twist in the $\partial F$-direction. The positive twisting slices for contact 1-and 2-handles are shown in \fullref{fig:twisting slices}.  The \textit{negative twisting slice} $\TW_{\SF,-}$ for $\SF$ is defined the same way, except that one twists in the $-\partial F$-direction. Note that $-\TW_{\SF,+} \cong \TW_{-\SF,-}$.

Given a partially sutured manifold $\W$ with non-sutured boundary $\SF$, the \textit{double} $\D(\W)$ of $\W$ is the sutured manifold $\overline{\W}\cup_{\overline{\SF}}\TW_{\overline{\SF},-}\cup_{-\SF}\W$. Suppose $\W$ is the cap for a dividing set $\gamma$ on $\SF$, and consider the the $[0,1]$-invariant contact structure $\xi_{\gamma}$ on $F\times [0,1]$ associated to $\gamma$. Then $\xi_{\gamma}$ induces a dividing set $\Gamma_{\gamma}$ on $F\times [0,1]$ such that $(F\times [0,1],\Gamma_{\gamma})$ is diffeomorphic to $\D(\W)$ as sutured manifolds. Note that $\Gamma_{\gamma}$ is a dividing set in the sense of convex surface theory, not in the sense of sutured surfaces discussed above.

\subsection{Gluing along sutured surfaces}
\label{sec:bordered gluing}
We now review the operation of gluing sutured manifolds in the sense of \cite{Zar11} and its connection with contact geometry.

\begin{figure}
\includegraphics[width=1.0\textwidth]{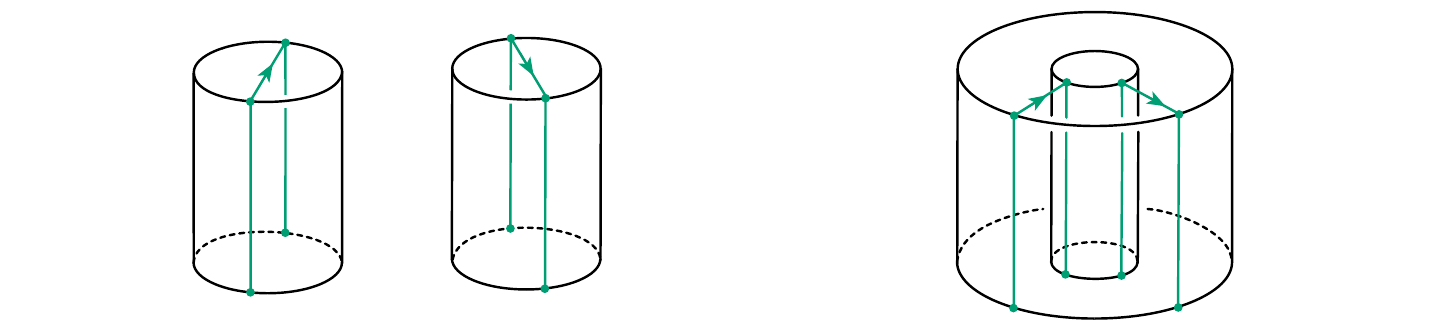}
  \caption{Caps for contact 1-and 2-handles. The dividing set $\gamma$ is shown on the top and sides, while the bottom $-\SF$ is undecorated.}
\label{fig:unbordered caps}
\end{figure}

Let $\SF = (F,\Lambda)$ be a sutured surface with dividing set $\gamma$. Let $(M_1,\Gamma_1)$ and $(M_2,\Gamma_2)$ be sutured manifolds with embeddings $(F,\gamma)\subset(\partial M_1,\Gamma_1)$ and $(-F,\gamma) \subset (\partial M_2,\Gamma_2)$. The embedding $F\to\partial M_1$ extends to an embedding $\W\to M_1$, where $\W $ is the cap for the dividing set $\gamma$ on $\SF$. Similarly, $-F\to\partial M_2$ extends to an embedding $\overline{\W}\to M_2$. Zarev defines the result of gluing $(M_1,\Gamma_1)$ to $(M_2,\Gamma_2)$ along $(\SF,\gamma)$ to be the manifold
\[
(M_1,\Gamma_1)\cup_{\SF}(M_2,\Gamma_2) := ((M_1,\Gamma_1)\setminus \W)\cup_{\SF} \TW_{\SF,+}\cup_{-\overline{\SF}} ((M_2,\Gamma_2)\setminus \overline{\W}),
\]
where $\TW_{\SF,+}$ is the positive twisting slice associated to $\SF$. We refer to the gluing operation above as a \textit{bordered gluing} to distinguish it from a proper inclusion of sutured manifolds; a bordered gluing operation is shown in \fullref{fig:gluing manifolds}. Note that the manifold $(M_1,\Gamma_1)\cup_{\SF}(M_2,\Gamma_2)$ is independent of the dividing set $\gamma$; however, the bordered gluing map defined below will depend on it.

\begin{figure}
\includegraphics[width=1.0\textwidth]{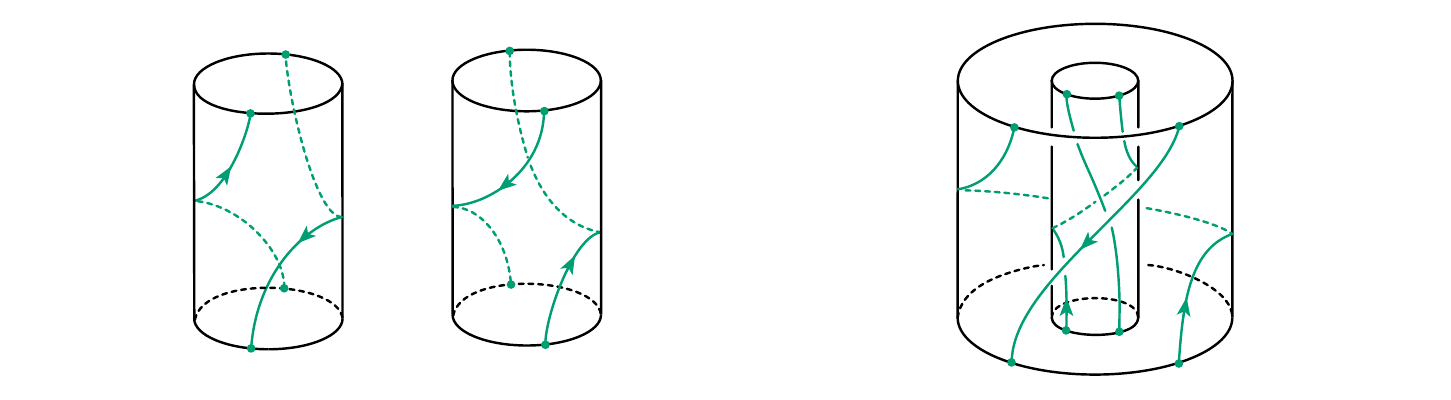}
  \caption{The positive twisting slices for contact 1-and 2-handles.}
\label{fig:twisting slices}
\end{figure}

\begin{figure}
\includegraphics[width=1.0\textwidth]{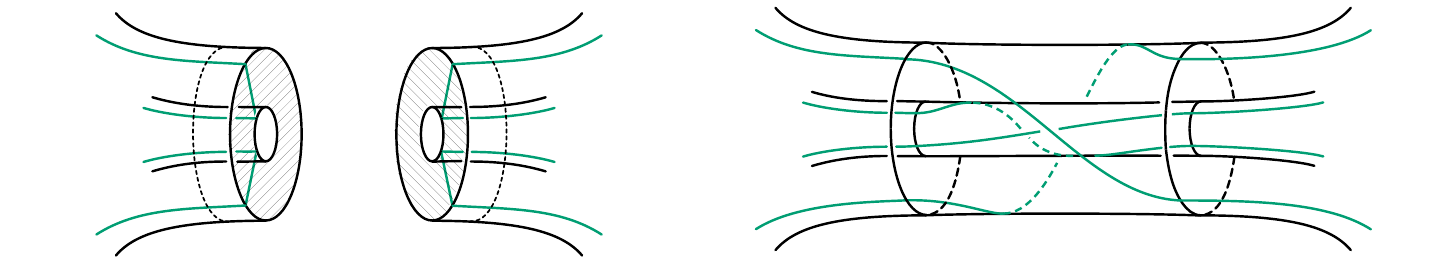}
  \caption{A bordered gluing operation along an annulus. The caps $\W$ and $\overline{\W}$ are a neighborhood of the shaded region.}
\label{fig:gluing manifolds}
\end{figure}

Bordered gluing operations describe how to glue the sutured manifolds underlying contact structures with convex boundary; however, there is a necessary orientation reversal. Let $(M,\Gamma,\xi)$ be a connected sutured contact manifold and let $\SF = (F,\Lambda)$ be a sutured surface with $(F,\partial F,\Lambda)\subset (M,\partial M,\Gamma)$; $(\INT(F)\subset \INT(M)$; $F$ convex with respect to $\xi$; and $\partial F$ Legendrian. If $M\setminus F$ has two connected components, then its closure  may be decomposed as sutured contact manifolds $(M_1,\Gamma_1,\xi_1)\sqcup (M_2,\Gamma_2,\xi_2)$, where $\xi_i = \xi|_{M_i}$ and $\Gamma_i$ is the dividing set induced by $\xi_i$ after smoothing corners. Suppose the normal vector for $F$ in $M$ points out of $M_1$ and into $M_2$. If $\W=(W,\Gamma_1|_F,-\SF)$ is the cap for $\SF$, then we can write
\[
(M,\Gamma) \cong ((M_1,\Gamma_1)\setminus\W)\cup_{\SF}\TW_{\SF,-}\cup_{-\overline{\SF}} ((M_2,\Gamma_2)\setminus\overline{\W}).
\]
We will use the notation $\xi = \xi_1\cup_{\SF}\xi_2$ to describe the relationship among these contact structures.

Note that the twist of dividing sets which arises from smoothing corners is negative, instead of the positive twist used in the bordered gluing operation. This is consistent with the orientation reversal associated to the contact invariant, since we have
\[
(-M_1,-\Gamma_1)\cup_{-\SF}(-M_2,-\Gamma_2)= ((-M_1,-\Gamma_1)\setminus -\W)\cup_{-\SF} \TW_{-\SF,+}\cup_{\overline{\SF}} ((-M_2,-\Gamma_2)\setminus -\overline{\W}),
\]
which is just
\[
-[((M_1,\Gamma_1)\setminus \W)\cup_{\SF} \TW_{\SF,-}\cup_{-\overline{\SF}} ((M_2,\Gamma_2)\setminus \overline{\W})] = (-M,-\Gamma).
\]
Thus, it makes sense for the contact invariants $\EH(\xi_1)\in\SFH(-M_1,-\Gamma_1)$ and $\EH(\xi_2)\in\SFH(-M_2,-\Gamma_2)$ to be paired by the bordered gluing map to $\EH(\xi_1\cup_{\SF}\xi_2)$; see \fullref{lem:zarev preserves eh}.

\begin{remark}
\label{rmk:gluing symbol}
Note that we are using the symbol $\cup_{\SF}$ for the distinct operations of gluing partially sutured manifolds along components of their non-sutured boundaries, gluing balanced sutured manifolds along sutured subsurfaces of their boundaries with dividing sets, and gluing contact structures along convex surfaces. We rely on context to distinguish them.
\end{remark}

Zarev shows that if $(M_1,\Gamma_1)$ and $(M_2,\Gamma_2)$ are balanced, so is $(M_1,\Gamma_1)\cup_{\SF}(M_2,\Gamma_2)$. He also defines a gluing map
\[
\SFH(M_1,\Gamma_1)\otimes\SFH(M_2,\Gamma_2) \xrightarrow{\Psi_{\SF,\gamma}} \SFH((M_1,\Gamma_1)\cup_{\SF}(M_2,\Gamma_2)),
\]
associated to the bordered gluing operation. To avoid confusion with the $\HKM$ contact gluing map, we will refer to this map as the \textit{bordered gluing map}. Also, when $\SF$ is embedded in the boundary of a sutured manifold, we will assume unless otherwise stated that the dividing set is the restriction of the sutures to $\SF$; in this case we omit the dividing set from the notation and write $\Psi_{\SF}$ for the bordered gluing map.

\begin{remark}
\label{rmk:zarev naturality}
The bordered gluing map is well-defined up to maps induced by graded homotopy equivalence. It is not known to be natural with respect to holomorphic triangle maps in the way that the $\HKM$ map is.
\end{remark}

The bordered gluing operation is a special case of the \textit{join operation} found in \cite{Zar11}. Similarly, the bordered gluing map is a special case of the \textit{join map}. We will use the fact the the bordered gluing map possesses the following properties of the join.

\begin{theorem}[Zarev]
\label{thm:join map properties}
The bordered gluing map $\Psi_{\SF}$ satisfies the following properties:

\begin{enumerate}
\item (Symmetry) Let $(M_1,\Gamma_1),(M_2,\Gamma_2)$ be sutured manifolds with embeddings $\SF\subset \partial M_1$, $\overline{\SF}\subset \partial M_2$. There are graded isomorphisms $f,g$ such that the following diagram commutes

\[
\begin{tikzcd}
\SFH(M_1,\Gamma_1)\otimes\SFH(M_2,\Gamma_2)\arrow[r, "f"] \arrow[d, "\Psi_{\SF}"]
& \SFH(M_2,\Gamma_2)\otimes\SFH(M_1,\Gamma_1) \arrow[d, "\Psi_{\overline{\SF}}"] \\
\SFH((M_1,\Gamma_1)\cup_{\SF}(M_2,\Gamma_2)) \arrow[r, "g" ]
& \SFH((M_2,\Gamma_2)\cup_{\overline{\SF}}(M_1,\Gamma_1)))
\end{tikzcd}
\]

\item (Associativity) Let $(M_1,\Gamma_1),(M_2,\Gamma_2), (M_3,\Gamma_3)$ be sutured manifolds and $\SF_1,\SF_2$ be sutured surfaces with embeddings $\SF_1\hookrightarrow\partial M_1$, $(\overline{\SF}_1\sqcup \SF_2)\hookrightarrow\partial M_2$, and $\overline{\SF}_2\hookrightarrow\partial M_3$. Write $\M_k = (M_k,\Gamma_k)$. The following diagram commutes up to graded isomorphism:
\[
\begin{tikzcd}
\SFH(\M_1)\otimes\SFH(\M_2)\otimes\SFH(\M_3)\arrow[r, "\Psi_{\SF_2}"] \arrow[d, "\Psi_{\SF_1}"]
& \SFH(\M_1)\otimes \SFH(\M_2\cup_{\SF_2}\M_3) \arrow[d, "\Psi_{\SF_1}"] \\
\SFH(\M_1\cup_{\SF_1}\M_2)\otimes\SFH(\M_3) \arrow[r, "\Psi_{\SF_2}" ]
& \SFH(\M_1\cup_{\SF_1}\M_2\cup_{\SF_2}\M_3)
\end{tikzcd}
\]
\item (Identity) Let $(M,\Gamma)$ be a sutured manifold with embedded cap $\W$ for $\SF$, and let $\D(\W)$ be the double of $\W$. There is a distinguished class $[\Delta]\in\SFH(\D(\W))$, called the \textnormal{diagonal class} such that the map
\[
\Psi_{\SF}(\hphantom{ }\cdot\hphantom{ },[\Delta]): \SFH(M,\Gamma)\to \SFH((M,\Gamma)\cup_{\SF} \D(\W))
\]
is a graded isomorphism.
\end{enumerate}
\end{theorem}

%%%%%%%%%%%%%%%%%%%%%%%%%%%%%%%%%%%%%%%%%%%%%%%%%%%%%%%
\section{The Bordered Contact Gluing Map} % (fold)
\label{sec:the bordered contact gluing map}
%%%%%%%%%%%%%%%%%%%%%%%%%%%%%%%%%%%%%%%%%%%%%%%%%%%%%%%

The goal of this section is to define the bordered contact gluing map and prove that it satisfies functorial properties akin to the $\HKM$ map.

\subsection{The bordered contact gluing map}
\label{sec:zarev map}

We will need to introduce a technical condition on the dividing set $\gamma$ in order to ensure that the bordered contact gluing map possesses the desired functorial properties found in \fullref{sec:zarev contact gluing}.

\begin{definition}
\label{def:rank one div set}
A dividing set $\gamma$ on a sutured surface $\SF$ is \textit{disk-decomposable} if the underlying sutured manifold of the $I$-invariant contact structure for $\gamma$ on $(F\times I)$ has a product disk decomposition.
\end{definition}

\begin{definition}
\label{def:zarev gluing map}

Let $(M''',\Gamma'')$ and $(M,\Gamma)$ be sutured manifolds which can be glued along a sutured surface $\SF$ with disk-decomposable dividing set $\gamma$. Let $\xi$ be a compatible contact structure on $M''$. The associated \textit{bordered contact gluing map}
\[
\Psi_{\xi}:\SFH(-M,-\Gamma)\to \SFH((-M'',-\Gamma'')\cup_{-\SF}(-M,-\Gamma))
\]
is defined to be $\Psi_{\xi}([\by]) = \Psi_{-\SF}(\EH(\xi),[\by])$. 
\end{definition}

If $(M'',\Gamma'')$ is a contact handle $h^i$ with attaching region $\SF_i$, then we denote this map by $\Psi_{i}: \SFH(-M,-\Gamma)\to \SFH(-M_i,-\Gamma_i)$.

\begin{remark}
\label{rmk:handle div sets}
Note that the dividing sets for contact handle attachments are disk-decomposable; see \fullref{sec:computing the pairing}.
\end{remark}

\begin{remark}
\label{rmk:generalized gluing}
 Note that some restriction on the dividing set is necessary, since the identity property \fullref{lem:zarev identity} is false when the $I$-invariant contact structure is overtwisted. We conjecture that the bordered contact gluing map map satisfies \fullref{lem:zarev identity}, \fullref{lem:zarev preserves eh}, and \fullref{lem:zarev composition} whenever the $I$-invariant contact structure associated to the dividing set is tight. It would suffice to prove \fullref{lem:contact diagonal} for such dividing sets, since the rest of the arguments carry through.
\end{remark}

\begin{lemma}
\label{lem:well defined}
The map $\Psi_{\xi}$ is well-defined up to graded isomorphism.
\end{lemma}

\begin{proof}
This is immediate from the fact that $\Psi_{-\SF}$ is defined up to graded isomorphism.
\end{proof}

\subsection{Properties of the bordered contact gluing map}
\label{sec:zarev contact gluing}

We will now prove that $\Psi_{\xi}$ satisfies identity, $\EH$-preservation, and composition laws analogous to the ones satisfied by the $\HKM$-map. The composition law \fullref{lem:zarev composition} is necessary for the proof of \fullref{thm:main}, and our proof the composition law requires the other two properties.

\begin{lemma}
\label{lem:zarev identity}
Let $(M,\Gamma)$ be a sutured manifold and $\SF\subset \partial M$ be a sutured surface with disk-decomposable dividing set $\gamma$. Let $\W$ be the cap for $-\gamma$ on $-\SF$ and let $\xi_{\gamma}$ be the $I$-invariant contact structure for $-\gamma$ on the double $\D(\W)$. Then
\[
\Psi_{\xi_{\gamma}}: \SFH(-M,-\Gamma)\to \SFH(\D(\W)\cup_{-\overline{\SF}}(-M,-\Gamma))
\]
is a graded isomorphism.
\end{lemma}

\begin{proof}
This follows immediately from the identity and symmetry properties of the bordered gluing map along with \fullref{lem:contact diagonal}.
\end{proof}

\begin{lemma}
\label{lem:contact diagonal}
Let $\W$ be the cap for a disk-decomposable dividing set $\gamma$ on a sutured surface $\SF$. Let $\xi$ be the $I$-invariant contact structure on $F\times I$. Then the diagonal class $[\Delta]\in \SFH(\D(\W))$ is the contact class $\EH(\xi)$.
\end{lemma}

\begin{proof}
The diagonal class is always non-trivial by \cite{Zar11}. Because the dividing set is disk-decomposable, $\SFH(\D(\W))$ has rank 1 and the contact invariant is non-trivial. Thus, the contact invariant and the diagonal class are the same.
\end{proof}

We now wish to prove that the bordered contact gluing map preserves the contact invariant. To do this, we rely on the identification of the bordered contact handle maps with $\HKM$ handle maps provided by \fullref{lem:one handle maps} and \fullref{lem:two handle maps}. This shows that each bordered contact handle map $\Psi_i$ preserves the contact invariant. We then need to leverage associativity of the bordered gluing map; see \fullref{thm:join map properties}. However, we can only do this with contact handles which have disjoint attaching regions, and this cannot always be achieved when the attaching curve of a 2-handle runs over a 1-handle. To get around this problem, we now introduce a version of the padded handles of \fullref{sec:hkm handles} adapted to our needs. We will prove that the bordered contact maps associated to these modified handle attachments are equivalent to the $\Psi_i$.

Let $(M,\Gamma)$ be a balanced sutured manifold and $\xi_{\Gamma}$ be the $[-1,0]$-invariant contact structure in a closed neighborhood $N(\partial M)\cong \partial M\times [-1,0]$ of the boundary $\partial M\times \{0\}$. Let $\Sigma$ be a subsurface of $\partial M$ such that $\gamma = \Gamma|_\Sigma$ is a dividing set for the (possibly degenerate) sutured surface $(\Sigma,-\partial \gamma)$. Let $\xi_\gamma$ be the $[0,1]$-invariant contact structure on $\Sigma\times [0,1]$ compatible with $\gamma$.

If $h^i$ is a contact handle whose attaching region $A$ is contained in $\Sigma$, let $L$ be a Legendrian graph of the form $\{q_k\}\times [0,1]$ where each $q_k$ is a point in $\gamma$ disjoint from $A$ and each component of $\gamma$ contains at least one $q_k$. The \textit{punctured padding} associated to $(\Sigma,A)$ and $L$ is the contact manifold $((\Sigma\times [0,1])\setminus N(L), \xi_\gamma^L)$, where $\xi_\gamma^L$ is the restriction of $\xi_\gamma$ to $(\Sigma\times [0,1])\setminus N(L)$. We denote the underlying sutured manifold by $(P^L,\Gamma^L)$, where $\Gamma^L$ is the dividing set induced by $\xi_\gamma^L$. The \textit{punctured padded contact i-handle} associated to $\Sigma$, $h^i$, and $L$ is $((\Sigma\times [0,1])\setminus N(L), \xi_\gamma^L)\cup_{A\times\{1\}} h^i$; we denote the associated sutured contact manifold by $(P^L_i,\Gamma^L_i,\xi_i^L)$. The sutured surface $(\Sigma\setminus N(\{q_j\}), -\partial( \gamma\setminus N(\{q_j\})))$  is denoted $\SF^L$. Note that $\xi_\Gamma\cup_{A} h^i$ is contactomorphic to $\xi_\Gamma\cup_{\SF^L}\xi^L_i$ relative to $\partial M\times [-1,-\frac12]$.

\begin{remark}
\label{rmk:stacking punctured handles}
Note that a contact handle decomposition of 1-and 2-handles can be described by a punctured padded contact handle decomposition. If a contact structure $\xi$ can be factored into padded handles $(P_{i_1},\xi_{i_1}),\ldots,(P_{i_n},\xi_{i_n})$ attached to an $I$-invariant contact structure $(P,\xi_{\Gamma})$, we can choose Legendrians $L_j\subset P_{i_j}$ so that each $L_j$ can be extended through $P_{i_{j+1}}\cup\ldots\cup P_{i_n}$ without intersecting the attaching region of any contact handle; see \fullref{fig:Punctured Padded contact handles}. One can do this in such a way that $\xi$ restricted to $P\cup P_{i_1}^{L_1}\cup\ldots\cup P_{i_n}^{L_n}$ is contactomorphic to $\xi$. One can also perform a similar construction if each padded handle $(P_{i_j},\xi_{i_j})$ is replaced by $(\Sigma_j\times [j,j+1],\xi_{\gamma_j})\cup h^{i_j}$, where each $\Sigma_j$ is a (possibly degenerate) sutured surface with dividing set $\gamma_j$. We employ this latter case in \fullref{lem:zarev preserves eh} and \fullref{lem:zarev composition}; details are left to the reader.
\end{remark}

In order to make use of these modified contact handles, we need to to know that the dividing set $\Gamma^L|_{\SF^L} =\gamma\setminus N(\{q_j\}) $ is disk-decomposable.

\begin{lemma}
\label{lem:rank one div sets}
Let $\Sigma$ be a (possibly degenerate) sutured surface with dividing set $\gamma$, and let $\{p_k\}$ be a finite collection of points which has non-empty intersection with each closed component of $\gamma$. Let $F$ be the surface $\Sigma\setminus N(\{p_k\})$. Then $(F,-\partial\gamma|_F)$ is a non-degenerate sutured surface and the dividing set $\gamma|_F$ is disk-decomposable.
\end{lemma}

\begin{proof}
$(F,-\partial\gamma|_F)$ is non-degenerate by construction. Let $\xi$ be the $[0,1]$-invariant contact structure on $F\times [0,1]$ compatible with $\gamma|_F$ and $\Gamma_{\xi}$ be the sutures induced by $\xi$. Note that $\xi$ is tight, since $\gamma|_F$ has no closed components.

We can choose a collection of disjoint embedded arcs $\{a_i\}\subset F\setminus\gamma|_F$ with $\partial a_i\subset \partial F$ such that $F\setminus \cup a_i$ is a number of disks. The collection $\{a_i\times[0,1]\}\subset F\times[0,1]$ determines a product disk decomposition of $\xi$. (Each $\partial(a_i\times[0,1])$ intersects $\Gamma_{\xi}$ twice because of smoothing at the corners.) Thus, $\SFH(F\times[0,1],\Gamma_{\xi})$ has rank one.
\end{proof}

Note that punctured padded handles depend on a choice of subsurface $\Sigma$ and Legendrian graph $L$. We will show that this ambiguity does not affect the associated gluing maps. To prove this and the remaining properties of $\Psi_{\xi}$, we assume the following consequence of \fullref{lem:one handle maps} and \fullref{lem:two handle maps}; we may do this, since the proof of these these lemmas does not depend on any of the results in this section.

\begin{figure}[H]
\labellist
	
	\pinlabel $\cong$ at 227 143
	
\endlabellist
\includegraphics[width=1.0\textwidth]{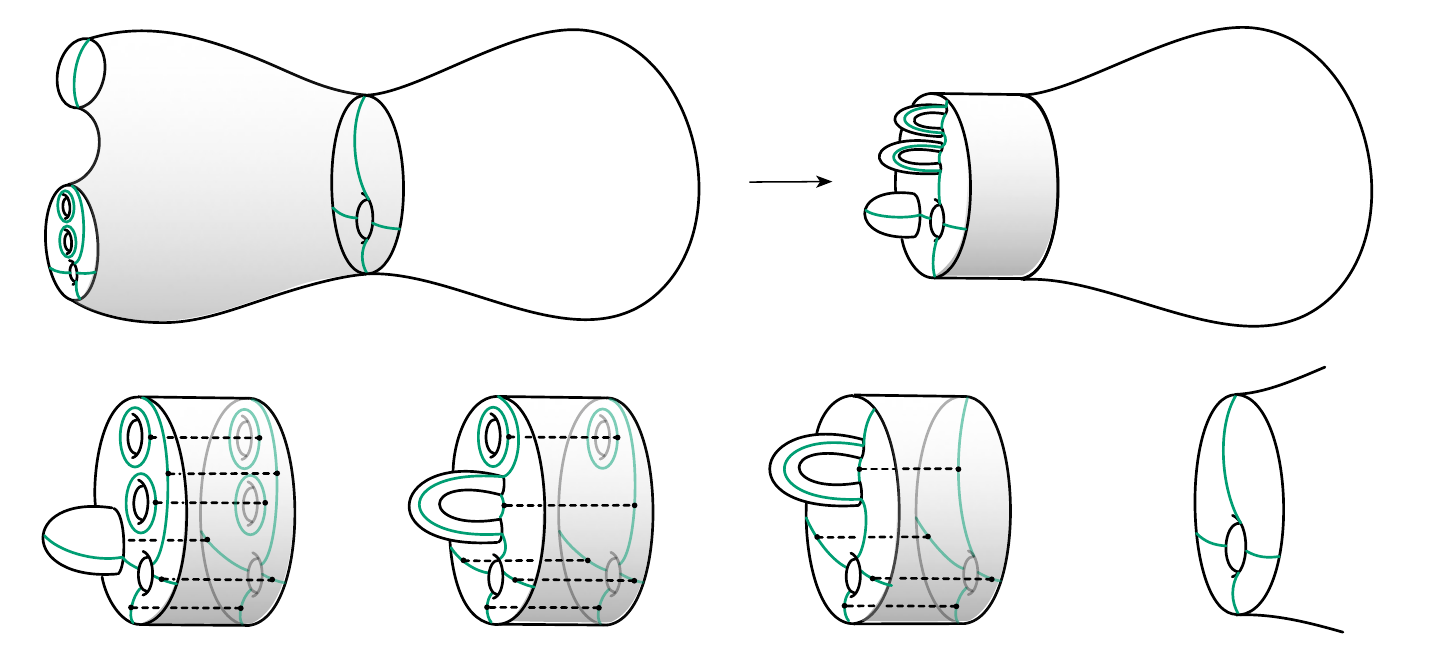}
\caption{Decomposing a contact structure into padded handle attachments. A choice of Legendrians which determine a punctured padded handle decomposition are indicated by dashed lines.}
\label{fig:Punctured Padded contact handles}
\end{figure}

\begin{lemma}
\label{lem:handle preserves eh}
Let $(M_i,\Gamma_i,\xi_i)$ be a sutured contact manifold obtained from $(M,\Gamma,\xi)$ by a contact $i$-handle attachment, and let $\Psi_i:\SFH(-M,-\Gamma)\to (-M_i,-\Gamma_i)$ be the associated bordered contact gluing map. Then $\Psi_i(\EH(\xi)) = \EH(\xi_i)$ up to graded isomorphism.

\end{lemma}

\begin{proof}
By \fullref{lem:padded handles decomposition}, \fullref{lem:one handle maps}, and \fullref{lem:two handle maps}, the bordered contact gluing map $\Psi_i$ equals the corresponding $\HKM$-map $\Phi_i$ up to graded isomorphism. Since $\Phi_i$ preserves the contact class, so does $\Psi_i$ up to graded isomorphism.
\end{proof}

Now we show that the bordered contact gluing maps for punctured padded handle attachments and ordinary contact handle attachments agree.

\begin{lemma}
\label{lem:padding independence}
Let $(M,\Gamma)$ be a sutured manifold, $h^i$ be a contact $i$-handle, and $(P^L_i,\Gamma^L_i,\xi^L_i)$ be a punctured padded contact $i$-handle. Under the identification $(M,\Gamma)\cong (P^L,\Gamma^L)\cup_{\SF^L}(M,\Gamma)$ the corresponding bordered contact gluing maps
\[
\Psi_i,\Psi_{\xi^L_i}:\SFH(-M,-\Gamma)\to\SFH(-M_i,-\Gamma_i)
\]
are equal up to graded isomorphism.
\end{lemma}

\begin{proof}
Split $(P^L_i,\Gamma^L_i,\xi^L_i)$ into a contact handle $h^i$ and the $I$-invariant contact structure $\xi^L_{\Gamma}$ on the punctured padding $(P^L,\Gamma^L)$. By associativity of the bordered gluing map, the following diagram commutes up to graded isomorphism:

\[
\begin{tikzcd}
\SFH(-h^i)\otimes\SFH(-P^L,-\Gamma^L)\otimes\SFH(-M,-\Gamma)\arrow[r, "\Psi_{-\SF_i}"] \arrow[d, "\Psi_{-\SF^L}"]
&\SFH(-P^L_i,-\Gamma^L_i)\otimes\SFH(-M,-\Gamma) \arrow[d, "\Psi_{-\SF^L}"] \\
\SFH(-h^i)\otimes\SFH(-P^L\cup -M, -\Gamma) \arrow[r, "\Psi_{-\SF_i}" ]
& \SFH(-P^L_i\cup -M,\Gamma_i)
\end{tikzcd}
\]

By \fullref{lem:handle preserves eh}, the bordered gluing map $\Psi_{-\SF_i}$ on the top sends $(\EH(h^i),\EH(\xi^L_{\Gamma}))$ to $\EH(\xi^L_i)$, up to graded isomorphism. Evaluating on the contact elements yields the following commutative diagram:
\[
\begin{tikzcd}
\SFH(-M,-\Gamma)\arrow[r, "\id"] \arrow[d, "\Psi_{\xi^L_{\Gamma}}"]
& \SFH(-M,-\Gamma) \arrow[d, "\Psi_{\xi^L_i}"] \\
\SFH(-P^L\cup -M,-\Gamma) \arrow[r, "\Psi_i" ]
& \SFH(-P^L_i\cup -M,\Gamma_i)
\end{tikzcd}
\]
The map $\Psi_{\xi^L_{\Gamma}}$ is a graded isomorphism by \fullref{lem:zarev identity}, since the dividing set on $\SF^L$ is disk-decomposable by \fullref{lem:rank one div sets}. Since the diagram commutes up to graded isomorphism, the lemma follows.
\end{proof}

We now show that the bordered contact gluing map preserves the contact class.

\begin{lemma}
\label{lem:zarev preserves eh}
Let $(M,\Gamma)$ and $(M'',\Gamma'')$ be sutured manifolds with compatible contact structures $\zeta$ and $\xi$ respectively, and let $(M',\Gamma') = (M'',\Gamma'')\cup_{\SF}(M,\Gamma)$ for some sutured surface $\SF$ with disk-decomposable dividing set. Then
\[
\Psi_{\xi}(\EH(\zeta)) = \EH(\xi\cup_{\SF}\zeta)
\]
up to graded isomorphism.
\end{lemma}

\begin{proof}
Decompose $(M'',\Gamma'')$ into a number of punctured padded handles $(P^{L_1}_{i_1},\Gamma^{L_1}_{i_1})$, \ldots, $(P^{L_n}_{i_n}, \Gamma^{L_n}_{i_n})$
and let us consider two particular compositions of bordered gluing maps
\[
\SFH(-P^{L_n}_{i_n},-\Gamma^{L_n}_{i_n})\otimes\ldots\otimes\SFH(-P^{L_1}_{i_1},-\Gamma^{L_1}_{i_1})\otimes\SFH(-M,-\Gamma)\to \SFH(-M',-\Gamma')
\]
corresponding to gluing in different orders. One composition corresponds to attaching the punctured padded handles sequentially to $(M,\Gamma)$, while another composition corresponds to attaching the punctured padded handles sequentially to $(P^{L_1}_{i_1},\Gamma^{L_1}_{i_1})$ to get $(M'',\Gamma'')$ before finally gluing to $(M,\Gamma)$. By \fullref{thm:join map properties}, these compositions give rise to the following diagram which commutes up to graded isomorphism:
\[
\begin{tikzcd}
\SFH(-P^{L_n}_{i_n},-\Gamma^{L_n}_{i_n})\otimes\ldots\otimes\SFH(-P^{L_1}_{i_1},-\Gamma^{L_1}_{i_1})\otimes\SFH(-M,-\Gamma) \arrow[dr, " "] \arrow[d, " "]\\
\SFH(-M'',-\Gamma'')\otimes\SFH(-M,-\Gamma) \arrow[r, "\Psi_{-\SF}" ]
& \SFH(-M',-\Gamma')
\end{tikzcd}
\]

By \fullref{lem:padding independence} and \fullref{lem:handle preserves eh}, each punctured padded handle attachment map preserves the contact invariant up to graded isomorphism. The vertical arrow thus maps $\EH(\xi_{i_n}^{L_n})\otimes\ldots\otimes\EH(\xi_{i_1}^{L_1})\otimes\EH(\zeta)$ to $\EH(\xi)\otimes\EH(\zeta)$, while the diagonal arrow maps  $\EH(\xi_{i_n}^{L_n})\otimes\ldots\otimes\EH(\xi_{i_1}^{L_1})\otimes\EH(\zeta)$ to $\EH(\xi\cup_{\SF}\zeta)$. Since $\Psi_{\xi}$ is the evaluation of $\Psi_{-\SF}$ on $\EH(\xi)$, it must send $\EH(\zeta)$ to $\EH(\xi\cup_{\SF}\zeta)$ up to graded isomorphism.
\end{proof}

We can finally prove the composition law which we need in our proof of \fullref{thm:main}.

\begin{lemma}
\label{lem:zarev composition}
Let $(M_1,\Gamma_1)$, $(M_2,\Gamma_2)$, and $(M_3,\Gamma_3)$ be sutured manifolds with compatible contact structures $\xi$ on $(M_1,\Gamma_1)$ and $\xi'$ on $(M_2,\Gamma_2)$. Suppose that $\SF$ and $\SF'$ are sutured surfaces with disk decomposable dividing sets such that the bordered gluing $(M_1,\Gamma_1)\cup_{\SF}(M_2,\Gamma_2)\cup_{\SF'}(M_3,\Gamma_3)$ is defined. Then
\[
\Psi_{\xi\cup_{\SF}\xi'} = \Psi_{\xi}\circ\Psi_{\xi'}
\]
up to graded isomorphism.
\end{lemma}

\begin{proof}
Write $\M_k = (M_k,\Gamma_k)$. The following diagram commutes up to graded isomorphism by associativity of Zarev's bordered gluing map.
\[
\begin{tikzcd}
\SFH(-\M_1\sqcup -\M_2\sqcup -\M_3)\arrow[r, "\Psi_{-\SF'}"] \arrow[d, "\Psi_{-\SF}"]
& \SFH(-\M_1\sqcup -\M_2\cup_{-\SF'}-\M_3) \arrow[d, "\Psi_{-\SF}"] \\
\SFH(-\M_1\cup_{-\SF}-\M_2\sqcup -\M_3)\arrow[r, "\Psi_{-\SF'}" ]
& \SFH(-\M_1\cup_{-\SF} -\M_2\cup_{-\SF'}-\M_3)
\end{tikzcd}
\]
By \fullref{lem:zarev preserves eh}, the map $\Psi_{-\SF}$ on top sends $\EH(\xi)\otimes\EH(\xi')\otimes[\by]$ to $\EH(\xi\cup_{\SF}\xi')\otimes[\by]$ for any cycle $\by$. By evaluating on contact elements, we obtain the following diagram which also commutes up to graded isomorphism.
\[
\begin{tikzcd}
\SFH(-\M_3)\arrow[r, "\Psi_{\xi'}"] \arrow[d, "\id"]
& \SFH(-\M_2\cup_{-\SF'}-\M_3) \arrow[d, "\Psi_{\xi}"] \\
 \SFH(-\M_3) \arrow[r, "\Psi_{\xi\cup_{\SF}\xi'}" ]
& \SFH(-\M_1\cup_{-\SF}-\M_2\cup_{-\SF'}-\M_3)
\end{tikzcd}
\]

\end{proof}

%%%%%%%%%%%%%%%%%%%%%%%%%%%%%%%%%%%%%%%%%%%%%%%%%%%%%%%
\section{Background on Computing the Bordered Gluing map} % (fold)
\label{sec:computing bordered}
%%%%%%%%%%%%%%%%%%%%%%%%%%%%%%%%%%%%%%%%%%%%%%%%%%%%%%%

This section contains the background necessary for computing the bordered contact gluing map in the case of a contact handle attachment. We will need this in \fullref{sec:diagrammatic zarev} when identifying the bordered map with the diagrammatic maps of \fullref{sec:simple maps}.

\subsection{Arc diagrams, bordered algebras, and bordered sutured diagrams}
\label{sec:arc diagrams}
Computing the bordered gluing map explicitly requires one to identify a particular generator in a certain Heegaard diagram for the twisting slice $\TW_{-\SF,+}$; see \fullref{lem:elementary join}. This section is devoted to the background necessary to construct this diagram and identify the generator. Those familiar with Auroux-Zarev diagrams or willing to take this identification on faith may skip to \fullref{sec:proof of the main theorem} and simply refer to \fullref{fig:twisting slice diagram} and \fullref{lem:elementary join} when proving \fullref{lem:two handle maps}. We first review arc diagrams and describe the bordered algebras associated to contact handles.

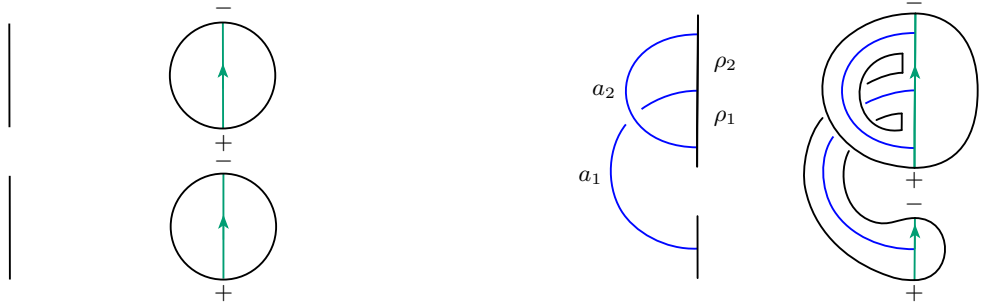
\begin{figure}
\labellist
	\begin{footnotesize}
	\pinlabel $-$ at 114 106
	\pinlabel $+$ at 114 56
	\pinlabel $-$ at 114 49
	\pinlabel $+$ at 114 0
	
	\pinlabel $-$ at 370 108
	\pinlabel $+$ at 370 42
	\pinlabel $-$ at 370 33
	\pinlabel $+$ at 370 0
	\pinlabel $a_1$ at 250 43
	\pinlabel $a_2$ at 255 75
	\pinlabel $\rho_1$ at 300 65
  	\pinlabel $\rho_2$ at 300 85
    \end{footnotesize}
\endlabellist
\includegraphics[width=1.0\textwidth]{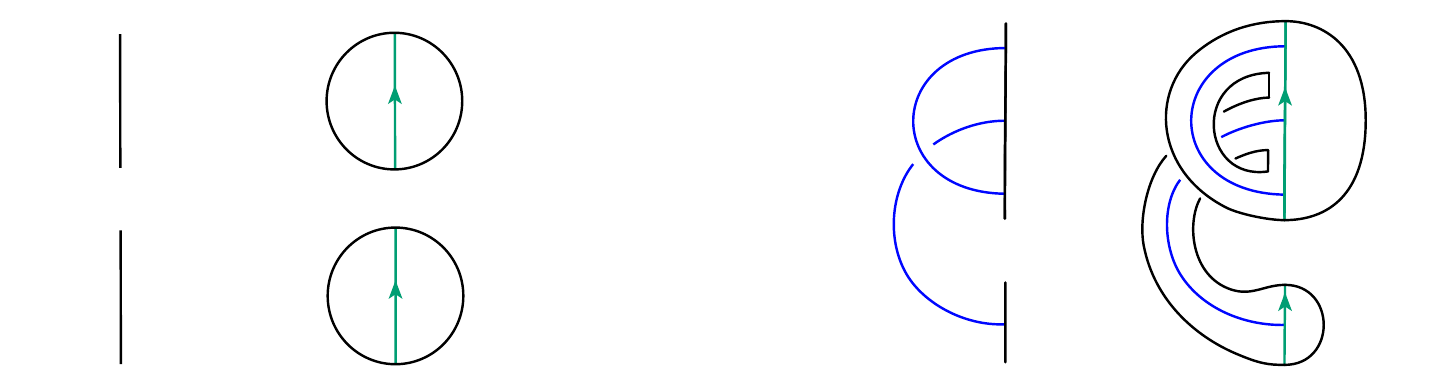}
\caption{Arc diagrams $\SZ_1$ (Far left) and  $\SZ_2$ (Middle right) for contact 1-and 2-handles. The sutured surfaces $\SF_1$ (Middle left) and $\SF_2$ (Far right)  they parametrize are also depicted. Note that these diagrams have $\alpha$-type, though we have drawn the arcs blue since we will be working with Heegaard diagrams of the form $(\Sigma,\bbeta,\balpha)$.}
\label{fig:arc diagram}
\end{figure}

An \textit{arc diagram} $\SZ$ is a union of oriented intervals $\BZ = \sqcup Z_i$, an even number of marked points $\{z_j\}\subset \INT(\BZ)$, and a 2-to-1 function $\mu: \{z_j\}\to \Z_{>0}$. Points $z_j$ and $z_k$ are said to be \textit{matched} if $\mu(z_j) = \mu(z_k)$. We require that the result of surgery on $\BZ$ along all pairs of matched points $\mu^{-1}(l)$ have no closed components. An oriented subinterval of $\BZ$ from one marked point to another (not necessarily matched) marked point is called a \textit{Reeb chord}. Arc diagrams also come with a type: $\alpha$-type or $\beta$-type. Given $\SZ$, there is an arc diagram $-\SZ$ obtained by changing the orientation of each $Z_i$, and another arc diagram $\overline{\SZ}$ obtained by changing the type.

The \textit{graph} $G(\SZ)$ of an arc diagram is the ribbon graph obtained by attaching arcs $\{a_l\}$ to $\BZ$ so that $\partial a_l = \mu^{-1}(l)$. If $\SZ$ is $\alpha$-type, the arcs are attached on the left with respect to the orientation on the intervals $Z_i$; if $\SZ$ is $\beta$-type, the arcs are attached on the right. The graph $\overline{G(\SZ)}$ is obtained by reversing the cyclic ordering around each vertex in $\cup\partial a_l$. Alternatively, $\overline{G(\SZ)}$ is $G(\overline{\SZ})$ as a ribbon graph.

An arc diagram $\SZ$ \textit{parametrizes} a sutured surface $\SF = (F,\Lambda)$ if there is an embedding $G(\SZ)\subset F$ such that $\partial (\cup Z_i) = \Lambda$ as oriented manifolds, and $F$ deformation retracts onto $G(\SZ)$. If $\SZ$ parametrizes $\SF$, then $-\SZ$ parametrizes $-\SF$ and $\overline{\SZ}$ parametrizes $\overline{\SF}$.

The bordered algebra $A = A(\SZ)$ associated to an arc diagram $\SZ$ is generated as an $\F_2$-vector space by certain formal sums of strand diagrams; multiplication corresponds to concatenation of strand diagrams. $A$ is a differential graded algebra over the ground ring $I\subset A$ consisting of idempotents of $A$.

There is a preferred basis of idempotents which is in one-to-one correspondece with collections of arcs of $\SZ$; we denote the basis idempotent corresponding to $\{a_{i_1},\ldots,a_{i_p}\}$ by $\iota_{i_1,\ldots,i_p}$. The basis of idempotents can be extended to a basis of $A$ by adding algebra elements associated to collections of Reeb chords; we will abuse notation and write $\rho$ for the algebra element associated to a single Reeb chord $\rho$. (We will not need to consider algebra elements associated to collections of multiple Reeb chords.)

The algebra and the idempotents split as $A = \bigoplus_i A(i)$ and $I = \bigoplus_i I(i)$, each indexed by the number of strands in the associated diagrams. (Note that we are following Zarev's convention for indexing summands of the algebra, which deviates from the convention in \cite{LOT15}.) Since we only need to understand the bordered algebras for the arc diagrams we have chosen for contact 1-and 2-handles, we omit the general definition here, and explicitly describe the salient features of the algebras we need. The interested reader can find formal definitions in \cite{LOT15} and \cite{Zar11}. Note that the bordered algebras $A(-\SZ)$ and  $A(\overline{\SZ})$ are isomorphic to the opposite algebra $A(\SZ)^{op}$, which is $A(\SZ)$ with multiplication reversed.

Let $\SZ_1$ be the arc diagram on the left of \fullref{fig:arc diagram} which parametrizes the sutured surface $\SF_1 = (F_1,\Lambda_1)$ for a contact 1-handle. The associated bordered algebra $A_1 = A(\SZ_1)$ has rank 1; the only non-trivial summand is $A_1(0) = I_1(0)$, with basis element the idempotent $\iota_\varnothing$ corresponding to the strand diagram with no strands.

\begin{figure}
\labellist
	\begin{footnotesize}
	\pinlabel $\iota_1$ at 95 0
	\pinlabel $\iota_2$ at 155 0
	\pinlabel $\rho_1$ at 215 0
  	\pinlabel $\rho_2$ at 275 0
  	\pinlabel $\rho_{12}$ at 335 0
    \end{footnotesize}
\endlabellist
\includegraphics[width=1.0\textwidth]{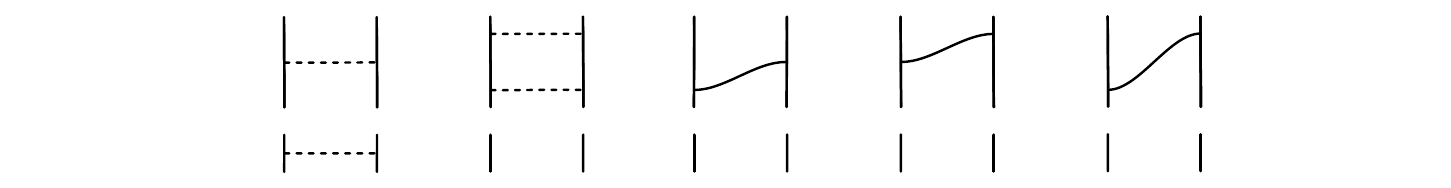}
  \caption{Strand diagrams corresponding to basis elements of $A_2(1)$. The two leftmost diagrams correspond to basis idempotents, while the others correspond to Reeb chords. Multiplication is given by concatenating, i.e. $\rho_1\cdot\rho_2 = \rho_{12}$.}
  \label{fig:two handle strand diagrams}
\end{figure}
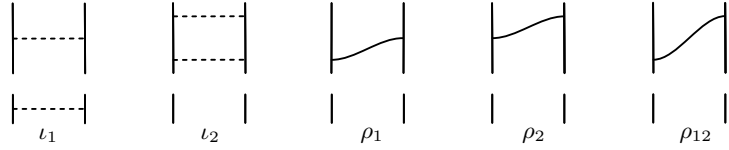

Let $\SZ_2$ be the arc diagram on the right of \fullref{fig:arc diagram} which parametrizes the sutured surface $\SF_2 = (F_2,\Lambda_2)$ for a contact 2-handle. The associated bordered algebra $A_2 = A(\SZ_2)$ has rank 9 with three non-trivial summands. Just as for $\SZ_1$, the summand $A_2(0) = I_2(0)$ has rank 1, with basis element the idempotent $\iota_\varnothing$ corresponding to the strand diagram with no strands. The summand $A_2(1)$ has rank 5, with basis elements corresponding to the strand diagrams depicted in \fullref{fig:two handle strand diagrams}. The ground ring $I_2(1)$ is generated by the basis idempotents $\iota_1,\iota_2$. Each of these corresponds to a formal sum of two strand diagrams, each such sum is represented by a superposition of the two diagrams with dashed lines. The nontrivial actions by idempotents are
\[
\iota_2\cdot\rho_1\cdot\iota_1 = \rho_1 \qquad \iota_1\cdot\rho_2\cdot\iota_2 = \rho_2 \qquad \iota_2\cdot\rho_{12}\cdot\iota_2 = \rho_{12},
\]
while the only other non-trivial multiplication is $\rho_1\cdot\rho_2 = \rho_{12}$. Multiplication is given by concatenating strand diagrams. The summand $A_2(2)$ has rank 3; it does not appear in our computations.

Furthermore, each algebra is equipped with a grading and a differential; we will not need to consider either.

Suppose $\SZ^{\alpha}$ is an $\alpha$-type arc diagram and $\SZ^{\beta}$ is a $\beta$-type arc diagram. A \textit{bordered sutured Heegaard diagram} $\HD = (\Sigma,\balpha,\bbeta,\SZ^{\alpha},\SZ^{\beta})$ consists of 
\begin{itemize}
\item an oriented surface $\Sigma$ without closed components
\item a collection $\balpha = \balpha^c\cup\balpha^a$  of pairwise disjoint, properly embedded curves $\balpha^c$ and arcs $\balpha^a$ in $\Sigma$
\item another collection $\bbeta = \bbeta^c\cup\bbeta^a$ of pairwise disjoint, properly embedded curves $\bbeta^c$ and arcs $\bbeta^a$ in $\Sigma$
\item embeddings $(G(\SZ^{\alpha}), \BZ^{\alpha})\subset (\Sigma,\partial\Sigma)$ and  $(\overline{G(\SZ^{\beta})},\BZ^{\beta})\subset (\Sigma,\partial\Sigma)$ such that $\BZ^{\alpha}$ and $\BZ^\beta$ are disjoint, the orientation on each component of $\BZ^\alpha\sqcup \BZ^\beta$ agrees with the orientation on $\partial\Sigma$, the arcs of $G(\SZ^{\alpha})$ agree with the arcs of $\balpha^a$, the arcs of $G(\SZ^{\beta})$ agree with the arcs of $\bbeta^a$, and no component of $\Sigma\setminus\balpha$ or $\Sigma\setminus\bbeta$ is disjoint from $\partial\Sigma\setminus(\BZ^\alpha \sqcup \BZ^\beta)$.
\end{itemize}

Let $A = A(\SZ^{\alpha})$ and $B=A(\SZ^{\beta})$ be the bordered algebras with respective subrings of idempotents $I^\alpha\subset A$, $I^\beta\subset B$. We can choose to consider the bordered module $\BSAA(\HD)$ as an $(I^\alpha)^{op}, I^{\beta}$--bimodule generated by tuples of intersection points which occupy each curve in $\balpha^c\cup\bbeta^c$ exactly once and each arc in $\balpha^a\cup\bbeta^a$ at most once. If a generator $\bx$ occupies arcs which correspond to $a_{i_1},\ldots, a_{i_p}\subset G(\SZ^\alpha)$ and $b_{j_1},\ldots, b_{j_q}\subset \overline{G(\SZ^\beta)}$, then we have the algebra action $\iota_{{i_1},\ldots, {i_p}}\cdot\bx \cdot \iota_{{j_1},\ldots, {j_q}} = \bx$. Since these are the unique basis idempotents which act nontrivially on $\bx$ we will also write this equation as $\iota_L(\bx)\cdot \bx \cdot \iota_R(\bx) = \bx$.

The differential and $A_{\infty}$-algebra actions are determined by counting certain pseudoholomorphic curves. The $A_{\infty}$-module actions are maps
\[
m_{i|1|j}: ( A^{op})^{\otimes i}\otimes \BSAA(\HD) \otimes  B^{\otimes j}\to \BSAA(\HD)
\]
which satisfy certain relations, where each tensor product is over $(I^\alpha)^{op}$ or $I^{\beta}$ as appropriate.

One similarly defines Type-A and Type-D structures $\BSA$, $\BSD$ for diagrams of the form $(\Sigma,\balpha,\bbeta,\SZ^\alpha,\varnothing)$ or $(\Sigma,\balpha,\bbeta,\varnothing,\SZ^\beta)$. Just as for $\BSAA$, these invariants are generated as $\F_2$-vector spaces over tuples of intersection points which occupy each full curve exactly once and each arc in at most once.

A bordered sutured diagram is $\textit{nice}$, if every component of $\Sigma\setminus(\balpha\cup\bbeta)$ is one of the following:
\begin{enumerate}
\item a region with non-trivial intersection with $(\partial \Sigma)\setminus(\BZ^\alpha\sqcup \BZ^\beta)$; referred to as a $\textit{basepoint region}$ or $\textit{suture region}$
\item a bigon with no side in $\BZ^\alpha\sqcup \BZ^\beta$
\item a quadrilateral with at most one side in $\BZ^\alpha\sqcup \BZ^\beta$.
\end{enumerate}

In this case the differential of $\BSAA(\HD)$ is given by counting each region which is an embedded bigon or rectangle in the interior of the Heegaard surface, and which is disjoint from the basepoint regions. The $A_{\infty}$-module actions are given by counting each region which is an embedded rectangle with one side a Reeb chord $\rho$ in $\BZ^\alpha\sqcup \BZ^\beta\subset \partial\Sigma$ and the other three sides in $\INT(\Sigma)$. More precisely, let $\bx = (x_0, z_1,\ldots,z_n)$ and $\by = (y_0, z_1,\ldots,z_n)$ be generators in a nice bordered diagram $\HD$ with basis idempotent actions $\iota_L(\bx)\cdot\bx\cdot\iota_R(\bx) = \bx$ and $\iota_L(\by)\cdot\by\cdot\iota_R(\by) = \by$, and let $\rho$ be a Reeb chord in $\BZ^\alpha$. Suppose there is a rectangular region $D$ which has corners $\{x_0,y_0\}\cup\partial\rho$, goes out of $x_0$, goes into $y_0$, and has no $z_k$ in its interior. Then $D$ contributes a $\by$ term to the action $m_{0|1|1}(\bx,\iota_L(\bx)\cdot\rho\cdot\iota_R(\by))$. If $\rho$ is a chord in $\BZ^{\beta}$ instead, then $D$ contributes a $\by$ term to the action $m_{1|1|0}(\iota_L(\by)\cdot\rho\cdot\iota_R(\bx),\bx)$. The only other non-trivial actions in either case come from compatible idempotents, i.e. $m_{0|1|1}(\bx,\iota_R(\bx)) = m_{1|1|0}(\iota_L(\bx),\bx) = \bx$.

As an example, we explicitly compute some actions of the $\A_{\infty}$-bimodule $\BSAA(\HD_{AZ})$ associated to the nice diagram $\HD_{AZ} = (\Sigma, \bbeta,\balpha, \overline{\SZ}_2 \sqcup \SZ_2)$ on the right hand side of \fullref{fig:twisting slice diagram}. We will give $\BSAA(\HD_{AZ})$ a right action by $A(-\SZ_2)\approx A(\SZ_2)^{op}$ and a left action by $A(\overline{\SZ}_2)\approx A(\SZ_2)^{op}$. Note that the usual roles of $\alpha$ and $\beta$ curves are reversed and $\SZ_2$ has $\alpha$-type.

The region $D_1$ contributes the action $m_{1|1|0}(\rho_1,z_1) = \rho_1\cdot z_1 = z_4$. The region $D_1 \cup D_2$ contributes the action $m_{1|1|0}(\rho_{12},z_1) = \rho_{12} \cdot z_1 = z_3$. The region $D_5$ contributes the action $m_{0|1|1}(z_2,\rho_1) = z_2 \cdot \rho_1 = z_3$. The region $D_3\cup D_4$ contributes $m_{0|1|1}(z_4,\rho_2) = z_4\cdot \rho_2 = z_5$. The region $D_1\cup D_3\cup D_4$ contributes no action, since it has two sides which hit the boundary.

In fact, $\BSAA(\HD_{AZ})$ is isomorphic to the dual algebra $A(-\SZ_2)^{\vee}$ as an $\A_{\infty}$-bimodule; see page 122 of \cite{Zar11} for the general relationship between the actions of a bimodule $M$ and its dual $M^{\vee}$. This can be seen by identifying $z_1$ with $\rho_{12}^{\vee}$, $z_2$ with $\rho_{1}^{\vee}$, $z_3$ with $\iota_{2}^{\vee}$, $z_4$ with $\rho_{2}^{\vee}$, and $z_5$ with $\iota_{1}^{\vee}$. Under this identification, the actions described above are dual to the multiplications $\rho_2\cdot\rho_1 = \rho_{12}$, $\iota_2\cdot\rho_{12}  = \rho_{12}$, $\rho_1\cdot\iota_2 = \rho_1$, and $\rho_2\cdot\iota_1 = \rho_2$ in $A(-\SZ_2)\approx A(\SZ_2)^{op}$. The reader may verify that the remaining actions of $\BSAA(\HD_{AZ})$ and $A(-\SZ_2)^{\vee}$ correspond exactly.

We will see in the next section that understanding this identification is necessary for computing the bordered gluing map.

\subsection{Computing the bordered gluing map with Heegaard diagrams}
\label{sec:computing the pairing}
In this section we discuss how to compute the bordered gluing map in the special case of contact handle attachments.

Let $(\W,-\SZ) = (W,\gamma,-\SF,-\SZ)$ be the bordered cap for a sutured surface $\SF= (F,\Lambda)$, where $\SF$ is parametrized by $\SZ$. Now, suppose that sutured manifolds $(M_1,\Gamma_1)$ and $(M_2,\Gamma_2)$ can be glued along $\SF$; this is equivalent to saying that there are embeddings $(F,\Lambda)\subset (\partial M_1,\Gamma_1)$ and  $(-F,\Lambda)\subset (\partial M_2,\Gamma_2)$ which extend to embeddings $\W\subset (M_1,\Gamma_1)$ and $\overline{\W}\subset (M_2,\Gamma_2)$. Then $(M_1,\Gamma_1)\setminus \W$ and $(M_2,\Gamma_2)\setminus -\W$ are partially sutured manifolds with sutured surfaces $\SF$ and $\overline{\SF}$, so that $((M_1,\Gamma_1)\setminus \W,\SZ)$ and $((M_2,\Gamma_2)\setminus \overline{\W},\overline{\SZ})$ are bordered sutured manifolds. Note that the positive twisting slice $(\TW_{\SF,+},-\SZ\sqcup -\overline{\SZ})$ is also a bordered sutured manifold.

Choose bordered sutured Heegaard diagrams $\HD_W$ for $(\W,-\SZ)$, $\HD_U$ for $((M_1,\Gamma_1)\setminus \W,\SZ)$, and  $\HD_V$ for $((M_2,\Gamma_2)\setminus \overline{\W},\overline{\SZ})$. These determine Type-D structures $U = \BSD(\HD_U)$ and $V = \BSD(\HD_V)$, along with a Type-A structure $W = \BSA(\HD_W)$. The mirror image $\HD_{-\W}$ of $\HD_{\W}$ is a Heegaard diagram for $\overline{\W}$ and the corresponding Type-A structure is $\BSA(\HD_{-W}) = W^{\vee}$. Let $A = A(\SZ)$. There is a nice diagram $\HD_{AZ}$ for $(\TW_{\SF,+},-\SZ\sqcup -\overline{\SZ})$, discovered independently by Auroux and Zarev, such that $\BSAA(\HD_{AZ}) \cong A^{\vee}$ as $\A_{\infty}$-bimodules over $A$. The Auroux-Zarev diagrams associated to contact handle attachments are depicted in \fullref{fig:twisting slice diagram}; see Proposition 12.4.2 in \cite{Zar11} for the general construction as well as \cite{Aur10} and \cite{LOT11} for computations.

Now, we have
\[
\SFC(M_1,\Gamma_1)\otimes\SFC(M_2,\Gamma_2) \simeq U \boxtimes W \otimes W^{\vee}\boxtimes V,\text{ and}
\]
\[
\SFC((M_1,\Gamma_1)\cup_{\SF}(M_2,\Gamma_2)) \simeq U\boxtimes A^{\vee}\boxtimes V.
\]
With these identifications, \cite{Zar11} gives an explicit formula for computing the bordered gluing map when the bordered module $W = \BSA(\W)$ has a single generator and trivial structure maps; in this case $W$ is called \textit{elementary}. Note that if $\HD = (\Sigma,\balpha,\bbeta,\SZ^\alpha,\SZ^\beta)$ is a bordered sutured diagram with a single generator and all regions adjacent to the intervals $\BZ^\alpha \sqcup \BZ^\beta\subset \partial\Sigma$ are basepoint regions, then the bordered modules for $\HD$ are elementary.

\begin{lemma}(Zarev)
\label{lem:elementary join}
Let $U,V,W,A$ be as above. Suppose that $W$ is an elementary module with unique generator $w$. The bordered gluing map $\Psi_{\SF}$ is induced by the map $\psi_{\SF}:U \boxtimes W \otimes W^{\vee}\boxtimes V\to U\boxtimes A^{\vee}\boxtimes V$ defined by
\[
\psi_{\SF}( \mathbf{u}\boxtimes w\otimes w^{\vee}\boxtimes \mathbf{v}) = \mathbf{u}\boxtimes\iota_L(w)^{\vee}\boxtimes \mathbf{v},
\]
where $\mathbf{u}\in U,\mathbf{v}\in V$ are arbitrary and $\iota_L(w)$ is the basis idempotent in $A$ which acts non-trivially on $w$.
\end{lemma}

\begin{figure}
\labellist

	\pinlabel $\rho_1$ at 230 125
	\pinlabel $\rho_2$ at 230 100  

	\begin{footnotesize}
	\pinlabel $z_4/\rho_2^{\vee}$ at 268 115
	\pinlabel $z_1/\rho_{12}^{\vee}$ at 325 140
	\pinlabel $z_5/\iota_1^{\vee}$ at 315 124
  	\pinlabel $z_2/\rho_1^{\vee}$ at 382 115
    \pinlabel $z_3/\iota_2^{\vee}$  at 325 97
    \end{footnotesize}

    \pinlabel $D_1$ at 253 125
    \pinlabel $D_2$ at 253 98
    \pinlabel $D_3$ at 340 121
    \pinlabel $D_4$ at 394 125
    \pinlabel $D_5$ at 394 98
    
    \pinlabel $\rho_2$ at 415 125
	\pinlabel $\rho_1$ at 415 100   
\endlabellist
\includegraphics[width=1.0\textwidth]{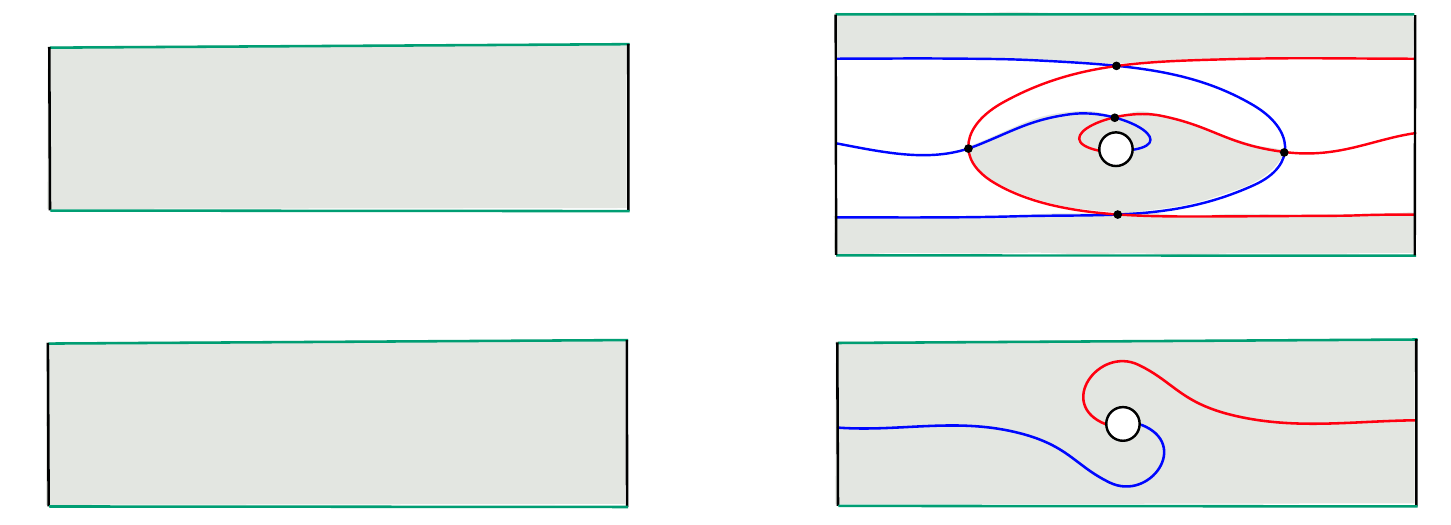}
  \caption{Left: $\HD_{AZ}$ for $\SZ_1$; Right: $\HD_{AZ}$ for $\SZ_2$. Each point is additionally labeled by the basis element it corresponds to under the identification $\BSAA(\HD_{AZ}) \cong A(-\SZ_2)^{\vee}$. The orientations of the Reeb chords agrees with the boundary orientation, i.e. the chords on the right are oriented upward, while the chords on the left are oriented downward.}
  \label{fig:twisting slice diagram}
\end{figure}

\begin{figure}
\includegraphics[width=1.0\textwidth]{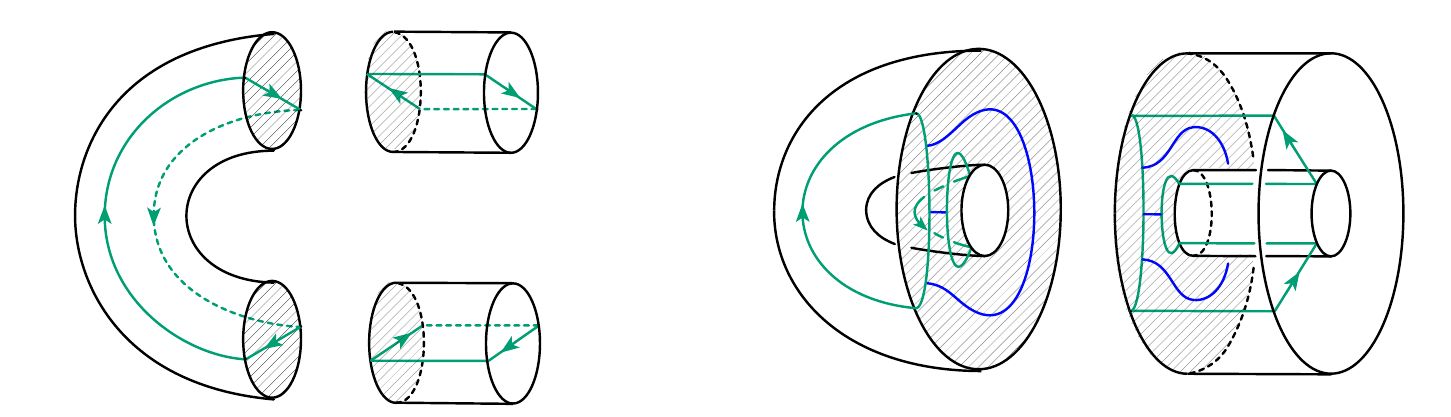}
\caption{Decomposition of a contact 1-handle (Left) and a contact 2-handle (Right) into a cap and a bordered contact handle. The bordered part is shaded and parametrizing arc diagrams are depicted.}
\label{fig:bs caps and handles}
\end{figure}

\fullref{fig:bs caps and handles} shows how to decompose each contact handle $-h^i$ into a bordered cap $(-\W_i,\SZ_i)$ and a bordered contact handle $(-h^i\setminus-\W_i,-\SZ_i)$, along with the parametrizations of $\SF_i$ and $-\SF_i$ in each case. \fullref{fig:bordered contact handle diagrams} shows the diagrams $\HD_U\cup\HD_W$ corresponding to this decomposition. Suppose that $\HD$ is a Heegaard diagram for the manifold $(-M,-\Gamma)$ to which the contact handle $-h^i$ is attached. After some Heegaard moves, we may arrange that a neighborhood of the attaching region in $-M$ be represented by a copy of $\HD_{-W}$ embedded in $\HD$, so that $\HD = \HD_{-W}\cup\HD_V$. To attach $-h^i$, one forms the concatenation $\HD_U\cup \HD_{AZ}\cup\HD_V$. Note that, for each contact handle, the bordered module $\BSA(\HD_W)$ is elementary, so that we can use \fullref{lem:elementary join} for computations.

\begin{figure}
\labellist
	\begin{footnotesize}
 
  \pinlabel $e$ at 108 32
  \pinlabel $w$ at 290 20
  \pinlabel $c$ at 270 32
  	\end{footnotesize}
\endlabellist
\includegraphics[width=1.0\textwidth]{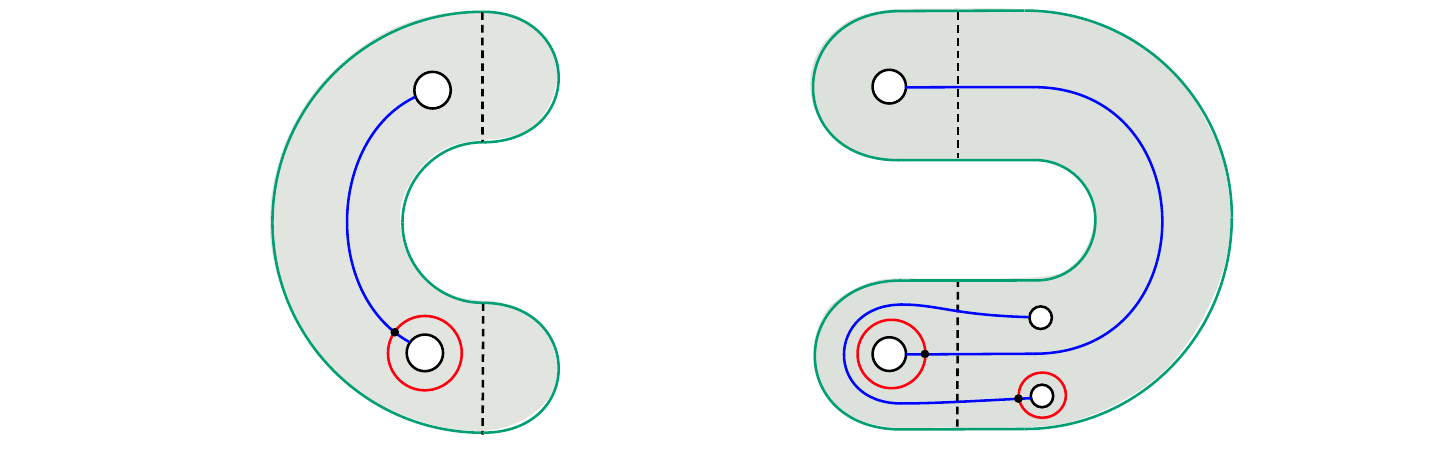}

\caption{$\HD_{U}\cup\HD_{W}$ for a contact 1-handle (Left) and a contact 2-handle (Right).}
  \label{fig:bordered contact handle diagrams}
\end{figure}

%%%%%%%%%%%%%%%%%%%%%%%%%%%%%%%%%%%%%%%%%%%%%%%%%%%%%%%
\section{Proof of the Main Theorem} % (fold)
\label{sec:proof of the main theorem}
%%%%%%%%%%%%%%%%%%%%%%%%%%%%%%%%%%%%%%%%%%%%%%%%%%%%%%%

The goal of this section is the proof of \fullref{thm:main}. We will start by relating the diagrammatic and bordered contact gluing maps for contact handles.

\subsection{Diagrammatic and bordered contact gluing maps}
\label{sec:diagrammatic zarev}

\begin{lemma}
\label{lem:one handle maps}
Given an admissible diagram $\HD$ for $(-M,-\Gamma)$, there is a diagrammatic handle attachment $\HD'$ for $(-M_1,-\Gamma_1)$ such that the diagrammatic map $\sigma_1: \SFC(\HD)\to\SFC(\HD')$ induces the bordered contact gluing map $\Psi_1:\SFH(-M,-\Gamma)\to\SFH(-M_1,-\Gamma_1)$ on homology.
\end{lemma}

\begin{figure}
\labellist
 \begin{footnotesize}
  \pinlabel $p$ at 65 46
  \pinlabel $q$ at 65 26
  
  \pinlabel $e$ at 200 40
 \end{footnotesize}
\endlabellist
\includegraphics[width=1.0\textwidth]{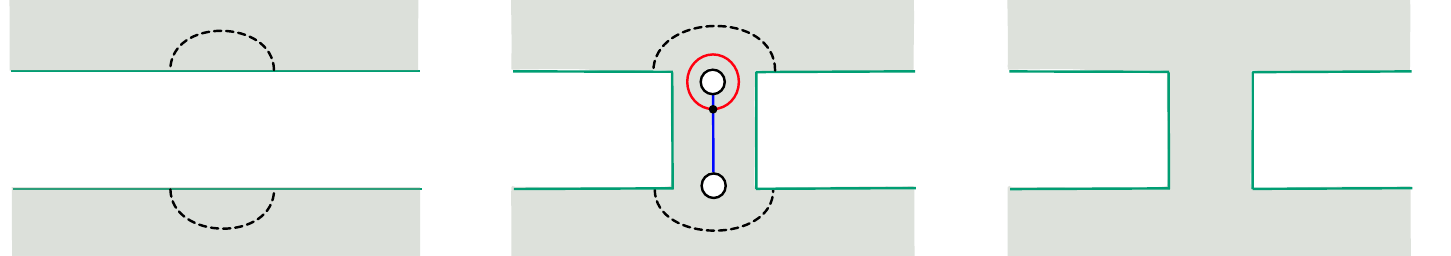}
\caption{Gluing a contact 1-handle.}
  \label{fig:one handle gluing}
\end{figure}

\begin{proof}
Use the diagram $\HD_{U}\cup\HD_W$ shown in \fullref{fig:bordered contact handle diagrams} for $-h^1$. Note that any admissible Heegaard diagram $\HD$ for $(-M,-\Gamma)$ splits as $\HD =  \HD_{-W}\cup\HD_{V}$, where $\HD_{-W}$ is a neighborhood of the attaching region of $h^1$ in $\HD$. 

The intersection point $e$ represents the sole generator $\EH(h^1)$ of $\SFH(-h^1)$; it is also the generator of $\BSD(\HD_U)$. The diagram $\HD_{W}$ for the cap has no curves, so the generator $w$ of $\BSD(\HD_{W})$ is the empty element. (In bordered sutured theory, a diagram with no generator is allowed; in this case the bordered invariants have rank 1, generated by the empty set.) We have $\iota_\varnothing\cdot w = w$, where $\iota_\varnothing$ is the unique generator of $A(-\SZ_1)$; it corresponds to the strand diagram with no strands. Fix an arbitrary cycle $\by\in\SFC(\HD)$ and note that it corresponds to $(w^{\vee},\by)\in\SFC(\HD_{-W}\cup \HD_V)$. Since $\BSD(\HD_W)$ is an elementary module, we can use \fullref{lem:elementary join} to see that the bordered contact gluing map is induced by

\[
\psi_{-\SF}(e\otimes \by) = \psi_{-\SF}((e,w)\otimes (w^{\vee},\by)) = (e,\iota_\varnothing^{\vee}, \by).
\]

\begin{figure}[H]
\labellist
 \begin{footnotesize}
  \pinlabel $x_0$ at 105 95
  \pinlabel $x_0$ at 105 20
  
  \pinlabel $y_0$ at 325 80
 \end{footnotesize}
 	\pinlabel $\alpha_0$ at 40 80
	\pinlabel  $\beta_0$ at 95 35
 
 	\pinlabel $\beta_{g+1}$ at 305 35
 	\pinlabel $\alpha_{g+1}$ at 365 35
\endlabellist
\includegraphics[width=1.0\textwidth]{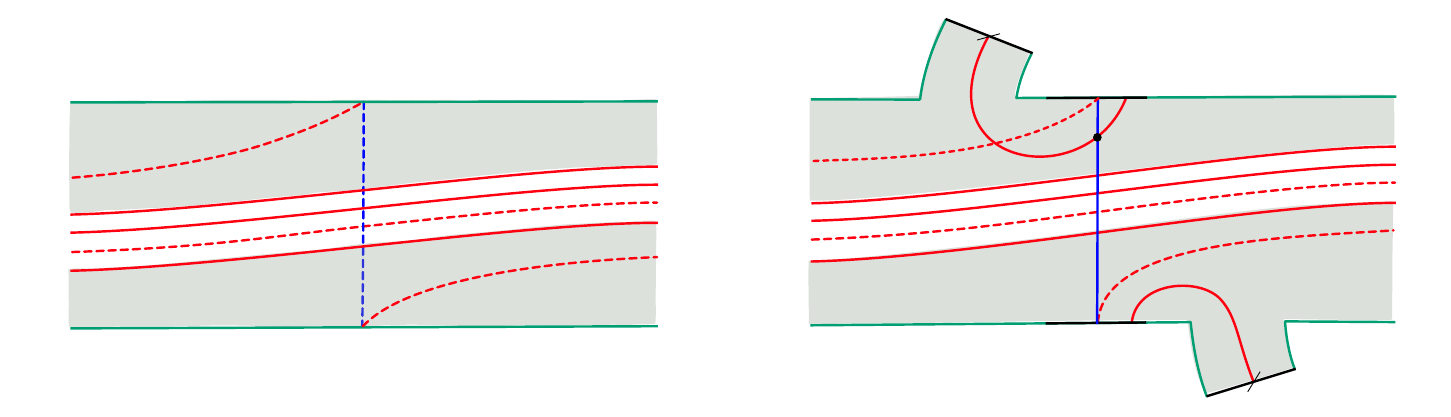}
  
  \caption{(Left) The diagram $\HD_1$ with arcs for the contact 2-handle; (Right) Attaching the first trivial bypass.}
  \label{fig:pregluing bypass1}
\end{figure}
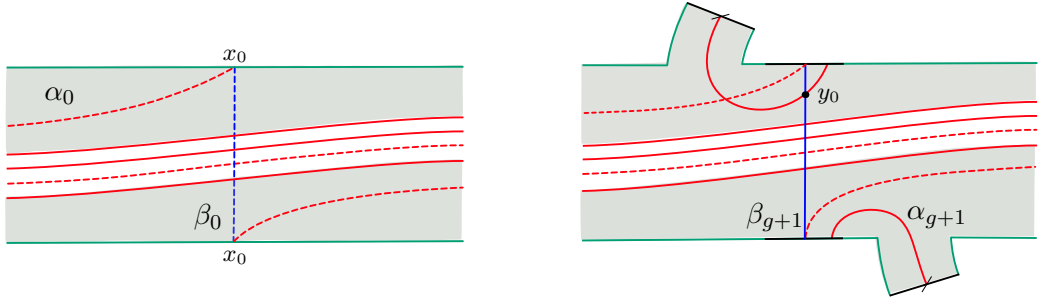

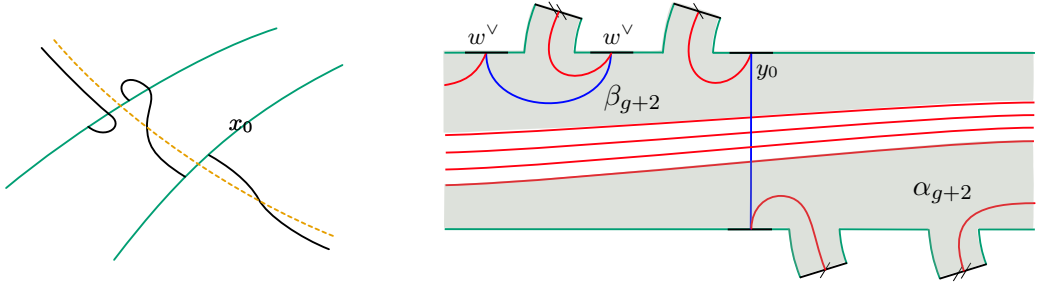
\begin{figure}[H]
\labellist
 \begin{footnotesize}
  \pinlabel $w^{\vee}$ at 195 94
  \pinlabel $w^{\vee}$ at 245 94
  \pinlabel $x_0$ at 105 60
  \pinlabel $x_0$ at 105 60
  \pinlabel $y_0$ at 300 80
 \end{footnotesize}
 \pinlabel $\beta_{g+2}$ at 250 70
 	\pinlabel $\alpha_{g+2}$ at 365 35
\endlabellist
\includegraphics[width=1.0\textwidth]{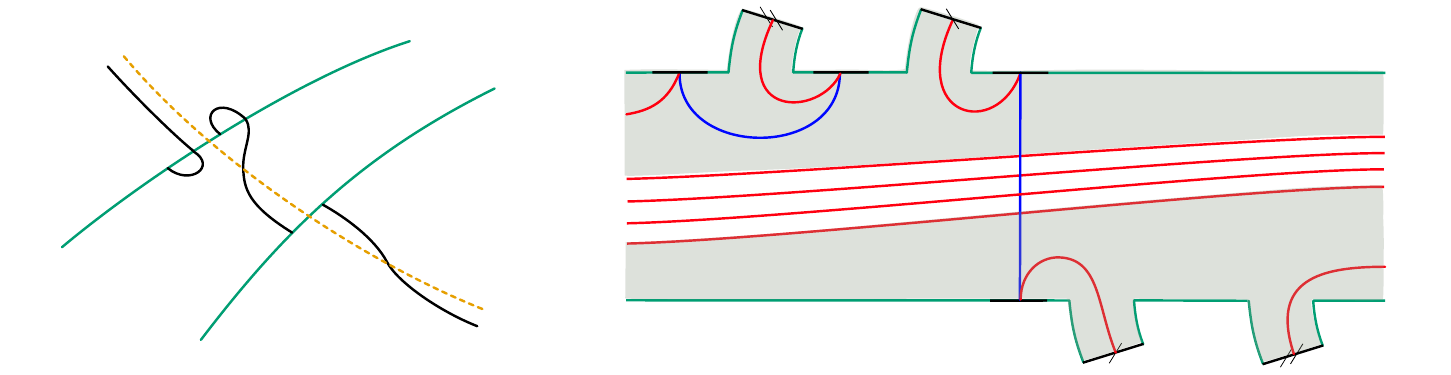}
  
  \caption{(Left) The attaching curve for the 2-handle and the associated trivial bypass arcs in $\partial M$; (Right) Attaching the second trivial bypass to obtain $\HD_2$.}
  \label{fig:pregluing bypass2}
\end{figure}

Since the diagram $\HD_{AZ}$ for the twisting slice has no curves, $\iota_\varnothing^{\vee}$ is the empty set, and we have $(e,\iota^{\vee}, \by) = (e,\by)$. Now, destabilizing yields the diagram for $(-M,-\Gamma)$ with a 1-handle attached to the Heegaard surface, exactly as in the case of the diagrammatic map. The map $(e,\by)\to \by$ is a graded homotopy equivalence, so that $\Psi_1[\by] =[ \by] = [\sigma_1(\by)]$.

Note that $\HD' = \HD_U\cup\HD_{AZ}\cup\HD_V$ is admissible, since its periodic domains correspond to the periodic domains of $\HD$.
\end{proof}

\begin{lemma}
\label{lem:two handle maps}
Given an admissible diagram $\HD_1$ for $(-M,-\Gamma)$, there are admissible diagrams $\HD_3$ for $(-M,-\Gamma)$, $\HD_4$ for $(-M_2,-\Gamma_2)$, and $\HD_6$ for $(-M_2,-\Gamma_2)$;  and graded isomorphisms $f,g$ making the following diagram commute.
\begin{center}
\begin{tikzcd}
\SFH(\HD_1) \arrow[r, "f"] \arrow[d, "(\sigma_{2})_*"]
&\SFH(\HD_3)  \arrow[d, "\Psi_{2}"]\\
\SFH(\HD_6) 
& \SFH(\HD_4) \arrow[l, "g" ]
\end{tikzcd}
\end{center}
\end{lemma}

\begin{proof}

We will show how to compute $\Psi_2$ from a given diagram for $\sigma_2$ by a sequence of diagrams $\{\HD_i\}$. For the most part, the diagram $\HD_{i+1}$ will differ from $\HD_i$ by a number of Heegaard moves, so we have graded isomorphisms $\Psi_{\HD_i,\HD_{i+1}}:\SFH(\HD_i)\to\SFH(\HD_{i+1})$ induced by holomorphic triangle counts; $f$ and $g$ will be compositions of such maps. To ease discussion, we defer the computations of these triangle maps to \fullref{sec:triangle maps}.

Let $\HD_1$ be an admissible Heegaard diagram for $(-M,-\Gamma)$ and use the diagram $\HD_{U}\cup\HD_W$ shown in \fullref{fig:bordered contact handle diagrams} for $-h^2$. The tuple $(c, w)$ represents the unique generator $\EH(h^2)$ of $\SFH(-h^2)$, while $w$ is the sole generator for $\BSA(\HD_{W})$. Referencing \fullref{fig:twisting slice diagram}, we have $\iota_2\cdot w = w$. The dashed lines in \fullref{fig:pregluing bypass1} are the curves $\alpha_0,\beta_0$ for the diagrammatic 2-handle attachment. Fix a cycle $\by\in\SFC(\HD_1)$.

In order to compute the bordered contact gluing map, we must find a way to realize $(-M,-\Gamma)$ as $-\overline{\W}_2\cup_{-\SF}((-M,-\Gamma)\setminus\-\overline{\W}_2)$ at the level of Heegaard diagrams, where $\W_2$ is the cap for the 2-handle. For this purpose, we peform two stabilizations. The first is depicted on the right side of \fullref{fig:pregluing bypass1} so that $\beta_{g+1}$ agrees with $\beta_0$. The second is depicted in \fullref{fig:pregluing bypass2}. This yields a diagram $\HD_2$ where $\alpha_{g+2}$ agrees with $\alpha_0$ outside a neighborhood of $\beta_0$. (One can think of these stabilizations as attaching trivial bypasses along arcs which approximate the attaching curve for the 2-handle; see \fullref{fig:pregluing bypass2}.) The map $\Psi_{\HD_1,\HD_2}$ sending $[\by]$ to $[(w^{\vee},y_0, \by)]$ is induced by a well-defined isomorphism of graded complexes. Handleslide $\beta_{g+1}$ over $\beta_{g+2}$ to obtain the diagram $\HD_3$. The triangle count $\Psi_{\HD_2,\HD_3}$ sends $[(w^{\vee},y_0, \by)]$ to $[(w^{\vee},y_0, \by)]$.

\begin{figure}
\labellist
	\begin{footnotesize}
  	\pinlabel $w^{\vee}$ at 30 75
  	\pinlabel $y_0$ at 112 78
  	
  	\pinlabel $x_0$ at 315 95
  	\pinlabel $y_0$ at 322 75
  	\pinlabel $x_0$ at 315 18
  	\end{footnotesize}
\endlabellist
\includegraphics[width=1.0\textwidth]{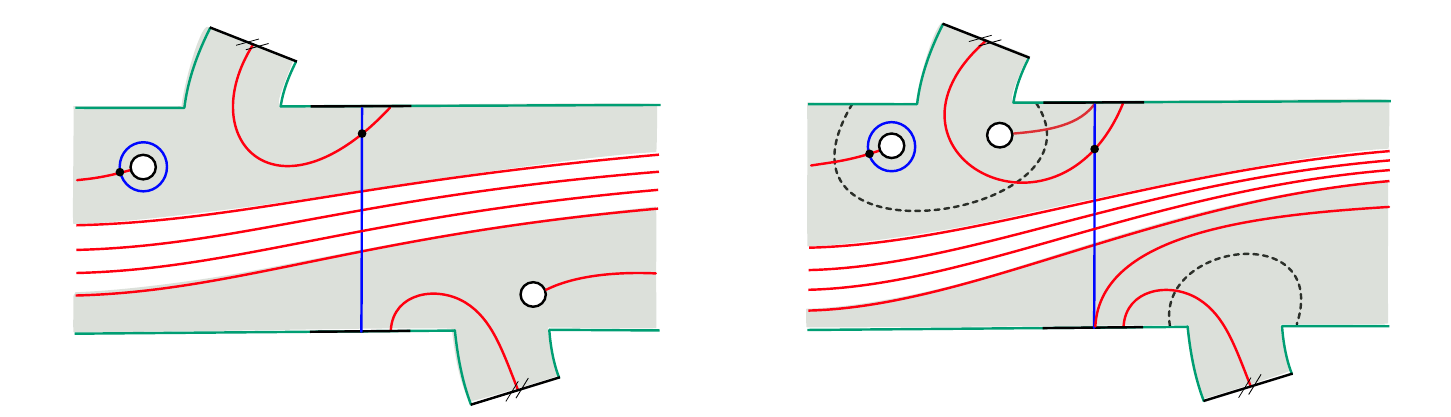}
  \caption{Left: $\HD_2$; Right: $\HD_3$}
  \label{fig:pregluing stabilization}
\end{figure}

\begin{figure}
\labellist
 \begin{footnotesize}
 
  \pinlabel $z_3$ at 137 93
  \pinlabel $c$ at 66 115

  \pinlabel $z_1$/$\rho_{12}^{\vee}$ at 324 133
  \pinlabel $z_2$/$\rho_1^{\vee}$ at 327 119
  \pinlabel $z_3$/$\iota_2^{\vee}$ at 324 105
  
 \end{footnotesize}
 	\pinlabel $\alpha_{g+2}'$ at 86 130
 	\pinlabel $\alpha_{g+3}$ at 62 95
 
\endlabellist
\includegraphics[width=1.0\textwidth]{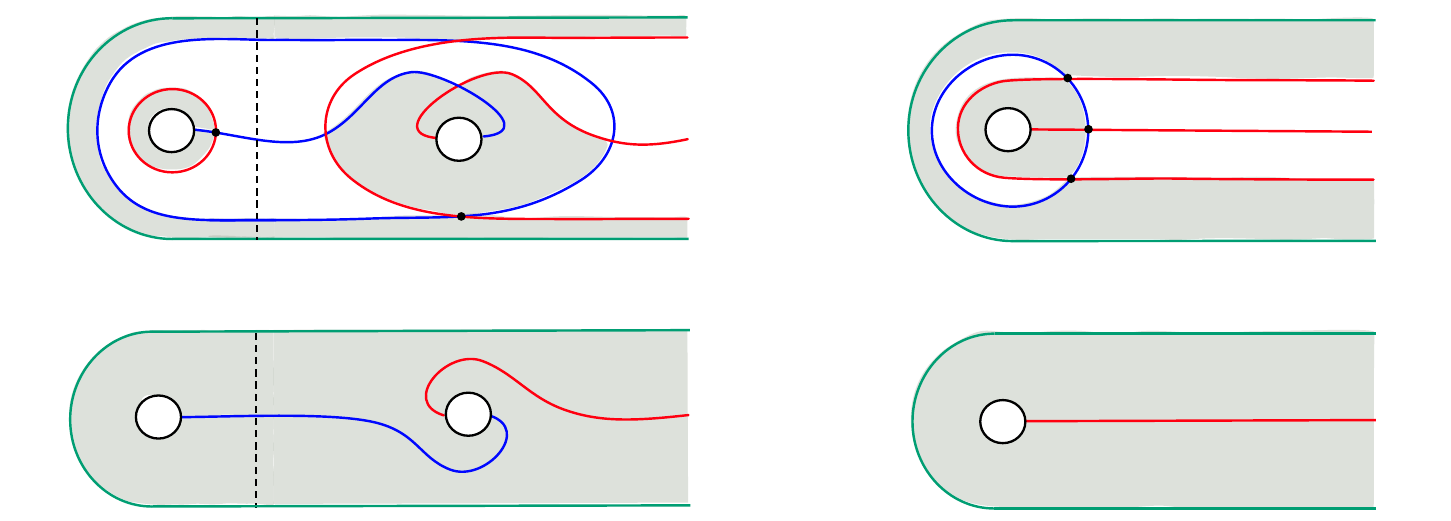}
  \caption{Destabilizing from $\HD_4$ to $\HD_5$; points on the right are additionally labelled by the algebra elements in $A(-\SZ_2)^{\vee}$ they correspond to as in \fullref{fig:twisting slice diagram}.}
  \label{fig:twisted handle destabilization}
\end{figure}

The right hand side of \fullref{fig:pregluing stabilization} shows how to split $\HD_3$ as $\HD_3 = \HD_{-W}\cup\HD_V$. Performing a bordered gluing with the 2-handle yields the diagram $\HD_4 = \HD_U\cup\HD_{AZ}\cup\HD_V$. Note that $z_3$ corresponds to the basis idempotent $\iota_2$ which acts non-trivially on $w$; compare with \fullref{fig:twisting slice diagram}. The associated bordered gluing map is induced by
\[
\psi_{-\SF}((c,w)\otimes(w^{\vee},y_0, \by)) = (c,z_3,y_0, \by).
\]
After gluing, we destabilize to obtain $\HD_5$ as depicted in \fullref{fig:two handle gluing}.  The part of $\HD_5$ which represents $-h^2\cup \TW_{-\SF,+}$ before and after the destabilization are shown in \fullref{fig:twisted handle destabilization}. This destabiliation is equivalent to performing two handleslides followed by a trivial destabilization. The triangle count $\Psi_{\HD_4,\HD_5}$ sends $[(c,z_3,y_0, \by)]$ to $[(z_3,y_0, \by)]$.

\begin{figure}
\labellist
	\begin{footnotesize}
  	\pinlabel $z_3$ at 125 120
  	\pinlabel $z_1$ at 158 120
  	\pinlabel $z_2$ at 140 87
  	\pinlabel $x_0$ at 220 115
  	\pinlabel $y_0$ at 225 95
  	\pinlabel $x_0$ at 220 17
	\end{footnotesize}
	\pinlabel $\beta_{g+2}'$ at 140 132
	\pinlabel $\beta_{g+1}$ at 205 40
\endlabellist
\includegraphics[width=1.0\textwidth]{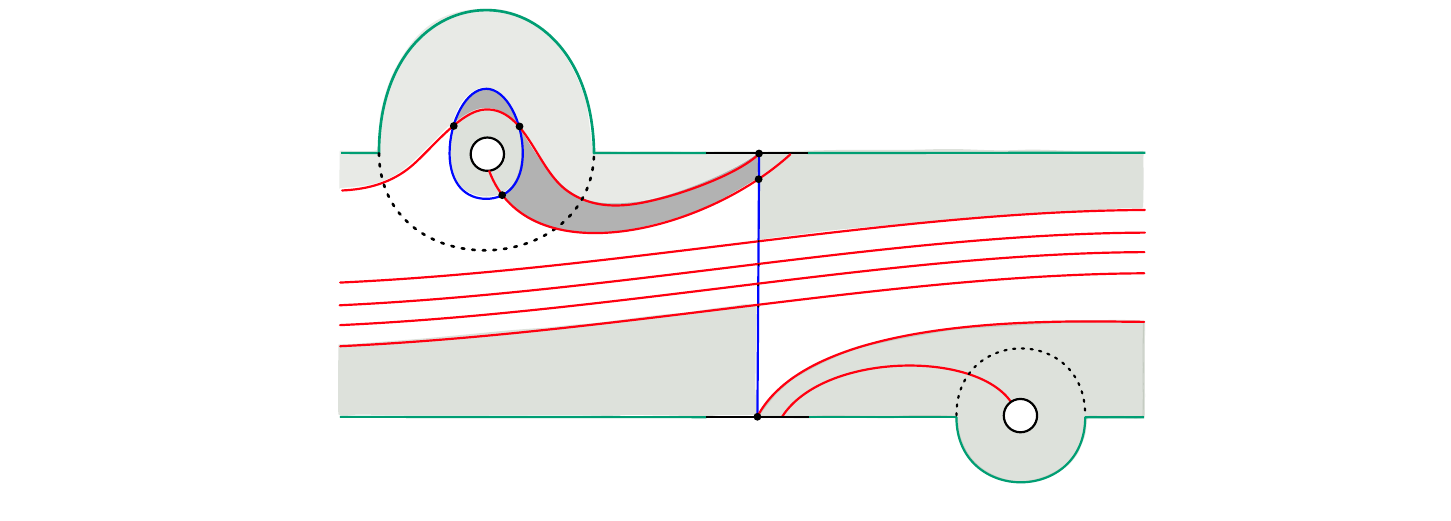}
  \caption{$\HD_5$; the regions shaded grey represent terms in $\partial(\by,y_0,z_1)$. }
  \label{fig:two handle gluing}
\end{figure}
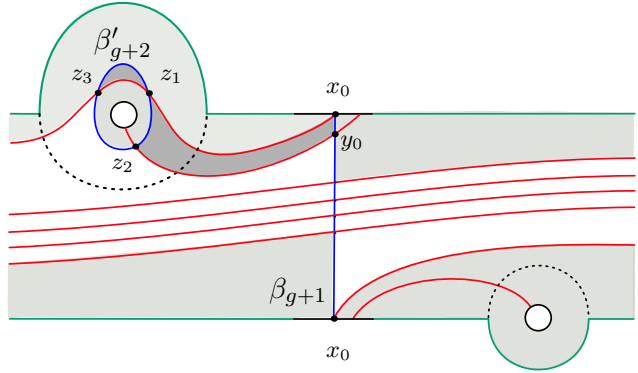

Since $\by$ is a cycle, we have $\partial (z_1,y_0, \by) = (z_3,y_0, \by) + (z_2,x_0,\by)$; this can be seen from \fullref{fig:two handle gluing}, where the shaded regions correspond to the terms on the right. Thus, $[(z_3,y_0, \by)] = [(z_2,x_0,\by)]$.

\begin{figure}
\labellist
	\begin{footnotesize}
  \pinlabel $x_0$ at 220 87
  \pinlabel $x_0$ at 220 0
  	\end{footnotesize}
\endlabellist
\includegraphics[width=1.0\textwidth]{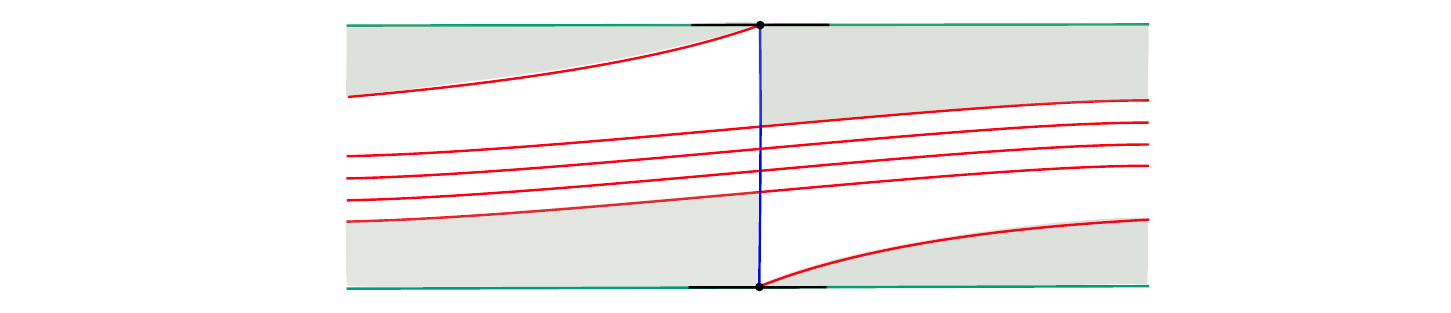}
\caption{$\HD_6$}
\label{fig:simple two handle comparison}
\end{figure}

Now, perform a small Hamiltonian isotopy to remove $z_1$ and $z_3$, handleslide $\beta_{g+1}$ over $\beta_{g+2}'$, and perform a trivial destabilization to obtain $\HD_6$; this is the diagram for a diagrammatic 2-handle attachment on $\HD_1$. The associated triangle count $\Psi_{\HD_5,\HD_6}$ sends $[(z_2,x_0, \by)]$ to $[(x_0,\by)]$.

By composing these maps, we see that the bordered contact gluing map from $\HD_1$ to $\HD_6$ is
\[
[\by]\xrightarrow{\Psi_{\HD_1,\HD_2}} [(w^{\vee},y_0, \by)] \xrightarrow{\Psi_{\HD_2,\HD_3}} [(w^{\vee},y_0, \by))]\xrightarrow{\Psi_2} [(c,z_3,y_0, \by)]
\]
\[
\xrightarrow{\Psi_{\HD_4,\HD_5}} [(z_3,y_0, \by)] = [(z_2,x_0,\by)]\xrightarrow{\Psi_{\HD_5,\HD_6}} [(x_0,\by)] = [(\by,x_0)] ,
\]
which is the map on homology induced by $\sigma_2$.

Note that all the diagrams we have used are admissible by the following argument. In general, let $\alpha\subset\balpha$ be a curve in a diagram which abuts a suture region on each side. Without loss of generality, suppose that $\alpha$ appears in the boundary of a periodic domain $\D$, and the multiplicity of $\D$ on each region to the left of $\balpha$ is non-negative. Since there is a basepoint region $D$ on the left where $\D$ has multiplicity zero; let $a = \partial D \cap\alpha$. Let $D'$ be the region with $\partial D'\cap \alpha = -a$. Then $\D$ must have negative multiplicity on $D'$. 

Thus, if one is looking for a periodic domain with no negative coefficients, one can erase any curve (not just an $\alpha$-curve) which abuts the basepoint region on both sides. In each of the diagrams $\HD_2$ through $\HD_6$, we can erase a number of such curves to obtain $\HD_1$ with a number of strips attached to the boundary; the periodic domains of each diagram obtained this way are exactly the periodic domains of $\HD_1$. Since $\HD_1$ is admissible by hypothesis, so are $\HD_2$ through $\HD_6$.
\end{proof}

\subsection{Proof of Theorem 1}
\label{sec:main proof}

We are now ready to prove the main theorem. The precise statement is as follows.

\begin{theorem}
\label{thm:main duplicate}

Suppose that $\Phi_{\xi}$ is the $\HKM$ map for  a proper inclusion $(M,\Gamma)\subset(M',\Gamma')$ of sutured manifolds with no isolated components and compatible contact structure $\xi$. There is a sutured contact manifold $(M'',\Gamma'',\xi'')$, a sutured surface $\SF = (F,\Lambda)$, and graded isomorphisms $f,g$ such that
\begin{enumerate}
\item $(M',\Gamma') \cong  (M'',\Gamma'')\cup_{\SF}(M,\Gamma)$
\item $(M'\setminus \INT(M),\xi)$ is contactomorphic to $(M''\cup_F(\partial M\times[0,1]),\xi''\cup_{\SF} \xi_{\Gamma})$, where $\xi_\Gamma$ is the $[0,1]$-invariant contact structure compatible with $\Gamma$
\item the dividing set $\Gamma|_F$ is disk-decomposable as defined in \fullref{sec:zarev map}
\item the following diagram commutes
\end{enumerate}
\begin{center}
\begin{tikzcd}
\SFH(-M,-\Gamma) \arrow[r, "f" ] \arrow[d, "\Phi_{\xi}"]
&\SFH(-M,-\Gamma)\arrow[d, "\Psi_{\xi''}"]\\
\SFH(-M',-\Gamma') \arrow[r, "g" ]
& \SFH((-M'',-\Gamma'') \cup_{-\SF} (-M,-\Gamma))
\end{tikzcd}
\end{center}

\end{theorem}

\begin{proof}[Proof of \fullref{thm:main}]
Decompose $\xi$ into a sequence of contact handles $h^{i_1},\ldots,h^{i_n}$ attached to $\partial M\times\{1\}$ in $(\partial M\times [0,1],-\Gamma\sqcup\Gamma, \xi_{\Gamma})$. We can factor the $\HKM$ map as $\Phi_{\xi} = \Phi_{i_n}\circ\ldots\circ \Phi_{i_1}$, where each $\Phi_{i_j}$ is the $\HKM$-map for attaching a padded handle $(P_{i_j},\Gamma_{i_j},\xi_{i_j})$.

We can also decompose $\xi$ as a sequence of punctured padded handles $(P_{i_1}^{L_1},\Gamma_{i_1}^{L_1},\xi_{i_1}^{L_1}),\ldots,$
$(P_{i_n}^{L_n},\Gamma_{i_n}^{L_n},\xi_{i_n}^{L_n})$ attached to $(\partial M\times [0,1],-\Gamma\sqcup\Gamma, \xi_{\Gamma})$ along a surface $F\subset \partial M\times\{1\}$. Let $(M'',\Gamma'',\xi'')$ be the structure obtained by attaching $(P_{i_2}^{L_2},\Gamma_{i_2}^{L_2},\xi_{i_2}^{L_2}),\ldots,(P_{i_n}^{L_n},\Gamma_{i_n}^{L_n},\xi_{i_n}^{L_n})$ to $(P_{i_1}^{L_1},\Gamma_{i_1}^{L_1},\xi_{i_1}^{L_1})$.

By construction, $(M'',\Gamma'')\cup_{\SF}(M,\Gamma)$ can be decomposed as a number of punctured padded handle attachments to $(M,\Gamma)$ with bordered contact gluing maps $\Psi_{i_1},\ldots,\Psi_{i_n}$, where each $\Psi_{i_j}$ corresponds to 
$(P_{i_j}^{L_j},\Gamma_{i_j}^{L_j},\xi_{i_j}^{L_j})$. By \fullref{lem:padding independence}, these are also the maps for the $h_{i_j}$ up to graded isomorphism. The bordered contact gluing map $\Psi_{\xi''}$ factors as $\Psi_{\xi''} = \Psi_{i_n}\circ\ldots\circ \Psi_{i_1}$ up to graded isomorphism by \fullref{lem:zarev composition}.

 By \fullref{lem:padded handles decomposition}, \fullref{lem:one handle maps}, and \fullref{lem:two handle maps} the gluing maps $\Phi_{i_j}$ and $\Psi_{i_j}$ equal $(\sigma_{i_j})_*$ up to graded isomorphism for each $j$. By the composition law for the $\HKM$ map and \fullref{lem:zarev composition}, $\Phi_{\xi} = \Psi_{\xi''}$ up to graded isomorphism.

\end{proof}

While we have couched \fullref{thm:main} in terms of constructing a bordered contact map from an $\HKM$ map, we can also go in the other direction.

\begin{theorem}
\label{thm:main converse}

Let $(M'',\Gamma'')\cup_{\SF} (M,\Gamma)$ be a bordered gluing, where the dividing set $\Gamma|_F$ on $\SF$ is disk-decomposable, and let $\xi''$ be a compatible contact structure on $M''$. There is a sutured manifold $(M',\Gamma')$, a proper inclusion $(M,\Gamma)\subset(M',\Gamma')$, a compatible contact structure $\xi$ on $M'\setminus\INT(M)$, and graded isomorphisms $f',g'$ such that
\begin{enumerate}
\item $(M',\Gamma') \cong (M'',\Gamma'')\cup_{\SF} (M,\Gamma)$
\item $(M'\setminus \INT(M),\xi)$ is contactomorphic to $(M''\cup_{\SF}(\partial M\times[0,1]),\xi''\cup_{\SF} \xi_\Gamma)$, where $\xi_\Gamma$ is the $[0,1]$-invariant contact structure compatible with $\Gamma$
\item the proper inclusion has no isolated components
\item the following diagram commutes
\end{enumerate}
\begin{center}
\begin{tikzcd}
\SFH(-M,-\Gamma) \arrow[r, "f'" ] \arrow[d, "\Psi_{\xi''}"]
&\SFH(-M,-\Gamma)\arrow[d, "\Phi_{\xi}"]\\
\SFH((-M'',-\Gamma'') \cup_{-\SF} (-M,-\Gamma)) \arrow[r, "g'" ]
& \SFH(-M',-\Gamma')
\end{tikzcd}
\end{center}

\end{theorem}

\begin{proof}
Let $(M',\Gamma') = (M'',\Gamma'')\cup_{\SF\times\{1\}}(M\cup_{\partial M\times \{0\}} (\partial M\times [0,1]),\Gamma)$ and let $\xi = \xi''\cup_{\SF\times\{1\}}\xi_\Gamma$. It is clear that the first two conditions are satisfied by construction.

Furthermore, any isolated component of the proper inclusion $(M,\Gamma)\subset (M',\Gamma')$ would correspond to a component of $M''$ whose boundary is entirely contained in $F$, or a closed component of $M''$. Since bordered gluings are only allowed along proper sutured subsurfaces, the former case is not allowed, while the latter is excluded by hypothesis.

Since $\xi$ and $\xi''$ are constructed from $I$-invariant contact structures by attaching corresponding contact handles, we can argue as in the proof of \fullref{thm:main} that $\Phi_{\xi}$ and $\Psi_{\xi''}$ are equal up to graded isomorphism.
\end{proof}

\subsection{Triangle counts}
\label{sec:triangle maps}

We now verify that the maps in the proof of \fullref{lem:two handle maps} are graded isomorphisms. All these maps are induced by the standard holomorphic triangle counts associated to handleslides. For $\alpha$-handleslides, triangle counts induce chain maps of the form $\SFC(\Sigma,\bbeta, \balpha)\otimes \SFC(\Sigma,\balpha,\balpha') \to \SFC(\Sigma,\bbeta, \balpha')$, while triangle counts induce maps of the form $\SFC(\Sigma,\bbeta', \bbeta)\otimes \SFC(\Sigma,\bbeta,\balpha) \to \SFC(\Sigma,\bbeta', \balpha)$ for $\beta$-handleslides. We color $\balpha$-curves red and $\balpha'$-curves orange; we color $\bbeta$-curves blue and $\bbeta'$-curves purple.

We will decorate generators in each $(\Sigma,\bbeta,\balpha')$ or $(\Sigma,\bbeta',\balpha)$ with primes and leave generators in each $(\Sigma,\bbeta,\balpha)$ undecorated. We relabel after each computation, i.e. the generator $(c',z_3',y_0',\by')$ in $\HD_{4,4.5}$ is the generator $(c,z_3,y_0,\by)$ in $\HD_{4.5,5}$. We denote the highest degree element in each $(\Sigma,\balpha,\balpha')$ or $(\Sigma,\bbeta',\bbeta)$ by $\Theta$ and the lowest degree element by $\eta$.

We will make use of the following standard fact in our computations. Let $D$ and $D'$ be regions in a Heegaard diagram with a common edge in an oriented $\alpha$-curve $\alpha_0$, and let $\D$ be the domain of a Whitney polygon with holomorphic representative. If $\alpha_0$ abuts a basepoint region on the left and a basepoint region on the right, then the multiplicities of $\D$ on $D$ and $D'$ differ by at most 1. In particular, if the common edge is contained in $\partial\D$ and $D$ is a basepoint region, then $\D$ has multiplicity one on $D'$.

\begin{figure}
\labellist
	\begin{footnotesize}
	\pinlabel $w'$ at 250 126
  	\pinlabel $w$ at 226 123
	\pinlabel $y_0'$ at 230 86
  	\pinlabel $y_0$ at 190 86
 	\pinlabel $\Theta$ at 208 70
	\end{footnotesize}
 	\begin{large}
 		\pinlabel $D'$ at 190 103
 		\pinlabel $D$ at 210 90
 	\end{large}
\endlabellist
\includegraphics[width=1.0\textwidth]{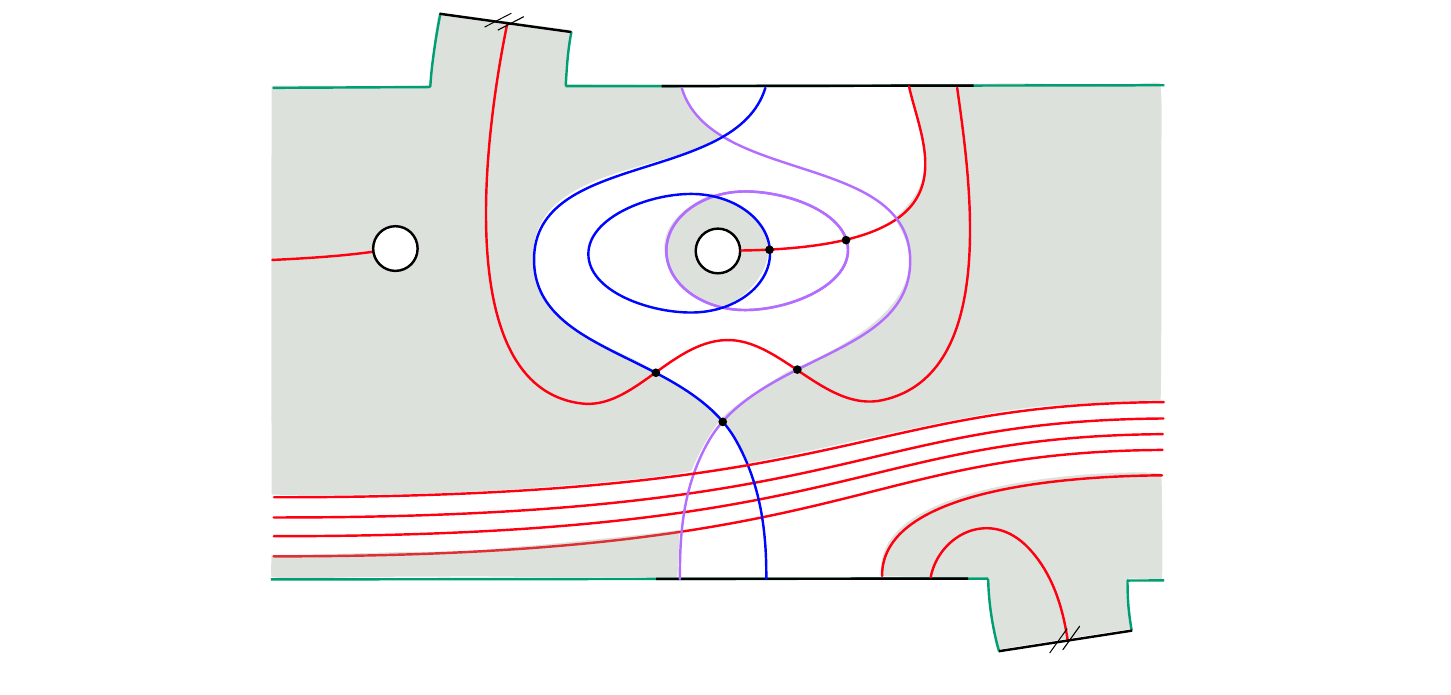}
  \caption{The diagram for $\Psi_{\HD_2,\HD_3}$.}
  \label{fig:nat 2 3}
\end{figure}

\begin{proposition}
\label{prop:nat 2 3}
$\Psi_{\HD_2,\HD_3}[(w,y_0,\by)] = [(w',y_0',\by')]$ .
\end{proposition}

\begin{proof} Consider \fullref{fig:nat 2 3}. Note that any triangle exiting $y_0$ must have multiplicity 1 on $D$ and multiplicity 0 on $D'$. There is a unique such triangle, and it sends $y_0$ to $y_0'$. We can erase the curves which contain $y_0$ or $y_0'$, and the new diagram is the diagram for a small Hamiltonian isotopy. The corresponding triangle map is chain homotopic to the identity, and thus sends $[(w,\by)]$ to $[(w',\by')]$. We then have $\Psi_{\HD_2,\HD_3}[(w,y_0,\by)] = [(w',y_0',\by')]$.
\end{proof}

\begin{proposition} 
\label{prop:nat 4 5}
$\Psi_{\HD_3,\HD_4}[(c,z_3,y_0,\by)] = [(z_3',y_0',\by')]$.
\end{proposition}

\begin{figure}
\labellist
	\begin{footnotesize}
  	\pinlabel $\Theta$ at 238 35
  	\pinlabel $z_3'$ at 285 75
 	\pinlabel $z_3$ at 275 24
 	\end{footnotesize}
 	\begin{large}
 	\pinlabel $D_1$ at 240 18
  	\pinlabel $D_2$ at 265 40
  	\pinlabel $D_3$ at 300 50
 	\end{large}
\endlabellist
\includegraphics[width=1.0\textwidth]{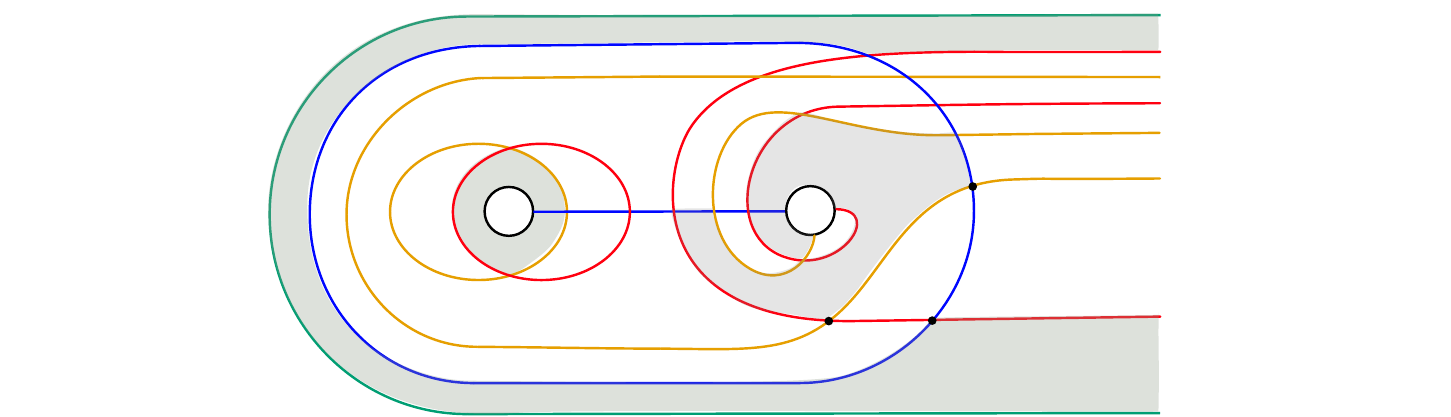}
  \caption{The diagram for $\Psi_{\HD_4,\HD_{4.5}}$.}
  \label{fig:nat 4 4.5}
\end{figure}

\begin{proof}

Let $\HD_{4.5}$ be the diagram obtained by handlesliding $\alpha_{g+2}'$ over $\alpha_{g+3}$ and consider the triple diagram in \fullref{fig:nat 4 4.5} for this move. Any triangle for $\Psi_{\HD_4,\HD_{4.5}}$ exiting $z_3$ must have must have multiplicity 0 on $D_1$ and $D_3$ and multiplicity 1 on $D_2$; this follows from the observation that the curve $\alpha_{g+2}'$ abuts a basepoint region on both the left and the right. (This cannot be seen in \fullref{fig:nat 4 4.5}, but it is true in the complete diagram.) There is a unique such triangle; it enters $z_3'$. Erasing curves yields a diagram for a small Hamiltonian isotopy, so that $\Psi_{\HD_4,\HD_{4.5}}[(c,z_3,y_0,\by)] = [(c',z_3',y_0',\by')]$.

\begin{figure}

\labellist
	\begin{footnotesize}
	\pinlabel $\Theta$ at 135 37
  	\pinlabel $z_3'$ at 180 61
 	\pinlabel $z_3$ at 175 25
 	
 	\pinlabel $c'$ at 299 66
	\pinlabel $c$ at 313 64
  	\pinlabel $\eta$ at 390 73
	\pinlabel $\Theta$ at 270 41
	\end{footnotesize}
	
	\begin{large}
  	\pinlabel $D_1$ at 147 22
  	\pinlabel $D_2$ at 163 40
  	\pinlabel $D_3$ at 190 40
  	
  	\pinlabel $D'$ at 317 40
	\pinlabel $D$ at 297 50
 	\end{large}
\endlabellist
\includegraphics[width=1.0\textwidth]{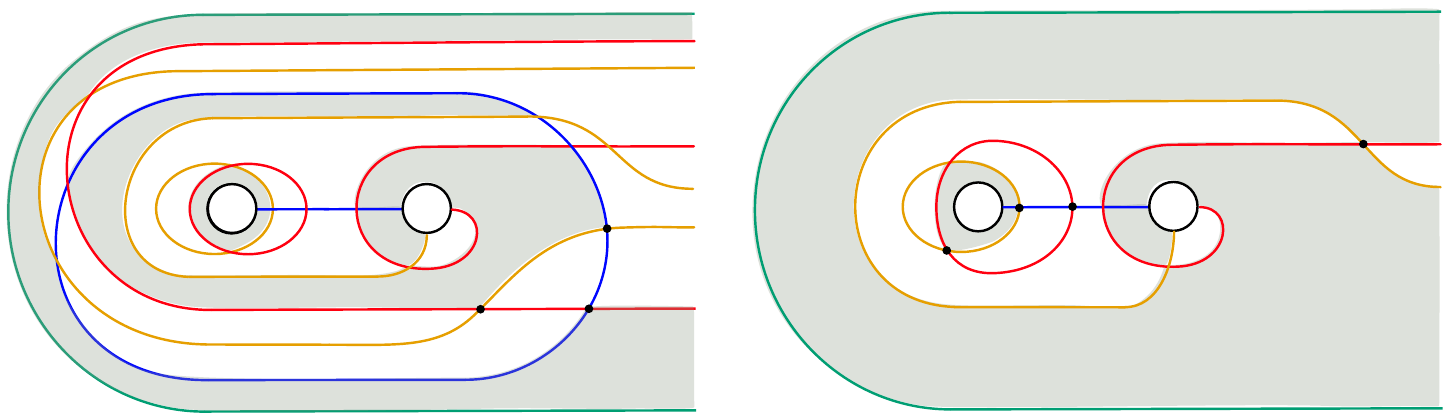}
  \caption{The diagrams for $\Psi_{\HD_{4.5},\HD_5}$.}
  \label{fig:nat 4.5 5}
 
\end{figure}

To compute $\Psi_{\HD_{4.5},\HD_5}$, consider the diagram on the left of \fullref{fig:nat 4.5 5}. Any triangle exiting $z_3$ must have multiplicity 0 on $D_1$ and $D_3$ and multiplicity 1 on $D_2$. There is a unique such triangle; it enters $z_3'$. Erasing curves yields the diagram on the right of \fullref{fig:nat 4.5 5}. Since no triangle in our count can enter $\eta$, all triangles must have multiplicity 0 on $D'$. There is a unique such triangle exiting $c$; its domain is the region $D$, and it enters $c'$. Erasing curves yields a diagram for a small Hamiltonian isotopy, and the corresponding triangle map sends $[(c,z_3,y_0,\by)]$ to  $[(c',z_3',y_0',\by')]$. We now obtain $\HD_5$ by performing a trivial destabilization. The associated map sends $[(c,z_3,y_0,\by)]$ to $[(z_3',y_0',\by')]$ so that $\Psi_{\HD_{4.5},\HD_5}[(c,z_3,y_0,\by)] = [(z_3',y_0',\by')]$.
\end{proof}

\begin{figure}
\labellist

	\begin{footnotesize}
	\pinlabel $x_0'$ at 122 93
  	\pinlabel $x_0$ at 80 85
	\pinlabel $z_2'$ at 52 46
	\pinlabel $z_2$ at 34 63
 	\pinlabel $\Theta$ at 19 80
 	\pinlabel $\Theta$ at 108 52

    \pinlabel $x_0'$ at 300 97
  	\pinlabel $x_0$ at 337 98
	\pinlabel $z_2'$ at 354 73
	\pinlabel $z_2$ at 331 64
 	\pinlabel $\Theta$ at 303 55
 	\pinlabel $\Theta$ at 318 48
 	\pinlabel $D_1$ at 317 98
	\pinlabel $D_2$ at 317 111
	\pinlabel $D_3$ at 344 113
	
	\end{footnotesize}
	
	\begin{large}
 	\pinlabel $D$ at 100 70
	\pinlabel $D'$ at 120 65

 	\end{large}
\endlabellist
\includegraphics[width=1.0\textwidth]{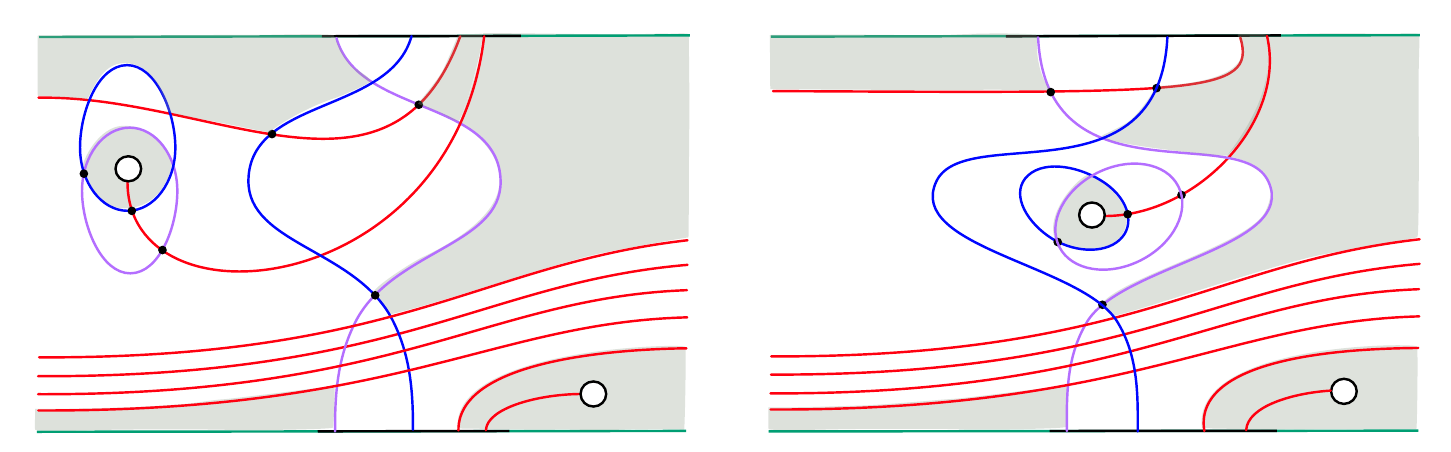}
  \caption{$\Psi_{\HD_{5},\HD_{5.5}}$ and $\Psi_{\HD_{5.5},\HD_{6}}$.}
  \label{fig:nat 5 6}
\end{figure}

\begin{proposition}
\label{prop:nat 5 6}
$\Psi_{\HD_5,\HD_6}[(z_2,x_0,\by)] = [(x_0',\by')]$.
\end{proposition}

\begin{proof} Let $\HD_{5.5}$ be the diagram obtained by performing the obvious finger move on $\HD_{5}$; diagrams for $\Psi_{\HD_{5},\HD_{5.5}}$ and $\Psi_{\HD_{5.5},\HD_{6}}$ are shown in \fullref{fig:nat 5 6}. On the left, any triangle exiting $x_0$ must have multiplicity one on $D$ and $D'$. There is a unique such triangle; it enters $x_0'$. On the right, any triangle exiting $x_0$ must have multiplicity one on $D_2$ and multiplicity zero on $D_1$ and $D_3$. There is a unique such triangle; it enters $x_0'$.

In each diagram in \fullref{fig:nat 5 6}, erasing curves yields a diagram for a small Hamiltonian isotopy. Each corresponding triangle map sends $[(z_2,x_0,\by)]$ to $[(z_2',x_0',\by')]$. We then obtain $\HD_6$ by performing a trivial destabilization; the corresponding map sends $[(z_2,x_0,\by)]$ to $[(x_0',\by')]$.
\end{proof}

Note that the diagrams used in \fullref{prop:nat 2 3}, \fullref{prop:nat 4 5}, and \fullref{prop:nat 5 6} are all admissible, since the admissibility argument at the end of \fullref{lem:two handle maps} holds for triply-periodic domains in triple diagrams. Each of the diagrams can be reduced to a diagram for a small Hamiltonian isotopy by erasing curves which abut suture regions on both sides. Since the diagram for a small Hamiltonian isotopy in an admissible diagram is itself admissible, so is each diagram in \fullref{prop:nat 2 3}, \fullref{prop:nat 4 5}, and \fullref{prop:nat 5 6}.

%%%%%%%%%%%%%%%%%%%%%%%%%%%%%%%%%%%%%%%%%%%%%%%%%%%%%%%
\section{Applications} % (fold)
\label{sec:applications}
%%%%%%%%%%%%%%%%%%%%%%%%%%%%%%%%%%%%%%%%%%%%%%%%%%%%%%%

Our first corollary is the extension of contact gluing maps to the bordered sutured category, \fullref{cor:bordered hkm} from \fullref{sec:intro}. This corollary is an essential ingredient in \cite{EVZ17}.

\begin{corollary}
\label{cor:bordered hkm duplicate}
Let $\M=(M,\gamma,\SF,\SZ)$ and $\M'=(M',\gamma',\SF,\SZ)$ be bordered sutured manifolds with $M\subset M'$ and $M'\setminus int(M)$ a sutured manifold equipped with a compatible contact structure $\xi$ and no isolated components. Let $\HD_M$ and $\HD_{M'}$ be admissible diagrams for $-\M$ and $-\M'$, respectively. Then there exists a map of type-D structures induced by $\xi$
\[
	\phi_{\xi}: \BSD(\HD_M)\to\BSD(\HD_{M'}),
\]
satisfying the following property. If $\SN = (N,\gamma_N,-\SF,-\SZ)$ is a bordered sutured manifold with diagram $\HD_N$ for $-\SN$, then the map
\[
\phi_{\xi}\boxtimes\id_{\BSA(\HD_{\SN})}: \BSA(\HD_\M)\boxtimes\BSD(\HD_\SN)\to\BSA(\HD_{\M'})\boxtimes\BSD(\HD_{\SN})
\]
induces the contact gluing map $\Phi_{\xi}:\SFH(-\M\cup_{-\SF}-\SN)\to \SFH(-\M'\cup_{-\SF}-\SN)$ up to graded homotopy equivalence.

Similar statements hold for the type-A and bimodule structures found in the bordered sutured theory. 
\end{corollary}

\begin{proof}
Bordered gluing operations and the bordered gluing map extend to the bordered sutured category. Using the notation of \fullref{sec:computing the pairing}, a bordered gluing map
\[
\Psi_{\SF}:\SFH(M_1,\Gamma_1)\otimes\SFH(M_2,\Gamma_2)\to \SFH((M_1,\Gamma_1)\cup_{\SF}(M_2,\Gamma_2))
\]
is induced by a map of the form
\[
\id_{U}\boxtimes\nabla_W\boxtimes\id_{V}:U\boxtimes W\otimes W^{\vee}\boxtimes V \to U\boxtimes A^{\vee}\boxtimes V.
\]

Zarev \cite{Zar11} proves that $\Psi_\SF$ is well-defined by showing that the algebraic join $\nabla_W: W\otimes W^{\vee}\to A^{\vee}$ is well-defined up to $\A_{\infty}$-homotopy, and using the fact that $(\boxtimes V,\boxtimes\id_V)$ is a $dg$-functor from the category of $\A_{\infty}$-modules to the category of chain complexes; see \cite{LOT15} for explanation.

Note that the bordered gluing operation $\cup_{\SF}$ as discussed in \fullref{sec:bordered gluing} extends to partially sutured and bordered sutured manifolds. The role of one or both $(M_i,\Gamma_i)$ can be replaced by $\M_i = (M_i,\gamma_i,\SF_i,\SZ)$ with $(\pm F,\gamma)\subset (\partial M_i)\setminus F_i$. As in \fullref{thm:main}, there is a sutured manifold $(M'',\Gamma'')$ such that $(M'',\Gamma'')\cup_{\SF'} \M \cong \M'$, where the dividing set on $\SF'$ is  disk-decomposible. Furthermore, we can choose $\SF'$ such that the dividing set has no closed components, so that we can choose a parametrization of $\SF'$ by an arc diagram $\SZ'$ so that the dividing set is elementary with respect to $\SZ'$. Let $A'$ be the associated bordered algebra. Let $\HD_{W'}$ be a diagram for the cap $\W'$ for the dividing set on $\SF'$; let $\HD_{U'}$ be a diagram for $(-M'',-\Gamma'')\setminus -\W'$; and let $\HD_{V'}$ be a diagram for $(-\M\setminus -\overline{\W'}$. Let $U', V'$, and $W'$ be the bordered modules $\BSD(\HD_{U'}),\BSDD(\HD_{V'})$, and $\BSA(-\HD_{W'})$ respectively. Note that $W'$ has a single generator $w$, since the dividing set on $\SF'$ is elementary by Proposition 15.1.2 in \cite{Zar11}. We can define a map

\[
\psi_{-\SF'} = \id_{U'}\boxtimes\nabla_{W'}\boxtimes\id_{V'}: U'\boxtimes W'\otimes (W')^{\vee}\boxtimes V'\to U'\boxtimes (A')^{\vee}\boxtimes V'.
\]
which is well-defined up to graded homotopy equivalence of Type-D structures, since by Lemma 2.3.13 of \cite{LOT15}, $\boldsymbol{\cdot}\boxtimes \id_{V'}$ is an $\A_{\infty}$-functor.

Similar to \fullref{thm:main}, note that there is a compatible contact structure $\xi''$ on $M''$ such that $\xi$ is contactomorphic to $\xi''$ attached to an $I$-invariant contact structure. Since $U'\boxtimes W'$ is homotopy equivalent to $\SFC(-M'',-\Gamma'')$, we can define $\phi_{\xi}:(W')^{\vee}\boxtimes V\to U'\boxtimes (A')^{\vee}\boxtimes V'$ by 
\[
\phi_{\xi}(w^{\vee}\boxtimes \by) = \psi_{-\SF'}(\bx\boxtimes w,w^{\vee}\boxtimes\by),
\]
where $\bx\boxtimes w$ represents $\EH(\xi'')$.

Denote $\BSD(\HD_N)$ by $X$. Since $(W')^{\vee}\boxtimes V'\boxtimes X \simeq \SFC(\HD_M\cup_{\SF}\HD_N)$ and $U'\boxtimes (A')^{\vee}\boxtimes V'\boxtimes X\simeq \SFC(\HD_{M'} \cup_{\SF}\HD_N)$, the bordered contact gluing map $\Psi_{\xi''}$ is induced by evaluating
\[
\id_{U'}\boxtimes\nabla_{W'}\boxtimes\id_{V'\boxtimes X}
\]
on 
$\bx\boxtimes w$, while $\phi_{\xi}\boxtimes\id_X$ is the evaluation of
\[
\id_{U'}\boxtimes\nabla_{W'}\boxtimes\id_{V'}\boxtimes\id_X
\]
on $\bx \boxtimes w$. It is clear that these maps are equal up to homotopy equivalence since $\id_{V'}\boxtimes\id_{X}$ is homotopic to $\id_{V'\boxtimes X}$. (This is from Lemma 2.3.13 in \cite{LOT15}.)

Note that by \fullref{thm:main}, $\Psi_{\xi''}$ agrees with the contact gluing map $\Phi_\xi$ up to graded isomorphism. Also, since $\psi_{-\SF'}$ is defined up to graded homotopy equivalence, so is $\phi_{\xi}$, so we can define $\phi_{\xi}:\BSD(\HD_M)\to\BSD(\HD_{M'})$ by using homotopy equivalences $W'\boxtimes V' \simeq \BSD(\HD_M)$ and $U'\boxtimes(A')^{\vee}\boxtimes V'\simeq \BSD(\HD_{M'})$.

This completes the proof for type-$D$ structures. Similar arguments hold for type-$A$ modules and bimodules.
\end{proof}

We also obtain an independent proof of Juh\'asz and Zemke's result that diagrammatic maps agree with the corresponding contact gluing maps; see \cite{JuZe20}. This is \fullref{cor:simple maps} in \fullref{sec:intro}.

\begin{corollary}
\label{cor:simple maps duplicate}
Given a diagram $\HD$ for $(-M,-\Gamma)$, there is a diagram for $(-M_i,-\Gamma_i)$ such that the $\HKM$ map $\Phi_{i}:\SFH(-M,-\Gamma)\to\SFH(-M_i,-\Gamma_i)$ is induced by the diagrammatic map $\sigma_i$ up to graded isomorphism. Furthermore, given any other diagrammatic attachment for $(-M_i,-\Gamma_i)$, the associated diagrammatic map $\sigma_i'$ also induces $\Phi_i$ up to graded isomorphism.
\end{corollary}

\begin{proof}
The first statement is immediate from \fullref{lem:padded handles decomposition}. By \fullref{lem:one handle maps} and \fullref{lem:two handle maps}, both $(\sigma_i)_*$ and $(\sigma_i')_*$ are equal to bordered contact gluing maps $\Psi_i$ and $\Psi_i'$ respectively, up to graded isomorphism. Since the map $\Psi_{-\SF}$ is well-defined, $\Psi_i$ and $\Psi_i'$ are equal up to graded isomorphism. Then $(\sigma_i)_*$ and $(\sigma_i')_*$ are also equal up to graded isomorphism.
\end{proof}

As a final application, we use \fullref{cor:simple maps duplicate} to show that the HKM map can be computed using nice diagrams by  extending Plamenevskaya's application of the Sarkar--Wang algorithm in \cite{Pla07}. This is \fullref{cor:combinatorial hkm} in \fullref{sec:intro}.

\begin{corollary}
\label{cor:combinatorial hkm duplicate}
Suppose that $(-M,-\Gamma)\subset(-M',-\Gamma')$ is a proper inclusion of sutured manifolds with no isolated components and contact gluing map $\Phi_{\xi}$. Given an admissible diagram $\HD_0$ for $(-M,-\Gamma)$, there is a diagram $\HD_2$ for $(-M',-\Gamma')$, a nice diagram $\HD^{nice}$ for $(-M,-\Gamma)$, a nice diagram $\HD_2^{nice}$ for $(-M',-\Gamma)'$, a map $\phi_{\xi}^{nice}:\SFC(\HD^{nice})\to\SFC(\HD_2^{nice})$ of the form $\by\to(\by,\bx_0)$, and graded isomorphisms $f,g$ such that the following diagram commutes:

\begin{center}
\begin{tikzcd}
\SFH(\HD_0) \arrow[r, "f" ] \arrow[d, "\Phi_{\xi}"]
&\SFH(\HD^{nice})\arrow[d, " \phi_\xi^{nice} "]\\
\SFH(\HD_2) \arrow[r, "g" ]
& \SFH(\HD_2^{nice})
\end{tikzcd}
\end{center}
\end{corollary}

Some discussion is in order before we begin the proof. By \fullref{cor:simple maps}, the $\HKM$ map can be factored into diagrammatic maps $\Phi_{\xi} = (\sigma_{i_n}\circ\ldots \sigma_{i_1})_*$ up to graded isomorphism. Let $\HD$ be a diagram for $(-M,-\Gamma)$ and $\HD'$ a diagram for $(-M',-\Gamma')$ obtained by diagrammatic handle attachments. The composition $\sigma_{i_n}\circ\ldots\circ \sigma_{i_1}$ takes the form $\phi_{\xi}(\by) = (\by,\bx_0)$, where $\bx_0$ is the collection of preferred intersections of the 2-handles. We wish to show that this situation can be realized by nice diagrams and that the preferred intersections in the nice diagrams correspond to $\bx_0$ up to graded homotopy equivalence.

Note that we cannot just perform diagrammatic handle attachments on a given nice diagram for $(M,\Gamma)$, since a diagrammatic contact 2-handle attachment on a nice diagram need not result in a nice diagram. We will instead find a nice diagram for $(M',\Gamma')$ and work backwards to obtain a suitable diagram for $(M,\Gamma)$.

As in \cite{Pla07}, the difficulty is showing that the Sarkar--Wang algorithm may be applied in a way that preserves $\bx_0$. More precisely, we wish to show there is a diagram which is commutative up to homotopy,

\begin{center}
\begin{tikzcd}
\SFC(\HD) \arrow[r, "\psi" ] \arrow[d, "\phi_{\xi}"]
&\SFC(\HD^{nice})\arrow[d, " {\bz \to (\bz,\bx_0^{nice})} "]\\
\SFC(\HD') \arrow[r, "\psi'" ]
& \SFC((\HD^{nice})')
\end{tikzcd}
\end{center}
where $\bx_0^{nice}$ is the preferred intersection in $(\HD^{nice})'$ and $\psi,\psi'$ are triangle maps induced by Heegaard moves in the Sarkar--Wang algorithm. In \cite{Pla07}, they key idea is to show that one can apply Sarkar--Wang without performing a finger move around a full $\beta$-curve. This implies that no handleslide maps arise, and the preferred intersections are preserved by the maps which arise in the course of the algorithm.

In our setting,  we cannot always rule out finger moves around full $\beta$-curves in general, but we can rule out finger moves around $\beta$-curves coming from contact 2-handle attachments. Recall that attaching a contact 2-handle adds a strip to the Heegaard surface. If a finger move which enters this strip does not cross the corresponding $\alpha$-curve, then it can be isotoped outside the strip. If it crosses the $\alpha$-curve, it must terminate there, since the $\alpha$-curve borders a basepoint region on one side. In this case, we can also isotope the finger move outside the strip; see \fullref{fig:plam2}. This means that we can apply Sarkar--Wang without handlesliding over any of the new $\beta$-curves.

\begin{figure}
\includegraphics[width=1.0\textwidth]{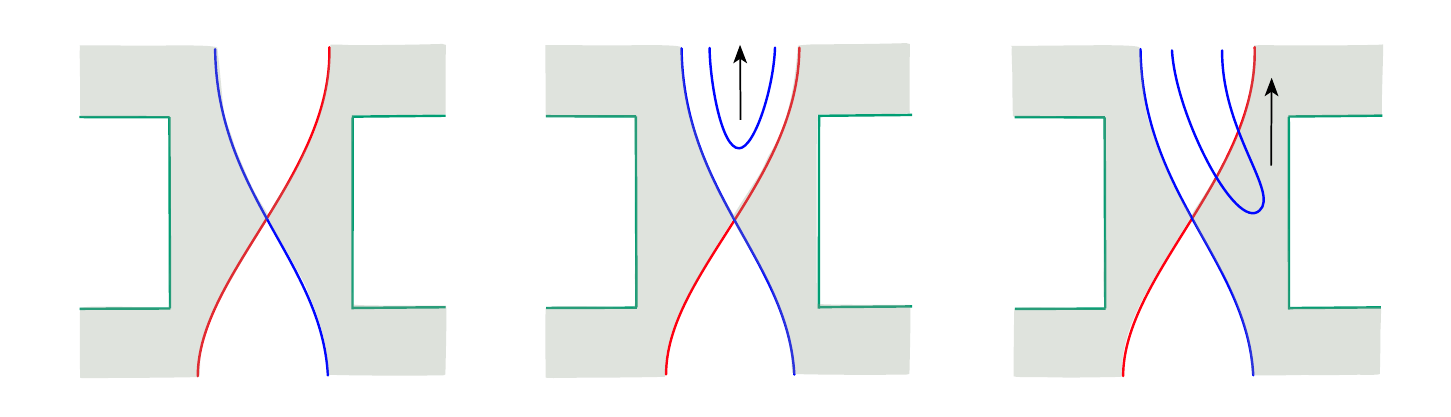}
  
  \caption{Finger moves cannot pass through a strip coming from a 2-handle attachment.}
	\label{fig:plam2}
\end{figure}

This is sufficient for our purposes, since if we perform a handleslide over a $\beta$-curve coming from $\HD$ or a finger move, then the diagram for the triangle count in the strip looks like \fullref{fig:plam1}. It is easy to see that these maps all preserve the preferred intersections.

\begin{figure}
\labellist
	\begin{footnotesize}
  	\pinlabel $x_0'$ at 213 73
  	\pinlabel $x_0$ at 206 47
  	\end{footnotesize}
  	\pinlabel $\Theta$ at 222 24
\endlabellist
\includegraphics[width=1.0\textwidth]{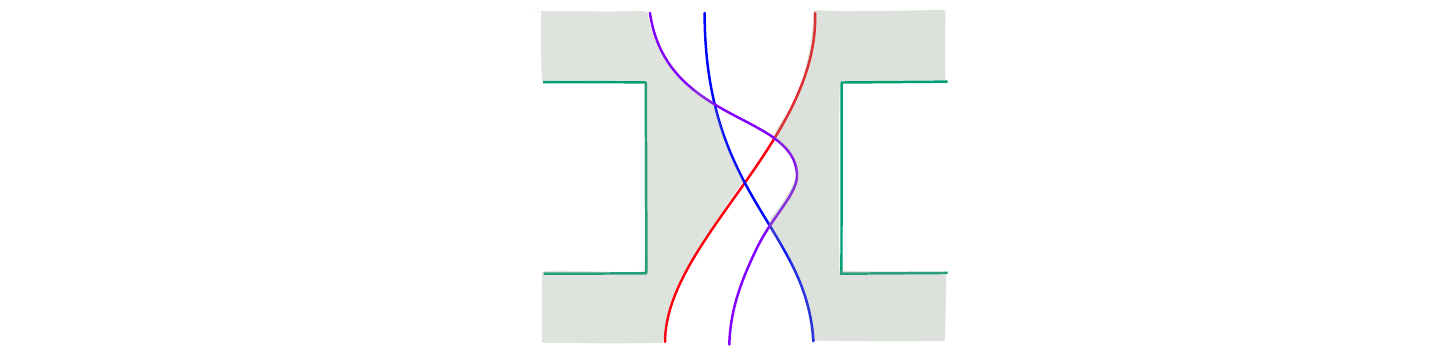}
  
  \caption{A strip coming from a 2-handle attachment in a triple diagram.}
	\label{fig:plam1}
\end{figure}

\begin{proof}
Let $\HD_0 = (\Sigma, \bbeta,\balpha)$ be a diagram for $(-M,-\Gamma)$ and apply the Sarkar--Wang algorithm to obtain a nice diagram $\HD_0^{nice}$. Decompose $(-M',-\Gamma')$ as $(-M,-\Gamma)\cup_j -h^1_j\cup_k -h^2_k$. We attach diagrammatic 1-handles to form a diagram $\HD_1^{nice}$ for $(-M,-\Gamma)\cup_j -h^1_j$ which is nice, since we have only modified the sutured region.

Now, attach diagrammatic 2-handles to $\HD_1^{nice}$ to obtain a diagram $\HD_2 = (\Sigma',\bbeta\cup\bbeta_0,\balpha\cup\balpha_0)$ for $(-M',-\Gamma')$. Apply Sarkar--Wang again to get a nice diagram $\HD_2^{nice} = (\Sigma',\bbeta'\cup\bbeta_0',\balpha\cup\balpha_0)$. As discussed above, we can do this so that all finger moves are performed in the complement of the strips for the contact handles, $\Sigma'\setminus \INT(\Sigma)$. Thus, the triangle map $\psi_{\HD_2,\HD_2^{nice}}$ sends generators of the form $(\by,\bx_0)$ to generators of the form $(\bz,\bx_0^{nice})$.

Now, let $\HD^{nice} = (\Sigma,\bbeta',\balpha)$ be the Heegaard diagram for $(-M,-\Gamma)$ obtained by removing all the diagrammatic handles from $\HD_2^{nice}$. Note that it is nice, since it differs from $\HD_0^{nice}$ by some finger moves which appear in the Sarkar-Wang algorithm. Consider the triple diagram $(\Sigma', \bbeta\cup\bbeta_0,\bbeta'\cup\bbeta_0',\balpha\cup\balpha_0)$ for $\Psi_{\HD_2,\HD_2^{nice}}$; let $\bx_0$ and $\bx_0^{nice}$ be the preferred intersections in $\bbeta_0\cap\balpha_0$ and $\bbeta_0'\cap\balpha_0$ respectively. Since we only care about the image of the contact gluing map, we can restrict the triangle count for $\Psi_{\HD_2,\HD_2^{nice}}$ to the subspace of $\SFC(\HD_2)$ generated by tuples of the form $(\by,\bx_0)$. This restricted count is the same as the full triangle count for $\Psi_{\HD_0,\HD^{nice}}$ arising from $(\Sigma, \bbeta,\bbeta',\balpha)$, so that for a cycle $\by$ in $\SFC(\HD_0)$, we have $\psi_{\HD_2,\HD_2^{nice}}(\by,\bx_0) = (\psi_{\HD_0,\HD^{nice}}(\by),\bx_0^{nice})$. We can define a map $\phi_{\xi}^{nice}:\SFC(\HD^{nice})\to \SFC(\HD_2^{nice})$ by
\[
\phi_{\xi}^{nice}(\psi_{\HD_0,\HD^{nice}}(\by)) = (\psi_{\HD_0,\HD^{nice}}(\by),\bx_0^{nice}),
\]
so that for any cycle $\bz\in\SFC(\HD^{nice})$ we have
\[
(\phi_{\xi}^{nice})_*[\bz] = [(\bz,\bx_0^{nice})].
\]
By construction, $(\phi_{\xi}^{nice})_*$ has the desired form and is equal to the $\HKM$ map $\Phi_{\xi}$ up to graded isomorphism.
\end{proof}

%%%%%%%%%%%%%%%%%%%%%%%%%%%%%%%%%%%%%%%%%%%%%%%%%%%%%%%
\bibliographystyle{amsalpha}      % amsalpha & amsplain
\bibliography{Bibliography}
%%%%%%%%%%%%%%%%%%%%%%%%%%%%%%%%%%%%%%%%%%%%%%%%%%%%%%%

\end{document}